\documentclass[11pt,reqno]{amsart}
\usepackage{graphicx}
\usepackage{color}
\usepackage{framed}
\usepackage{hyperref}
\usepackage{amsmath,amssymb,amsthm,amsfonts,
mathrsfs,stmaryrd}
\usepackage{cases,indentfirst}
\usepackage{multicol}
\makeatletter

\newcommand{\Rmnum}[1]{\expandafter\@slowromancap\romannumeral #1@}

\newcommand{\bu}{u}
\def\p{p}
\def\R{\mathbb{R}}

\makeatother
\newtheorem{theorem}{Theorem}[section]

\newtheorem{lemma}{Lemma}[section]

\newtheorem{proposition}{Proposition}[section]
\newtheorem{remark}{Remark}[section]
\usepackage[numbers,sort&compress]{natbib}
\theoremstyle{definition}
\numberwithin{equation}{section}

\allowdisplaybreaks[1]
\begin{document}
\author[G. Hong]{Guangyi Hong}\address{Guangyi Hong\newline\indent Department of Applied Mathematics\newline\indent Hong Kong Polytechnic University\newline\indent Hung Hom,
Kowloon, Hong Kong, P. R. China}
\email{gyhmath05@outlook.com}
\author[H. Peng]{Hongyun Peng}\address{Hongyun Peng\newline\indent School of Applied Mathematics\newline\indent Guangdong University of Technology\newline\indent Guangzhou, 510006, P. R. China}
\email{penghy010@163.com}

\author[Z.A. Wang]{Zhi-An Wang\textsuperscript{$ \ast $}}\address{Zhi-an Wang\newline\indent Department of Applied Mathematics\newline\indent Hong Kong Polytechnic University\newline\indent Hung Hom,
Kowloon, Hong Kong,  P. R. China}\email{mawza@polyu.edu.hk}

\author[C. Zhu]{Changjiang Zhu}\address{Changjiang Zhu\newline\indent School of Mathematics\newline\indent South China University of
Technology\newline\indent Guangzhou, 510641, P. R. China}
\email{machjzhu@scut.edu.cn}

\title[Phase transition steady state  to a hyperbolic-parabolic system]{\bf Nonlinear stability of phase transition steady states to a hyperbolic-parabolic system modelling vascular networks}
\begin{abstract}
This paper is concerned with the existence and stability of phase transition steady states to a quasi-linear hyperbolic-parabolic system of chemotactic aggregation, which was proposed in \cite{ambrosi2005review, gamba2003percolation} to describe the coherent vascular network formation observed {\it in vitro} experiment. Considering the system in the half line $ \mathbb{R}_{+}=(0,\infty)$ with Dirichlet boundary conditions, we first prove the existence \textcolor{black}{and uniqueness of non-constant phase transition steady states} under some structure conditions on the pressure function.  Then we prove that this unique phase transition steady state is nonlinearly asymptotically stable against a small perturbation. We prove our results by the method of energy estimates, the technique of {\it a priori} assumption and a weighted Hardy-type inequality.

\vspace*{4mm}

\vspace*{1mm}

\end{abstract}
\maketitle

 \section{Introduction} \label{sec:introduction}
 Experiments of {\it in vitro} blood vessel formation demonstrate that endothelial cells randomly dispersing on a gel substrate (matrix) can spontaneously organize into a coherent vascular network (see \cite{ambrosi2005review, gamba2003percolation} and figures therein), which is called angiogenesis - a major factor driving the tumor growth.  How endothelial cells self-organize geometrically into capillary networks and how separate individual cells cooperate in the formation of coherent patterns remain poorly understood biologically up to date. These networking patterns can not be explained by the macroscopic aggregation models such as Keller-Segel type chemotaxis models that lead to point-wise blowup or rounded aggregates, nor by the microscopic kinetic models that describe individual cell behaviors, as commented in \cite{chavanis2006jeans}. Strikingly they can be numerically reproduced by a hydrodynamic (hyperbolic-parabolic) models of chemotaxis proposed in \cite{ambrosi2005review, gamba2003percolation} as follows
\begin{eqnarray}\label{hyper}
\left\{\begin{array}{l}
\partial_t \rho+\nabla \cdot (\rho \bu)=0, \\[1mm]
\partial_t (\rho \bu)+\nabla \cdot (\rho \bu \otimes \bu)+\nabla \p(\rho)=\mu \rho \nabla \phi -\alpha \rho \bu,\\[1mm]
\partial_t \phi-\Delta \phi =a\rho-b\phi
 \end{array} \right.
\end{eqnarray}
where $\rho$ denotes the endothelial cell density, $\bu$ the cell velocity and $\phi$ the concentration of chemoattractant secreted by cells; $\p$ is a pressure function accounting for the fact that closely packed cells resist to compression due to the impenetrability of cellular matter, the parameter $\mu>0$ measures the intensity of cell response to the chemoattractant and $\alpha\rho \bu$ corresponds to a damping (friction) force due to the interaction between cells and underlying substratum with a drag coefficient $\alpha>0$; $a$ and $b$ are positive constants accounting for the growth and death rates of the chemoattractant. The convective term $\nabla \cdot (\rho \bu \otimes \bu)$ models the cell persistence (inertia effect). We refer more detailed biological interpretations on the model \eqref{hyper} to \cite{ambrosi2005review, gamba2003percolation}. The hydrodynamic system \eqref{hyper} has been (formally) derived from the following kinetic equation in \cite{chavanis2007kinetic} via the mean-field approximation
 \begin{eqnarray*}
\frac{\partial f}{\partial t} +v\cdot \nabla_x f= \int_{V} (T[\phi](v', v)f(t,x,v')-T[\phi](v,v')f(t,x,v))\mathrm{d} v',
\end{eqnarray*}
where $f(t,x,v)$ denotes cell density distribution at time $t$, position $x\in {\mathbb R}^d$ moving with a velocity $v$ from a compact set $V$ in $\mathbb{R}^d$, and tumbling kernel $T[\phi](v',v)$ describes the frequency of changing trajectories from velocity $v'$ (anterior) to $v$ (posterior) in response to a chemical concentration $\phi$.

At first glance from mathematical point of view, the above hydrodynamic system \eqref{hyper} is analogous to the well-known damped Euler-Poisson system where the $\phi$-equation is elliptic (i.e. $-\Delta \phi=\alpha \rho$) which appears in various important applications depending on the sign of $\mu$, such as the propagation of electrons in semiconductor devices (cf. \cite{markowich2012semiconductor}) and the transport of ions in plasma physics (cf. \cite{choudhuri1998physics}) and the collapse of gaseous stars due to self-gravitation \cite{chandrasekhar1957introduction}. However the parabolic $\phi$-equation in \eqref{hyper} will bring substantial differences in mathematical analysis and many existing mathematical frameworks developed for the Euler-Poisson system are inapplicable directly to \eqref{hyper}. Due to the competing interactions between parabolic dissipation and hyperbolic anti-dissipation effect plus nonlinear convection, the global well-posedness and regularity of solutions to \eqref{hyper} is very complicated as can be glimpsed from the Euler-Poisson equations for which the understanding of solution behaviors is rather incomplete despite of numerous studies attempted. Up to date, there are only few results obtained for \eqref{hyper} in the literature. First when the initial value $(\rho_0, u_0, \phi_0)$ is a small perturbation of a constant ground state $(\bar{\rho}, 0, \bar{\phi})$ in $H^s(\R^d) (s>d/2+1)$ with $\bar{\rho}>0$ sufficiently small, the global existence and stability of solutions with non-vacuum (i.e. $\inf\limits_{x\in \R^d} \rho>0$) to \eqref{hyper} was established in \cite{russo2013existence,russo-2012}. The linear stability of constant ground state $(\bar{\rho}, 0, \bar{\phi})$ was obtained under the condition $p'(\bar{\rho})>\frac{a \mu}{b} \bar{\rho}$ in \cite{kowalczyk2004stability} where an additional viscous term $\nabla ^{2}u $ is supplied to the second equation of \eqref{hyper}. The stationary solutions of \eqref{hyper} with vacuum (bump solutions) in a bounded interval with zero-flux boundary condition and in $\R$ were constructed in \cite{berthelin2016stationary} and further elaborated in \cite{carrillo2019phase}. The model \eqref{hyper} with $p(\rho)=\rho$ and periodic boundary conditions in one dimension was numerically explored in \cite{filbet2005approximation}. These appear to be the only results available to the system \eqref{hyper} in the literature and further studies are in demand. For results on some other types of \textcolor{black}{hyperbolic-parabolic} chemotaxis models, we refer to \cite{DiRusso,li2011hyperbolic,li2010nonlinear, mei2015asymptotic, zhang2007global} and references therein.

Note that the above-mentioned mathematical works on \eqref{hyper} prescribe initial data as a small perturbation of constant equilibria and the large-time profile of solutions is also constant, which can not explain the experimental observations of \cite{ambrosi2005review, gamba2003percolation} showing prominent phase transition patterns connecting regions clear (or low density) of cells. This motivates us to explore the possible non-constant phase transition profiles of solutions. The aim of this paper is to study the existence and stability of phase transition steady states without vacuum to the system \eqref{hyper} in one dimensional half space $\R_+=(0,\infty)$. For the convenience of presentation in the sequel, we set $m=\rho u$, namely $m$ denotes the momentum of cells, and recast the one-dimensional system \eqref{hyper} in $\R_+$ as
\begin{subnumcases}
\displaystyle \rho_{t}+m_{x}=0,  \ \ \ \ x \in \mathbb{R}_{+},\ \ t>0, \label{ori-eq-a}\\[1mm]
 m_{t}+\left(\frac{m^{2}}{\rho}+p(\rho)\right)_{x}=\mu \rho \phi_{x}-\alpha m, \label{ori-eq-b}\\[1mm]
\phi_{t}=\phi_{x x}+a \rho-b \phi.\label{ori-eq-c}
\end{subnumcases}
As in \cite{kowalczyk2004stability}, we assume the pressure function $p$ depends only on the density and satisfies
\begin{gather}\label{p-condition}
\displaystyle p(\rho) \in C ^{3}(0,\infty),\ \  p'(\rho)- \frac{a \mu}{b}\rho>0\ \ \mbox{for any }\rho>0.
\end{gather}
A typical form of $p$ is $ p(\rho) = \frac{K}{2}\rho ^{2}$ with $ K > \frac{a \mu}{b}$. We supplement the system \eqref{ori-eq-a}--\eqref{ori-eq-c} with the following boundary conditions
\begin{gather}\label{initial-ori}
\displaystyle m(0,t)=0,\ \ \
  \phi(0,t)=\phi _{-},
\end{gather}
and the initial value
\begin{gather}\label{far-field}
\displaystyle (\rho,m, \phi)(x,0)=(\rho _{0},m_{0},\phi _{0})(x) \rightarrow (\rho _{+},0,\phi _{+})\ \ \mbox{as}\ \  x \rightarrow +\infty
\end{gather}
\textcolor{black}{where $ \phi _{-} $, $ \rho _{+}>0 $, $ \phi _{+}>0 $ are constants and $\phi _{-}\ne \phi _{+}$.}

In this paper, we shall first use delicate analysis to show that the system \eqref{ori-eq-a}--\eqref{ori-eq-c} has a unique non-constant stationary solution $(\bar{\rho},0,\bar{\phi})$ without vacuum satisfying
\begin{gather*}
\displaystyle \bar{\phi}(0)=\phi _{-} \ \ \mbox{and}\ \ \lim _{x \rightarrow +\infty}\left( \bar{\rho}, \bar{\phi} \right)(x)=(\rho _{+},\phi _{+})
\end{gather*}
where $\bar{\rho}$ and $\bar{\phi}$ are monotone.
Then we show that the stationary solution $(\bar{\rho},0,\bar{\phi})$ is asymptotically stable, namely the solution of \eqref{ori-eq-a}--\eqref{ori-eq-c} converges to $(\bar{\rho},0,\bar{\phi})$ point-wisely as $t \to \infty$ if the initial value  $(\rho _{0},m_{0},\phi _{0})$ is an appropriate small perturbation of $(\bar{\rho},0,\bar{\phi})$. The monotonicity of $\bar{\rho}$ and $\bar{\phi}$ indicates that the steady states have phase transition profiles. To the best of our knowledge, there are not such results available in the literature \textcolor{black}{for \eqref{ori-eq-a}--\eqref{ori-eq-c}} and even for the Euler-Poisson equations (cf. \cite{yang-nishihara-1999-JDE, meiming-2005-JMFM,meiming-2010-JDE,meiming-2010-IJAM,hsiao-liutaiping}. To prove our results, we fully capture the structure of \eqref{ori-eq-a}--\eqref{ori-eq-c} where the stationary solution has exponential decay at far field, which is not enjoyed by the Euler-Poisson (or Euler) equations. Though part of the proof \textcolor{black}{of} our results is inspired by some ideas of \cite{hsiao-liutaiping,nishihara-1996,yang-nishihara-1999-JDE,huang-mei-wang-yu} on Euler-Poisson or Euler equations, we have added lof of extra efforts to deal with complicated couplings and boundary effects. The coupling term $\rho \phi _{x} $ leads to two linear terms in the linearized system around stationary solutions and how to make these linear terms under control is key to the asymptotic stability against small perturbations. We resolve this issue by the structural assumption \eqref{p-condition} to take up the dissipation and use the exponentially weighted Hardy inequality in half space to compensate for the lack of dissipation in the hyperbolic equations. Due to the couplings and boundary effects, the energy estimates are very sophisticated, where the lower-order estimates involve higher-order estimates and vice versa. We use the delicate energy estimates along with the technique of {\it a priori} assumptions to unravel these tangles and gradually achieve our results.

The rest of this paper is organized as follows. In section \ref{sec-mainresult}, we state our main results on the existence of non-constant stationary solutions (Theorem \ref{thm-sta-problem}) and stability of stationary solutions (Theorem \ref{thm-stability}). In section \ref{sec-stationary-solution}, we study the stationary problem and prove Theorem \ref{thm-sta-problem}. The proof of Theorem \ref{thm-stability} is given in Section \ref{sec:stability}.

\section{Statement of main results}\label{sec-mainresult}
\label{sec:section_name}
In this section, we shall state the main results of this paper. To be precise, we first introduce some notations used. Throughout the paper, we use $ \|\cdot\|_{L ^{\infty}} $, $ \|\cdot\| $ and $ \|\cdot\|_{k} $ to denote the norms of usual $ L ^{\infty}(\mathbb{R}_{+}) $, $ L ^{2}(\mathbb{R}_{+}) $ and the standard Sobolev space $ H ^{k}(\mathbb{R}_{+}) $, respectively. We also use $ \|(f _{1},\cdots,f _{n})\|~(\mathrm{resp}.~\|(f _{1},\cdots,f _{n})\|_{k}) $ to denote $ \|f _{1}\|+\cdots+\|f _{n}\|~(\mathrm{resp}.~\|f _{1}\|_{k}+\cdots+\|f _{n}\|_{k}) $ for some $ n \in \mathbb{Z}_{+}$. We denote by $ C $ a generic constant that may vary in the context, and by $ C _{\eta} $ a constant  depending on $ \eta $. Occasionally, we simply write $ f \sim g $ if $ C ^{-1}\leq f \leq Cg $ for some constant $ C>0 $.

\textcolor{black}{
It can be verified that the system \eqref{ori-eq-a}--\eqref{ori-eq-c} possesses the following energy functional (cf. \cite{berthelin2016stationary, chavanis2007kinetic})
$$F[\rho, u,\phi]=\frac{1}{2\mu }\int_{\mathbb{R}_{+}}\rho u^2 \mathrm{d} x+\frac{1}{\mu }\int_{\mathbb{R}_{+}}G(\rho) \mathrm{d} x+\frac{1}{2a}\int_{\mathbb{R}_{+}} (|\nabla \phi|^2+b \phi^2) \mathrm{d} x-\int_{\mathbb{R}_{+}} \rho \phi \mathrm{d} x$$
which, subject to the boundary condition \eqref{initial-ori}, satisfies that
$$\frac{\mathrm{d}}{\mathrm{d}t} F[\rho, u,\phi]+\frac{\alpha}{\mu}\int_{\mathbb{R}_{+}}\rho u^2 \mathrm{d} x+\frac{1}{a}\int_{\mathbb{R}_{+}}|\phi_t|^2 \mathrm{d}x=0,$$ where $\rho G''(\rho)=p'(\rho)$. Thus the stationary solution satisfying $\frac{\mathrm{d}}{\mathrm{d}t} F[\rho, u,\phi]=0$ gives rise to $\rho u=0$ and $\phi_t=0$ in $\mathbb{R}_{+}$. Since we are interested in non-constant profile for $\rho$, $u\equiv 0$ is the only (physical) stationary profile for the velocity $u$. Therefore stationary solutions of \eqref{ori-eq-a}--\eqref{ori-eq-c} without vacuum must possess the form $(\bar{\rho},0,\bar{\phi})$, where $(\bar{\rho},\bar{\phi})$ satisfies}
\begin{subnumcases}
\displaystyle p(\bar{\rho}) _{x}=\mu \bar{\rho}\bar{\phi}_{x},\ \ x \in \mathbb{R}_{+}, \label{stat-1}\\
\displaystyle \bar{\phi} _{xx}+a \bar{\rho}-b \bar{\phi}=0,\label{stat-2}\\
\displaystyle \bar{\phi}(0)=\phi _{-},\label{stat-c}\\
\displaystyle \lim _{x \rightarrow +\infty}\left( \bar{\rho},\bar{\phi} \right)=(\rho_{+},\phi _{+}).  \label{stat-3}
\end{subnumcases}
Here the pressure $p$ satisfies \eqref{p-condition} and the constants $\rho_{+} $ and $\phi_{\pm} $ are the same as in \eqref{initial-ori} and \eqref{far-field}.

Then our first result concerning the existence and uniqueness of solutions to the stationary problem \eqref{stat-1}--\eqref{stat-3} is given below.
\begin{theorem}\label{thm-sta-problem}
Let $ \rho _{+}>0 $ and $ \phi _{-}\neq \phi _{+} $ such that $ a \rho_+=b \phi _{+} $. If
\begin{gather}
     \displaystyle \phi _{-}-\phi _{+}+\int _{0}^{\rho _{+}}\frac{p'(s)}{\mu s}\mathrm{d} s>0,\label{fip-condi}
\end{gather}
then there is a unique constant $\rho_->0$ such that the problem \eqref{stat-1}--\eqref{stat-3} with \eqref{p-condition} admits a unique solution $ (\bar{\rho},\bar{\phi}) $ satisfying $\bar{\rho}(0)=\rho_-$ and
\begin{gather}\label{monotonic}
\displaystyle \begin{cases}
	\displaystyle \bar{\rho}'(x)<0,\, \bar{\phi}'(x)<0 &\text{ if }\phi _{-}>\phi _{+},\\[1mm]
	\displaystyle\bar{\rho}'(x)>0,\, \bar{\phi}'(x)>0 &\text{ if }\phi _{-}<\phi _{+}.
\end{cases}
\end{gather}
Moreover, if $ \left\vert \phi _{-}-\phi _{+}\right\vert $ is small enough, it holds that
\begin{gather}\label{con-stat-decay}
\displaystyle  \sum _{k=1}^{2}\left\vert \frac{\mathrm{d}^{k}}{\mathrm{d}x ^{k}}\left( \bar{\rho}, \bar{\phi}\right) \right\vert+\left\vert \bar{\rho}(x)-\rho _{+}\right\vert +\left\vert \bar{\phi}(x)-\phi _{+}\right\vert \leq C 		{\mathop{\mathrm{e}}}^{-\lambda x}\left\vert \phi _{-}- \phi _{+}\right\vert,\ \ x \geq 0
\end{gather}
for some constants $ C>0 $ and $ \lambda>0 $ which may depend on $ \rho _{+} $, $ a $ and $ b $, but independent of $ \phi _{-}-\phi _{+} $.

\end{theorem}
\begin{remark}
 Under the condition \eqref{p-condition}, the integral $\int _{0}^{\rho _{+}}\frac{p'(s)}{\mu s}\mathrm{d} s  $ in \eqref{fip-condi} is positive, but not necessarily finite since $\frac{p'(s)}{\mu s} \rightarrow +\infty  $ as $ s \rightarrow 0 $ is possible. While in the case of $\int _{0}^{\rho _{+}}\frac{p'(s)}{\mu s}\mathrm{d} s =+\infty $, the condition \eqref{fip-condi} is free for any given $ \phi _{-} $ and $ \phi _{+} $.
\end{remark}

\vspace*{3mm}

Our second result is the asymptotic stability of the stationary solutions obtained in Theorem \ref{thm-sta-problem}, which is stated in the following theorem.
\begin{theorem}\label{thm-stability}
Let the conditions in Theorem \ref{thm-sta-problem} hold and define
\begin{gather*}
 \displaystyle  \varphi _{0}=-\int _{x}^{\infty}\left( \rho _{0}(y)-\bar{\rho}(y) \right)\mathrm{d}y, \ \ \  \Phi _{0}= \phi _{0}(x)-\bar{\phi}.
 \end{gather*}
 If $\varphi _{0}\in H ^{3} $, $ m _{0}\in H ^{2} $ and $ \Phi _{0}\in H ^{4} $ with $\inf\limits_{x \in \mathbb{R}_{+}}\rho_0(x)>0$ (namely $\inf \limits_{x \in \mathbb{R}_{+}}\left\{  \varphi _{0x}+\bar{\rho}  \right\}>0)$, then there exists a constant $ \delta _{0}>0 $ such that if
\begin{gather*}
\displaystyle  \|\varphi _{0}\|_{3}+\|m _{0}\|_{2}+\|\Phi _{0}\|_{4} +\left\vert \phi _{-}- \phi _{+}\right\vert \leq \delta _{0},
\end{gather*}
the problem \eqref{ori-eq-a}--\eqref{ori-eq-c} subject to the initial-boundary conditions \eqref{initial-ori}--\eqref{far-field} admits a unique classical solution $ (\rho(x,t),m(x,t), \phi(x,t)) $ in $ \mathbb{R}_{+}\times (0, \infty) $ satisfying $\inf\limits_{x \in \mathbb{R}_{+}}\rho(x)>0$ for any $t>0$ and
\begin{gather}\label{large-time-thm}
\displaystyle \lim _{t \rightarrow \infty} \sup _{x \in \mathbb{R}_{+}}\left\vert (\rho,m, \phi)(x,t)-(\bar{\rho},0,\bar{\phi})(x)\right\vert \rightarrow 0.
\end{gather}
\end{theorem}
\begin{remark}
With the condition $ \Phi _0 \in H ^{4} $, we can define the initial values of $ \phi _{t} $ and $ \phi _{tt} $ through the equation for $ \phi $. That is
\begin{align}
 \displaystyle \Phi _{t0}&: =\phi _{t0}=(\Phi _{0})_{xx}+a (\varphi _{0})_{x}-b\Phi _{0},\label{fida-t-initial}\\
  \displaystyle  \displaystyle \Phi _{tt0}&:=\phi _{tt0}= (\Phi _{t0})_{xx}-a (m _{0})_{x}-b \Phi _{t0}.\label{fida-tt-intial}
 \end{align}
 These initial values of time derivatives are of importance in deriving the higher-order estimates in section 4. Furthermore, we always assume that the initial data is compatible with the boundary conditions at $ x=0 $.
\end{remark}

\section{Stationary problem (Proof of Theorem \ref{thm-sta-problem})}\label{sec-stationary-solution}
In this section, we shall study the stationary problem \eqref{stat-1}--\eqref{stat-3} and complete the proof of Theorem \ref{thm-sta-problem}. To this end, we first reformulate our problem \eqref{stat-1}--\eqref{stat-3}, and then prove the existence and uniqueness of solutions. Finally, we derive the monotone and decay properties of solutions.

\subsection{Reformulation of our problem} We start by proving the following lemma, which plays a key role in the reformulation of our problem.
\begin{lemma}\label{lem-stat-axu}
If $ f $ is a solution to the problem
\begin{gather}\label{eq-f}
\displaystyle \begin{cases}
 \displaystyle f _{x}=\omega(f(x)),\ \ x \in \mathbb{R}_{+},\\
\displaystyle f(+\infty)=k _{0},
\end{cases}
\end{gather}
where $ \omega $ is a continuous function, $ k _{0}$ is a constant. Then we have
\begin{gather*}
\displaystyle \lim _{x \rightarrow +\infty}f _{x}=\omega(k _{0})=0.
\end{gather*}

\end{lemma}
\begin{proof}
Since $ \omega $ is continuous, and $ \displaystyle\lim _{x \rightarrow +\infty}f(x)=k _{0} $, we have $ \displaystyle\lim _{x \rightarrow +\infty}f _{x}=\omega(k _{0}) $. It remains to show $ \omega(k _{0})=0 $. We proof this by contradiction. Supposing that $ \omega(k _{0})\neq 0 $, without loss of generality, we assume $ \omega(k _{0})>0 $. Thanks to the continuity of $ \omega $ and $ \displaystyle\lim _{x \rightarrow +\infty}f(x)=k _{0} $, there exists a constant $ X _{0} \in \mathbb{R}_{+} $ such that for any $ x \geq X _{0} $, $ \omega(f(x))>\frac{\omega(k _{0})}{2}>0 $. This along with $ \eqref{eq-f}_{1} $ implies that
\begin{align*}
\displaystyle f(x)=f(X _{0})+\int _{X _{0}}^{x}\omega(f(x))\mathrm{d}x \geq  f(X _{0})+\frac{\omega(k _{0})}{2}(x -X _{0}) \rightarrow +\infty\ \ \mbox{as}\ \ x \rightarrow +\infty,
\end{align*}
which contradicts the fact $ f(+\infty)=k _{0} $. Hence, $ \omega(k _{0})=0 $. The proof of Lemma \ref{lem-stat-axu} is complete.

\end{proof}
Define
\begin{align}\label{F-defi}
\displaystyle F(s):=\int_{\rho _{+}} ^{s}\frac{p'(\tau)}{\mu \tau}\mathrm{d}\tau
\end{align}
for any $ s >0 $. Clearly, $ F(\rho _{+})=0 $. We claim that under the conditions $ \eqref{p-condition} $, \eqref{fip-condi} and  $ \rho _{+}>0 $, there exists a unique constant $ \rho _{-} >0$ such that
\begin{align}\label{rho-fu}
\displaystyle F(\rho _{-})=\phi _{-}-\phi _{+}
\end{align}
and
\begin{align}\label{fip-to-rho}
\displaystyle \phi _{-}>\phi _{+}  \Longleftrightarrow \rho _{-}>\rho _{+}~~(\mathrm{resp}.~\phi _{-}<\phi _{+}  \Longleftrightarrow \rho _{-}<\rho _{+}).
\end{align}
Indeed, in view of \eqref{p-condition}, we know that
\begin{align}\label{F-dervi}
\displaystyle F'(s) =\frac{p'(s)}{s \mu}> \frac{a}{b}> 0.
\end{align}
This implies that the function $ F(s) $ is strictly monotonically increasing. Furthermore, we have $ F(s)<0 $ if $ s<\rho _{+} $, and $ \displaystyle\lim _{s \rightarrow +\infty} F(s)=+\infty $. For the case $ \phi _{-}>\phi _{+} $, since $ F(\rho _{+})=0<\phi _{-}-\phi _{+}<\infty $, then there exists a unique constant $ \rho _{-} \in(\rho _{+},+\infty) $ such that \eqref{rho-fu} holds. For the case $ \phi _{-}< \phi _{+} $, we have $ \phi _{-}-\phi _{+}<0 $. If $  \int _{0}^{\rho _{+}}\frac{p'(\tau)}{\mu \tau}\mathrm{d}\tau=\infty $, we know that $\displaystyle \lim _{s \rightarrow 0}F(s)=- \infty<\phi _{-}-\phi _{+}<0= F(\rho _{+}) $, then similar to the case $ \phi_{-}>\phi _{+} $, there exists a unique constant $ \rho _{-} \in(0,\rho _{+}) $ such that $ F(\rho _{-})=\phi _{-}-\phi _{+} $. Now it remains to consider the case when $ \phi _{-}<\phi _{+} $ and $ \int _{0}^{\rho _{+}}\frac{p'(\tau)}{\mu \tau}\mathrm{d}\tau<\infty $. In this case, since the $ F(s) $ is continuous, monotonic and bounded below, we can extend $ F(s) $ by defining $\displaystyle F(0):=\lim _{s \rightarrow 0}F(s)>-\infty $. Then the extended function $ F(s) $ is continuous on $ [0, \rho _{+}] $. Furthermore, from \eqref{fip-condi}, we get $ F(0)<\phi _{-}-\phi _{+}<0=F(\rho _{+}) $. Hence, there exists a unique constant $ \rho _{-} \in(0,\rho _{+}) $ such that $ F(\rho _{-})=\phi _{-}-\phi _{+} $. Then \eqref{rho-fu} is proved. Moreover, with the help of \eqref{rho-fu} and \eqref{F-dervi}, we immediately get \eqref{fip-to-rho}. We thus finish the proof of the claim.

To proceed, assume that $ (\bar{\rho},\bar{\phi}) $ is a classical solution to \eqref{stat-1}--\eqref{stat-3} with $ \bar{\rho}>0 $. Dividing \eqref{stat-1} by $ \bar{\rho} $ and integrating the resulting equation over $ (x,+\infty) $, we get
 \begin{gather}\label{barfip}
\displaystyle \bar{\phi}(x)= F(\bar{\rho}(x))-F(\rho _{+})+\phi _{+},
\end{gather}
where $ F(s) $ is as in \eqref{F-defi}. Sending $ x \rightarrow 0 ^{+} $ along with \eqref{stat-c}, \eqref{rho-fu} and the fact $ F(\rho _{+})=0 $, we get
\begin{align}
\displaystyle \lim _{x \rightarrow 0 ^{+}}F(\bar{\rho}(x))=F(\rho _{-}). \nonumber
\end{align}
By using the monotonicity and continuity of $ F(s) $, we further have that
\begin{align}
\displaystyle \bar{\rho}(0)=\lim _{x \rightarrow 0 ^{+}}\bar{\rho}(x)=\lim _{x \rightarrow 0 ^{+}}F ^{-1}(F(\bar{\rho}(x)))=F ^{-1}(F (\rho _{-}))=\rho _{-}. \nonumber
\end{align}
Inserting \eqref{barfip} into \eqref{stat-2}, we get
\begin{gather}\label{STAT-REFOR}
\displaystyle [F(\bar{\rho})]_{xx}= b[F(\bar{\rho})-F(\rho _{+})]+b \phi _{+}
- a \bar{\rho} .
\end{gather}
Multiplying \eqref{STAT-REFOR} by $ 2F ^{\prime}(\bar{\rho})\bar{\rho}_{x} $, it follows that
\begin{align}\label{equ-multided}
\displaystyle 2F ^{\prime}(\bar{\rho})\bar{\rho}_{x}[F'(\bar{\rho})\bar{\rho}_{x}]_{x}& =2\left\{ b[F(\bar{\rho})-F(\rho _{+})]+b \phi _{+} - a \bar{\rho}\right\}=:2 H(\bar{\rho})
F ^{\prime}(\bar{\rho})\bar{\rho}_{x},
\end{align}
with
\begin{gather}\label{H-defi}
\displaystyle H(s): = b[F(s)-F(\rho _{+})]+b \phi _{+}
- a s .
\end{gather}
Thus,
\begin{gather}\label{modified-bar-rhoeq}
\displaystyle [F'(\bar{\rho})\bar{\rho}_{x}]^{2}=G(\bar{\rho})+C _{0} \geq 0,\ \ x \in \mathbb{R}_{+}
\end{gather}
for some constant $ C _{0} $ and some function $ G(s) $ with
\begin{gather}\label{G-deriva}
\displaystyle  G'(s)=2F'(s)H(s).
\end{gather}
By virtue of \eqref{p-condition} and \eqref{H-defi}, we get
\begin{gather}\label{H-deriva}
\displaystyle H'(s) =b F'(s)-a>0
\end{gather}
for any $ s >0 $. Thanks to the condition $ a \rho _{+}= b \phi _{+} $, it holds that $ H(\rho _{+})=0 $. This along with \eqref{H-deriva} yields that
\begin{gather}\label{H-rho-plus}
 \displaystyle H(s)>H(\rho _{+})=0 \ \mbox{if}\ s>\rho _{+}\  \mbox{and}\   H(s)<H(\rho _{+})=0\ \mbox{if}\ s< \rho _{+}.
 \end{gather}
Then by \eqref{F-dervi}, \eqref{G-deriva} and \eqref{H-rho-plus}, we get $ G'(s)>0 $ if $ s >\rho _{+} $ and $ G'(s)<0 $ if $ s <\rho _{+} $. This gives
 \begin{gather}\label{G-positive}
 \displaystyle G(s)-G(\rho _{+})>0
 \end{gather}
 for any $s \neq \rho _{+} $. We claim that $ C _{0}=-G(\rho _{+}) $. Otherwise, we have $ C _{0}<-G(\rho _{+}) $ or $ C _{0}>-G(\rho _{+}) $.  If $ C _{0}<-G(\rho _{+}) $, by the continuity of $ G $ and $ \bar{\rho} $, there exists a constant $ K _{0} >0$ such that if $ x \geq K _{0} $,
\begin{gather*}
\displaystyle  G(\bar{\rho}) <\frac{G(\rho _{+})-C _{0}}{2}<- C _{0}.
\end{gather*}
Then $ G(\bar{\rho})+C _{0}<0 $ for $ x \geq K _{0} $. This contradicts to \eqref{modified-bar-rhoeq}. If $ C _{0}>-G(\rho _{+}) $, using \eqref{G-positive}, we get
 $ \bar{\rho}_{x}\neq 0  $ for any $ x \in \mathbb{R}_{+} $. Therefore, for any $ x \in \mathbb{R}_{+} $, it holds that
 \begin{align*}
 \displaystyle \bar{\rho }_{x}= -\frac{\sqrt{G(\bar{\rho})+C _{0}}}{F'(\bar{\rho})}\  \mbox{if}\ \rho _{-}>\rho _{+}\ \  \mbox{and}\ \ \bar{\rho}_{x}= -\frac{\sqrt{G(\bar{\rho})+C _{0}}}{F'(\bar{\rho})}\  \mbox{if}\ \rho _{-}<\rho _{+}.
 \end{align*}
With the fact $ \displaystyle \bar{\rho}(+\infty)=\rho_{+} $ and Lemma \ref{lem-stat-axu}, we have $ C _{0}=-G(\rho _{+}) $. This is a contradiction. Hence, we have $ C _{0}=-G(\rho _{+}) $ and
 \begin{gather}\label{rho-eq-square}
\displaystyle \begin{cases}
  [F'(\bar{\rho})\bar{\rho}_{x}]^{2}=G(\bar{\rho})-G(\rho _{+}),\ \ x \in \mathbb{\mathbb{R}}_{+},\\
  \bar{\rho}(0)=\rho _{-},\ \ \bar{\rho}(+\infty)=\rho _{+}.
\end{cases}
\end{gather}
This together with \eqref{F-dervi} and \eqref{G-positive} implies that
 $ \bar{\rho}_{x} \leq 0$ if $ \rho _{-}>\rho _{+} $ and $ \bar{\rho}_{x} \geq 0 $ if $ \rho _{+}>\rho _{-} $, and that
 \begin{gather}\label{deivative-zero}
 \displaystyle  \bar{\rho}_{x}(x)=0 \ \mbox{if and only if}\ \bar{\rho}(x)=\rho _{+}
 \end{gather}
for any $ x \in \mathbb{R}_{+} $. Hence, we can solve $ \bar{\rho}_{x} $ from \eqref{rho-eq-square} that
 \begin{align}
 \displaystyle \displaystyle \bar{\rho }_{x}= -\frac{\sqrt{G(\bar{\rho})-G(\rho _{+})}}{F'(\bar{\rho})}\  \mbox{if}\ \phi _{-}>\phi _{+}\ \  \mbox{and}\ \ \bar{\rho}_{x}= -\frac{\sqrt{G(\bar{\rho})-G(\rho _{+})}}{F'(\bar{\rho})}\  \mbox{if}\ \phi _{-}<\phi _{+}, \nonumber
 \end{align}
 where we have used \eqref{fip-to-rho}.
\\
Summing up, we have the following lemma.
 \begin{lemma}\label{lem-stat-property}
 Under the conditions of Theorem \ref{thm-sta-problem}, if $ (\bar{\rho},\bar{\phi}) $ is a classical solution to the problem \eqref{stat-1}--\eqref{stat-3} satisfying $ \bar{\rho}(x)>0 $ for any $ x \in \mathbb{R}_{+} $, then $ (\bar{\rho},\bar{\phi}) $ is also a solution to the following problem:
 \begin{gather}\label{stat-equiv-syst}
 \displaystyle   \begin{cases}
    \displaystyle  \bar{\rho}_{x}= -\frac{\sqrt{G(\bar{\rho})-G(\rho _{+})}}{F'(\bar{\rho})} \ \mbox{if}\  \phi _{-}>\phi _{+}\ \  (\mathrm{reps.}\  \bar{\rho}_{x}= \frac{\sqrt{G(\bar{\rho})-G(\rho _{+})}}{F'(\bar{\rho})}\ \mbox{if}\ \phi _{-}<\phi _{+}),\\
  \displaystyle   \displaystyle \bar{\phi}(x)=F(\bar{\rho}(x))-F(\rho _{+})+\phi _{+},\\
  \displaystyle \bar{\rho}(0)=\rho _{-},\ \ \bar{\rho}(+ \infty)=\rho _{+}.
  \end{cases}
 \end{gather}
 Here $ F(s) $ and $ G(s) $ are given in \eqref{F-defi} and \eqref{G-deriva}, respectively, and $ \rho _{-}>0 $ is determined by \eqref{rho-fu}.
  \end{lemma}

In the following, we shall show that the problem \eqref{stat-equiv-syst} is indeed equivalent to \eqref{stat-1}--\eqref{stat-3}.
\begin{lemma}[Reformulation]\label{lem-stat-equivalent}
Suppose that the conditions of Theorem \ref{thm-sta-problem} hold. Then $ (\bar{\rho}(x), \bar{\phi}(x)) $ is a classical solution to the problem \eqref{stat-1}--\eqref{stat-3} satisfying $ \bar{\rho}>0 $, if and only if it is a classical solution to the problem \eqref{stat-equiv-syst}.
\end{lemma}
\begin{proof}
 In view of Lemma \ref{lem-stat-property}, it remains to show that if $ (\bar{\rho}, \bar{\phi}) $ is a solution to the problem \eqref{stat-equiv-syst}, then $ (\bar{\rho}, \bar{\phi}) $ solves the problem \eqref{stat-1}--\eqref{stat-3}. By using $ \eqref{stat-equiv-syst}_{3} $ and \eqref{F-dervi}, one can easily derive \eqref{stat-1}. Thanks to \eqref{rho-fu} and $ \eqref{stat-equiv-syst}_{3} $, we have $ \bar{\phi}(0)=\phi _{-} $ and $ \bar{\phi}(+\infty)=\phi _{+} $. To show \eqref{stat-2}, by \eqref{equ-multided}, \eqref{rho-eq-square} and $ \eqref{stat-equiv-syst}_{3} $, it suffices to show that $ \bar{\rho}_{x}\neq 0 $ for any $ x \in \mathbb{R}_{+} $. We prove this for the case $ \rho _{-}>\rho _{+}~(i.e., \phi _{-}>\phi _{+}) $, and the proof for the case $ \rho _{-}<\rho _{+} ~(i.e.,\phi _{-}<\phi _{+}  )$ is similar. Since $ \rho _{-}>\rho _{+} $, we have $ \bar{\rho}_{x} \leq 0 $ for any $ x \in \mathbb{R}_{+} $. Denote
 \begin{gather*}
 \displaystyle \mathcal{D}(\bar{\rho}):= -\frac{\sqrt{G(\bar{\rho})-G(\rho _{+})}}{F'(\bar{\rho})}.
 \end{gather*}
 We claim that $ \mathcal{D}(\bar{\rho}) $ is Lipschitz continuous on $ [\rho _{+}, \rho _{-}] $. With this claim, we can prove that $ \bar{\rho}_{x} <0$ for $ x \in \mathbb{R}_{+} $, and hence finish the proof. Indeed, if there exists a point $  x _{0}\in \mathbb{R}_{+} $ such that $ \bar{\rho} _{x}(x _{0})=0 $, then from \eqref{deivative-zero}, we have $ \bar{\rho}(x _{0})=\rho _{+} $. This implies that $ \bar{\rho} $ is a solution to the following problem
\begin{gather}\label{contra-eq}
\begin{cases}
  \displaystyle \rho_{x}^{\ast}=\mathcal{D}(\rho ^{\ast}),\ \  0 \leq x \leq x _{0},\\
  \displaystyle \rho ^{\ast}( x_{0})=\rho _{+}.
\end{cases}
  \end{gather}
 Since $ \mathcal{D}(\bar{\rho}) $ is Lipschitz continuous on $ [\rho _{+},\rho _{-}]  $, the problem \eqref{contra-eq} admits a unique solution on $ [0,x _{0}] $. While $ \rho ^{\ast}\equiv \rho _{+} $ is also a  solution to \eqref{contra-eq}, and obviously, $ \bar{\rho}\not\equiv \rho _{+} $. This is a contradiction. Therefore, $ \bar{\rho}_{x} <0 $  for any $ x \in \mathbb{R}_{+} $. Now it remains to prove the claim that $ \mathcal{D}(\bar{\rho}) $ is Lipschitz continuous on $ [\rho _{+}, \rho _{-}] $. With the help of \eqref{p-condition}, \eqref{F-dervi} and \eqref{G-positive}, we know that $ \mathcal{D}(\bar{\rho}) $ is differentiable if $ \bar{\rho}\neq \rho _{+} $. Furthermore, a direct computation gives
\begin{align}\label{D-deriv-com}
\displaystyle &\displaystyle  \mathcal{D}'(\bar{\rho}) =\frac{\sqrt{G(\bar{\rho})-G(\rho _{+})}}{[F'(\bar{\rho})]^{2}}F''(\bar{\rho})- \frac{H(\bar{\rho})}{\sqrt{G(\bar{\rho})-G(\rho _{+})}}.
\end{align}
By using \eqref{G-deriva}, \eqref{H-deriva} and L'H\^{o}pital's rule, we have
\begin{align}\label{Fisr-limi-stat}
\displaystyle  \lim _{\,\bar{\rho}\rightarrow \rho _{+}^{\,+}}\frac{H ^{2}(\bar{\rho})}{G(\bar{\rho})-G(\rho _{+})} =  \lim _{\,\bar{\rho}\rightarrow \rho _{+}^{\,+}}\frac{2H'(\bar{\rho})H(\bar{\rho})}{G'(\bar{\rho})}= \lim _{\,\bar{\rho}\rightarrow \rho _{+}^{\,+}}\frac{H'(\bar{\rho})}{F'(\bar{\rho})}=\frac{bF'(\rho _{+})-a}{F'(\rho _{+})}>0.
\end{align}
From \eqref{H-rho-plus} and \eqref{G-positive}, we get $ \frac{H(\bar{\rho})}{\sqrt{G(\bar{\rho})-G(\rho _{+})}}>0 $ for $ \bar{\rho}>\rho _{+} $. This along with \eqref{Fisr-limi-stat} yields that
\begin{gather*}
\displaystyle  \lim _{\,\bar{\rho}\rightarrow \rho _{+}^{\,+}}\frac{H(\bar{\rho})}{\sqrt{G(\bar{\rho})-G(\rho _{+})}}= \left( \lim _{\,\bar{\rho}\rightarrow \rho _{+}^{\,+}}\frac{H ^{2}(\bar{\rho})}{G(\bar{\rho})-G(\rho _{+})} \right)^{\frac{1}{2}}=\sqrt{\frac{aF'(\rho _{+})-b}{F'(\rho _{+})}}.
\end{gather*}
Therefore,
\begin{align*}
\displaystyle  \lim _{\,\bar{\rho}\rightarrow \rho _{+}^{\,+}}\mathcal{D}'(\bar{\rho})&=-\lim _{\,\bar{\rho}\rightarrow \rho _{+}^{\,+}}\frac{H(\bar{\rho})}{\sqrt{G(\bar{\rho})-G(\rho _{+})}} =-\sqrt{\frac{aF'(\rho _{+})-b}{F'(\rho _{+})}}.
 \end{align*}
 Notice that $ \frac{\mathcal{D}(\bar{\rho})-\mathcal{D}(\rho _{+})}{\bar{\rho}-\rho _{+}}=\frac{\mathcal{D}(\bar{\rho})}{\bar{\rho}-\rho _{+}} <0$ for $ \bar{\rho}>\rho _{+} $, and that
 \begin{align*}
 \displaystyle \lim _{\bar{\rho}\rightarrow \rho _{+}}\frac{\left\vert \mathcal{D}(\bar{\rho})-\mathcal{D}(\rho _{+})\right\vert ^{2}}{\left\vert \bar{\rho}-\rho _{+} \right\vert ^{2}}&= \frac{1}{[F'(\rho _{+})]^2} \lim _{\bar{\rho}\rightarrow \rho _{+}}\frac{G(\bar{\rho})-G(\rho _{+})}{\left \vert \bar{\rho
 }-\rho _{+} \right \vert^{2} } =\lim _{\bar{\rho}\rightarrow \rho _{+}} \frac{H(\bar{\rho})}{\bar{\rho}-\rho _{+}}
  \nonumber \\
  &\displaystyle =\lim _{\bar{\rho}\rightarrow \rho _{+}} \frac{H'(\bar{\rho})}{F'(\rho _{+})}=\frac{aF'(\rho _{+})-b}{F'(\rho _{+})}>0,
 \end{align*}
 due to \eqref{G-deriva}, \eqref{H-deriva} and L'H\^{o}pital's rule, we have
 \begin{align*}
  \displaystyle \mathcal{D}_{+}'(\rho _{+})=\lim _{\bar{\rho}\rightarrow \rho _{+}^{\,+}} \frac{\mathcal{D}(\bar{\rho})-\mathcal{D}(\rho _{+})}{\bar{\rho}-\rho _{+}}=-\left( \lim _{\bar{\rho}\rightarrow \rho _{+}}\frac{\left\vert \mathcal{D}(\bar{\rho})-\mathcal{D}(\rho _{+})\right\vert ^{2}}{\left\vert \bar{\rho}-\rho _{+} \right\vert ^{2}} \right)^{\frac{1}{2}}= -\sqrt{\frac{bF'(\rho _{+})-a}{F'(\rho _{+})}},
  \end{align*}
where $  \mathcal{D}_{+}'(\rho _{+}) $ is the right derivative of $ \mathcal{D}(\bar{\rho}) $ at $ \rho _{+} $. We thus have $ \displaystyle\lim _{\,\bar{\rho}\rightarrow \rho _{+}^{\,+}}\mathcal{D}'(\bar{\rho})= \mathcal{D}_{+}'(\rho _{+})$. This in combination with \eqref{D-deriv-com} yields that $ \mathcal{D}'(\bar{\rho}) $ is continuous on $ [\rho _{+},\rho _{-}] $, and thus $ \left\vert \mathcal{D}'(\bar{\rho})\right\vert \leq C(\rho _{-},\rho _{+}) $ for some constant $C(\rho _{-},\rho _{+})>0  $ depending on $ \rho _{-} $ and $ \rho _{+} $. This implies that $ \mathcal{D}(\bar{\rho}) $ is Lipschitz continuous on $ [\rho _{+},\rho _{-}]  $. The proof of the present lemma is complete.
\end{proof}

\subsection{Existence and uniqueness of solutions} In this section, we will prove that the problem \eqref{stat-1}--\eqref{stat-3} admits a unique solution $ (\bar{\rho},\bar{\phi}) $ with $ \bar{\rho}>0 $.
Thanks to Lemma \ref{lem-stat-equivalent}, it now suffices to consider the problem \eqref{stat-equiv-syst}. As before, we focus only on the case $ \rho _{-}>\rho _{+}~(i.e.,~\phi _{-}>\phi _{+}) $, the proof for the case $ \rho _{-}<\rho _{+}~(i.e.,~\phi _{-}<\phi _{+}) $ is similar and so omitted. Let us begin with the following ODE problem
\begin{gather}\label{pro-affil}
  \begin{cases}
    \displaystyle  \bar{\rho}_{x}=\mathcal{D}(\bar{\rho}),\ \ x >0, \\
  \displaystyle \bar{\rho}(0)=\rho _{-}.
  \end{cases}
\end{gather}
By the Lipschitz continuity of $ \mathcal{D}(\bar{\rho}) $ on $ [\rho _{+}, \rho _{-}] $,  we conclude that the problem \eqref{pro-affil} admits a unique solution on $ [0,X _{\ast}) $ for some $ X _{\ast}\in \mathbb{R}_{+} $. Then by the contradiction argument and discussions in Step 1 on the uniqueness of solutions to \eqref{contra-eq}, we get $ \bar{\rho}(x)>\rho _{+} $ for any $ x \in  [0,X _{\ast})  $. This, along with the standard extension theorem for ordinary differential equations, implies that the solution $ \bar{\rho} $ to the problem \eqref{pro-affil} exists globally in $ \mathbb{R}_{+} $, and for any $ x \in \mathbb{R}_{+} $, $ \bar{\rho}(x)>\rho _{+} $. In addition, notice from \eqref{deivative-zero} that $ \bar{\rho}_{x}(x)=0  $ if and only if $ \bar{\rho}(x)=\rho _{+} $, we have that
\begin{align}\label{stat-rho-neg-deri}
\displaystyle   \bar{\rho}_{x}<0\ \ \mbox{for any }x \in \mathbb{R}_{+},
\end{align}
 and that $ \displaystyle\lim _{x \rightarrow +\infty}\bar{\rho}(x) $ exists. Denoting $ \displaystyle \bar{\rho}(+\infty):=\lim _{x \rightarrow +\infty}\bar{\rho}(x) $, from Lemma \ref{lem-stat-axu}, we obtain $ G(\bar{\rho}(+\infty))-G(\rho _{+})=0 $. This combined with \eqref{G-positive} give rises to $ \bar{\rho}(+\infty)=\rho _{+} $. With $ \bar{\rho}(x) $ at hand, we can define $ \bar{\phi}(x) $ from $ \eqref{stat-equiv-syst}_{2} $. Clearly, $ (\bar{\rho},\bar{\phi}) $ is a solution to \eqref{stat-equiv-syst} for $ \rho _{-}>\rho _{+} $. Finally, since $ \rho _{+} \leq\bar{\rho}(x)\leq \rho _{-} $, the uniqueness of solutions can be proved by the Lipschitz continuity of $ \mathcal{D}(\bar{\rho}) $ on $ [\rho _{+}, \rho _{-}] $.

\subsection{Monotonicity and decay properties}
Recalling \eqref{fip-to-rho} and \eqref{stat-rho-neg-deri}, we get $ \bar{\rho}'(x)<0 $ if $ \phi _{-}>\phi _{+} $. In a manner similar to the derivation of \eqref{stat-rho-neg-deri}, we  have $ \bar{\rho}'(x)>0 $ if $ \phi _{-}<\phi _{+} $. With the help of \eqref{F-dervi}, \eqref{barfip} and the properties of $ \bar{\rho}_{x} $, we have
\begin{gather*}
\displaystyle \bar{\phi}_{x}<0 \ \mbox{if}\ \phi _{-}>\phi _{+}\ \ \mbox{and}\ \ \bar{\phi}_{x}>0 \ \mbox{if}\ \phi _{-}<\phi _{+}.
\end{gather*}
This gives \eqref{monotonic}. Now let us turn to the decay properties of the solution. By using \eqref{rho-fu} and \eqref{F-dervi}, we get
\begin{align}\label{small-fip-rho-stat}
\displaystyle  \left\vert \phi _{-}-\phi _{+}\right\vert=\left\vert \int _{\rho _{-}}^{\rho _{+}}F'(s)\mathrm{d}s\right\vert \geq \frac{a}{b}\left\vert \rho _{-}-\rho _{+}\right\vert.
\end{align}
Thus $ \left\vert \phi _{-}-\phi _{+} \right\vert\ll 1 $ implies $ \left\vert \rho _{-}-\rho _{+}\right\vert \ll 1 $. In the following, without loss of generality, we assume that $ [\rho _{+},\rho _{-}] \subset [\rho _{+}, \rho _{+}+1] $, and thus for any continuous function $ f $ defined on $ \mathbb{R}_{+} $, $ \sup _{x \in [\rho _{+},\rho _{-}]}f(x) $ depends only on $ \rho _{+} $. If $ \phi _{-}>\phi _{+}~(i.e., \rho _{-}>\rho _{+}) $, recalling \eqref{monotonic} and \eqref{stat-equiv-syst}, we have
\begin{gather}\label{decay-eq}
\displaystyle  \rho _{-}>\bar{\rho}>\rho _{+} \ \mbox{and}\ \ \bar{\rho}_{x}= -\frac{\sqrt{G(\bar{\rho})-G(\rho _{+})}}{F'(\bar{\rho})}
\end{gather}
for any $ x \in \mathbb{R}_{+} $. Owing to \eqref{p-condition}, \eqref{F-dervi}, \eqref{G-deriva} and the condition $ a \rho _{+}=b \phi _{+} $, we get
\begin{align}\label{G-twic-deri}
\displaystyle G''(s)&=2[F'(s)H(s)]'=2F'(s)H'(s)+2F''(s)H(s)
 \nonumber \\
 & \displaystyle=2F'(s)\left( bF'(s)-a \right) + \frac{2}{\mu s ^{2}}\left( s p''(s)- p'(s) \right) \left( b[F(s)-F(\rho _{+})]+b \phi _{+}
- a s \right)
 \nonumber \\
 & \displaystyle \geq 2F'(s)\left( bF'(s)-a \right) - \left\vert s- \rho _{+}\right\vert \sup _{\iota \in [\rho _{+}, \rho _{-}]}\frac{2}{\mu \iota ^{2}}\left\vert \left( \left\vert \iota p''(\iota)\right\vert+ p'(\iota) \right)\right\vert \left(b \sup _{\iota \in [\rho _{+}, \rho _{-}]}F'(\iota) +a  \right)
 \nonumber \\
    & \geq 2F'(s)\left( bF'(s)-a \right) -C(\rho _{+})\left\vert s -\rho _{+}\right\vert
\end{align}
for any $ s \in[\rho _{+}, \rho _{-}] $, where $ C(\rho _{+}) >0$ is a constant depending on $ \rho _{+} $. From \eqref{G-deriva} and \eqref{H-rho-plus}, we get $ G'(\rho _{+})=0 $. This combined with \eqref{F-dervi}, \eqref{G-twic-deri} and the Taylor expansion implies that
\begin{align}\label{G-minu}
\displaystyle G(\bar{\rho})-G(\rho _{+})&=G(\bar{\rho})-G(\rho _{+})-G'(\rho _{+})(\bar{\rho}-\rho _{+})=
\int _{0}^{1}\int _{0}^{s}G''(\tau(\bar{\rho}-\rho _{+})+\rho _{+})\mathrm{d}\tau \mathrm{d}s \left\vert \bar{\rho}-\rho _{+}\right\vert ^{2}
 \nonumber \\
 & \geq \frac{1}{2}\left[ 2F'(s)\left( bF'(s)-a \right) -C(\rho _{+})\left\vert \bar{\rho} -\rho _{+}\right\vert  \right]\left\vert \bar{\rho}- \rho _{+}\right\vert ^{2}
  \nonumber \\
  & \geq C(\rho _{+})\left\vert \bar{\rho}- \rho _{+}\right\vert ^{2},
\end{align}
provided $ \left\vert \rho _{-}-\rho _{+}\right\vert$ is suitably small, where $ C(\rho _{+}) $ is a positive constant depending on $ \rho _{+} $. Combining \eqref{G-minu} with \eqref{decay-eq}, we get
\begin{align*}
\displaystyle (\bar{\rho}-\rho _{+})_{x}=- \frac{\sqrt{G(\bar{\rho})-G(\rho _{+})}}{F'(\bar{\rho})}  \leq - \frac{\sqrt{G(\bar{\rho})-G(\rho _{+})}}{\displaystyle\sup _{x \in \mathbb{R}_{+}} F'(\bar{\rho})} \leq -\lambda _{1}(\bar{\rho}- \rho _{+})
\end{align*}
for some constant $ \lambda _{1}>0 $ depending on $ \rho _{+} $, provided $ \left\vert \rho _{-}-\rho _{+}\right\vert$ is suitably small, where we have used \eqref{p-condition} and \eqref{F-dervi}. Consequently, with \eqref{small-fip-rho-stat}, we have the following decay estimate:
\begin{gather}\label{decay-concl}
 \displaystyle  |\bar{\rho}-\rho _{+}|=\bar{\rho}-\rho _{+}  \leq \left( \rho _{-}-\rho _{+} \right)     {\mathop{\mathrm{e}}}^{- \lambda _{1} x}=\left\vert\rho _{-}-\rho _{+} \right\vert {\mathop{\mathrm{e}}}^{- \lambda _{1} x} \leq \frac{b}{a}\left\vert \phi _{-}-\phi _{+}\right\vert{\mathop{\mathrm{e}}}^{- \lambda _{1} x} ,\ \ x \geq 0.
 \end{gather}
 For the case $ \phi _{-}<\phi _{+}~(i.e.,~ \rho _{-}<\rho _{+}) $, it holds that
\begin{gather}\label{decay-eq-1}
\displaystyle  \rho _{-}<\bar{\rho}<\rho _{+} \ \mbox{and}\ \ \bar{\rho}_{x}= \frac{\sqrt{G(\bar{\rho})-G(\rho _{+})}}{F'(\bar{\rho})}.
\end{gather}
Using \eqref{F-dervi}, \eqref{G-minu} and \eqref{decay-eq-1}, we get
\begin{align*}
\displaystyle  \bar{\rho}_{x} \geq \frac{\sqrt{C(\rho _{+})}\left\vert \bar{\rho}-\rho _{+}\right\vert}{\displaystyle \sup _{x \in [\rho _{-}, \rho _{+}]} F'(\bar{\rho})} \geq -\lambda _{2}(\bar{\rho}-\rho _{+}),
\end{align*}
that is,
\begin{align*}
\displaystyle (\rho _{+}- \bar{\rho})_{x}+\lambda _{2}(\rho _{+}- \bar{\rho})\leq 0
\end{align*}
for some constant $ \lambda _{2}>0 $ depending on $ \rho _{+} $, provided $ \left\vert \rho _{-}-\rho _{+}\right\vert$ is suitably small. It thus holds that
\begin{align*}
\displaystyle  \left\vert\rho _{+}- \bar{\rho} \right\vert=\rho _{+}- \bar{\rho} \leq \left( \rho _{+}-\rho _{-} \right)    {\mathop{\mathrm{e}}}^{-\lambda _{2}x} \leq\frac{b}{a}\left\vert \phi _{-}-\phi _{+}\right\vert{\mathop{\mathrm{e}}}^{- \lambda _{2} x} ,\ \ x \geq 0.
\end{align*}
Finally, by \eqref{stat-2}, \eqref{rho-fu}, \eqref{barfip}, \eqref{STAT-REFOR} and \eqref{decay-concl}, we get \eqref{con-stat-decay}. The proof is complete. \hfill $ \square $

\section{Global existence and asymptotic stability}\label{sec:stability}
In this section, we are devoted to studying the asymptotic stability of the unique stationary solution to \eqref{ori-eq-a}--\eqref{ori-eq-a} obtained in Section \ref{sec-stationary-solution}. To this end, we first reformulate the problem with the technique of taking anti-derivative for $ \rho $.
\subsection{Reformulation of problem} 
\vspace*{3mm}

 Combining \eqref{ori-eq-a}--\eqref{ori-eq-c} with \eqref{ori-eq-a}--\eqref{ori-eq-c}, we have
 \begin{subnumcases}
 \displaystyle   \displaystyle (\rho -\bar{\rho})_{t}+m_{x}=0,\label{eq-perturbed-one-a}\\[1mm]
   \displaystyle m _{t}+\left( \frac{m ^{2}}{\rho} \right)_{x}+\left[ p(\rho)-p(\bar{\rho}) \right]_{x}=\mu \rho \phi _{x}-\mu \bar{\rho}\bar{\phi}_{x}- \alpha m,\\[1mm]
   \displaystyle (\phi- \bar{\phi})_{t}=\left( \phi- \bar{\phi} \right)_{xx}+a(\rho- \bar{\rho})-b(\phi- \bar{\phi}).
 \end{subnumcases}
  It follows from \eqref{eq-perturbed-one-a} that
\begin{gather*}
\displaystyle \int _{\mathbb{R}_{+}}(\rho- \bar{\rho}) \mathrm{d}x= \int _{\mathbb{R}_{+}}(\rho _{0}- \bar{\rho}) \mathrm{d}x=\varphi _{0}(0),
\end{gather*}
which, together with the condition $ \varphi _{0}\in H _{0}^{3}(\mathbb{R}_{+}) $ in Theorem \ref{thm-stability}, gives
\begin{gather*}
\displaystyle \int _{\mathbb{R}_{+}} \left( \rho- \bar{\rho} \right) \mathrm{d}x=0.
\end{gather*}
Defining the perturbation function $ (\varphi,\psi, \Phi) $
 \begin{gather}\label{pertu-variable}
 \displaystyle \varphi =- \int_{x}^{\infty} \left( \rho- \bar{\rho} \right)\mathrm{d}y, \ \ \psi =m,\ \ \Phi= \phi - \bar{\phi},
 \end{gather}
 with
 \begin{gather*}
 \displaystyle  (\varphi _{0},\psi _{0}, \Phi _{0}):=\left(- \int _{x}^{\infty}(\rho _{0}-\bar{\rho})\mathrm{d}y, m _{0},\phi  _{0}- \bar{\phi}\right),
 \end{gather*}
 we get the reformulated problem:
 \begin{gather*}
 \displaystyle \begin{cases}
   \displaystyle \varphi _{t}+\psi=0,\\[1mm]
   \displaystyle \psi _{t}+\left( \frac{\psi ^{2}}{\varphi _{x}+\bar{\rho}} \right)_{x} + p(\varphi _{x}+\bar{\rho})_{x}-p(\bar{\rho})_{x} = \mu \left( \rho \phi _{x}-\bar{\rho}\bar{\phi}_{x} \right)- \alpha \psi,\\[7pt]
   \displaystyle \Phi _{t}=\Phi _{xx}+a \varphi _{x}-b \Phi, \\
   \displaystyle \left. (\varphi,\psi, \Phi)\right \vert _{t=0}=(\varphi _{0},\psi _{0}, \Phi _{0}),\\
   \displaystyle \left. \left( \varphi, \psi,\Phi \right)\right \vert_{x=0}=(0,0,0),
 \end{cases}
 \end{gather*}
 and its linearized problem around $ \bar{\rho} $ is
 \begin{subnumcases}
   \displaystyle \varphi _{tt}-\left( p'(\bar{\rho})\varphi _{x} \right)_{x}+\alpha \varphi _{t}=\mathcal{F}+\mathcal{H},\label{problem-linearized-simple-a}\\[1mm]
   \displaystyle  \Phi _{t}-\Phi _{xx}+b \Phi=a \varphi _{x}, \label{problem-linearized-simple-c}\\[1mm]
   \displaystyle \left. (\varphi, \varphi _{t},\Phi)\right \vert _{t=0}=(\varphi _{0}, -\psi _{0}, \Phi _{0}),
    \\[1mm]
   \displaystyle \left. \left( \varphi, \varphi_t, \Phi\right)\right \vert_{x=0}=(0,0,0), \label{problem-linearized-simple-e}\\[1mm]
       \psi=-\varphi _{t}, \label{problem-linearized-simple-b}
 \end{subnumcases}
where
 \begin{gather}
 \mathcal{F}=\mathcal{F}_{1}+\mathcal{F}_{2},\ \ \mathcal{H}= \left( \frac{\varphi _{t} ^{2}}{\varphi _{x}+\bar{\rho}} \right)_{x},\label{F-defi-simple}
  \end{gather}
and
 \begin{gather}\label{F-1-2}
 \displaystyle  \mathcal{F}_{1}=[p(\varphi _{x}+\bar{\rho})-p(\bar{\rho})-p'(\bar{\rho})\varphi _{x}]_{x},\ \ \mathcal{F}_{2}=-\mu[\varphi _{x}\Phi _{x}+\varphi _{x}\bar{\phi}_{x}+\bar{\rho}\Phi _{x}].
 \end{gather}
 To proceed, we define the solution space of the problem \eqref{problem-linearized-simple-a}--\eqref{problem-linearized-simple-e} as follows:
\begin{align*}
\displaystyle X(0,T)=\left\{ (\varphi, \psi,\Phi)\vert\right.& \left.\varphi \in C ([0,T];H ^{3})\cap C ^{1}([0,T];H ^{2}),\, \psi \in C ([0,T];H ^{2})\cap C ^{1}([0,T];H ^{1}), \right.
 \nonumber \\
 \qquad&\left.  \Phi \in C([0,T];H ^{4}) \cap C ^{1}([0,T];H ^{2})\right\}
\end{align*}
for any $ T \in (0,+\infty) $.

Since we are interested in the case where the solution has no vacuum, naturally we require that $ \inf\limits_{x \in \mathbb{R}_{+}}\rho_0(x)>0$, namely
\begin{gather}\label{positivity}
\displaystyle \inf _{x \in \mathbb{R}_{+}}\left\{  \varphi _{0x}+\bar{\rho}  \right\}>0.
\end{gather}

For simplicity, we denote
\begin{gather*}
\displaystyle \mathcal{N}_{0}:= \|\varphi _{0}\|_{3}^{2} +\|\psi _{0}\|_{2}^{2} +\|\Phi _{0}\|_{4}^{2}.
\end{gather*}
Then by the standard parabolic theory and fixed point theorem (cf.~\cite{MATSUMARA-1977}), we have the following local existence result.
\begin{proposition}[Local existence]\label{prop-local}
Let the conditions of Theorem \ref{thm-sta-problem} hold. Assume $ \varphi _{0}\in H ^{3} $, $ \psi _{0}\in H ^{2} $ and $ \Phi _{0}\in H ^{4}$ such that \eqref{positivity} holds.
Then there exists a positive constant $ T _{0} $ depending on $ \mathcal{N}_{0} $ such that the initial-boundary value problem \eqref{problem-linearized-simple-a}--\eqref{problem-linearized-simple-e} admits a unique solution $ (\varphi(x,t),\psi(x,t),\Phi(x,t)) \in X(0,T _{0}) $ such that  $ \displaystyle  \inf _{x \in \mathbb{R}_{+}}\left\{  \varphi _{x}+\bar{\rho}  \right\}>0$ for an $0 \leq t \leq T _{0}$ and
 \begin{gather*}
\displaystyle \sup _{t \in [0,T _{0}]}\left( \|\varphi\|_{3} ^{2}+\|\psi\|_{2}^{2}+\|\Phi\|_{4}^{2} \right)  \leq 2 \mathcal{N}_{0}.
\end{gather*}
\end{proposition}
\vspace*{2mm}

In what follows, we are devoted to proving the following theorem on the global existence and uniqueness of solutions to the problem \eqref{problem-linearized-simple-a}--\eqref{problem-linearized-simple-e}.
\begin{proposition}\label{prop-key}
Let the conditions in Theorem \ref{thm-sta-problem} hold and assume $\varphi _{0}\in H ^{3} $, $ \psi _{0}\in H ^{2} $ and $ \Phi _{0}\in H ^{4}$ with \eqref{positivity}. Then there exists a suitably small constant $ \delta _{1}>0 $ independent of $ t $ such that if
\begin{gather*}\displaystyle  \|\varphi _{0}\|_{3}+\|\psi _{0}\|_{2}+\|\Phi _{0}\|_{4} + \left\vert \phi _{-}-\phi _{+}\right\vert \leq
\delta _{1},
\end{gather*}
the problem \eqref{problem-linearized-simple-a}--\eqref{problem-linearized-simple-e} admits a unique global solution $ (\varphi(x,t), \psi(x,t), \Phi(x,t))\in X(0,\infty)$ such that for any $ t \geq 0 $ there holds that
\begin{gather}\label{regularity-prop}
\displaystyle  \|\varphi\|_{3}+\|\psi\|_{2} +\|\Phi\|_{4}\leq C \delta _{1}
\end{gather}
and
\begin{align}\label{prop-esti}
\displaystyle  \int _{0}^{t} \Big(\|\Phi\|_{3}^{2}+\|(\varphi _{x},\varphi _{\tau},\psi,\Phi _{\tau})\|_{2}^{2}+\|(\varphi _{\tau \tau},\psi _{\tau},\Phi _{\tau \tau})\|_{1}^{2} \Big)\mathrm{d}\tau \leq C \delta _{1}^{2}
\end{align}
where $C$ is a constant independent of $t$.
\end{proposition}

Theorem \ref{thm-stability} will be proved by Proposition \ref{prop-local} and Proposition \ref{prop-key}. Next, we are devoted to proving Proposition \ref{prop-key}.
\vspace*{3mm}

 \subsection{Some preliminaries}
The proof of Proposition \ref{prop-key} is based on the combination of the local existence result in Proposition \ref{prop-local} with the \emph{a priori} estimates given in \eqref{regularity-prop}-\eqref{prop-esti}. In the sequel, we assume that $\left( \varphi,\psi,\Phi \right) \in X(0,T)$ is a solution to the problem \eqref{problem-linearized-simple-a}--\eqref{problem-linearized-simple-e} obtained in Proposition \ref{prop-local} for some $T>0$ and derive the \emph{a priori} estimates \eqref{regularity-prop}-\eqref{prop-esti} based on the technique of \emph{a priori} assumption. That is we first assume that the solution $(\varphi,\psi,\Phi)$ of \eqref{problem-linearized-simple-a}--\eqref{problem-linearized-simple-e} satisfies
\begin{align}\label{aproiri-assumpt}
\displaystyle \sup _{0 \leq t<T}\left\{\left\|(\varphi,\Phi)(\cdot, t)\right\|_{3}^2+\left\| \psi(\cdot, t)\right\|_{2}^2\right\}\leq \varepsilon^2,
\end{align}
where $\varepsilon>0 $ is a constant to be determined later, and then derive the \emph{a priori} estimates to obtain the global existence of solutions. Finally we justify that the global solutions obtained satisfy the above \emph{a priori} assumption and thus close our argument.

Using the fact $\varphi _{t}=-\psi$ from \eqref{problem-linearized-simple-b} and the Sobolev inequality, we have
\begin{gather}\label{samlness-sumup}
\displaystyle \sum _{k=0}^{2}\|\partial _{x}^{k}(\varphi,\Phi)(\cdot, t) \|_{L ^{\infty}(\mathbb{R}_{+})}+
\sum _{k=0}^{1}\|\partial _{x}^{k}(\psi,\varphi _{t})(\cdot, t) \| _{L ^{\infty}(\mathbb{R}_{+})}\leq C \varepsilon.
\end{gather}
Denote $ \delta:=\left\vert \phi _{-}-\phi _{+}\right\vert $, by \eqref{con-stat-decay}, one can find a constant $ c _{1}>0 $ depending on $ \rho _{+} $, $ a $ and $ b $ such that
\begin{align}\label{bar-rho-bounds}
\displaystyle c _{1}^{-1} \leq \bar{\rho}(x) \leq c _{1},
\end{align}
provided $ \delta $ is suitably small. Combining \eqref{bar-rho-bounds} with \eqref{samlness-sumup}, we get
\begin{gather}\label{rho-lower-bd}
\displaystyle \frac{1}{c} \leq \rho = \varphi _{x}+\bar{\rho} \leq c,
\end{gather}
for some constant $ c>0 $ depending on $ \rho _{+} $, $ a $ and $ b $, provided $ \varepsilon $ and $ \delta $ are small enough. The boundary condition \eqref{problem-linearized-simple-e} together with the equation \eqref{problem-linearized-simple-a} leads to the following boundary conditions on higher-order derivatives:
\begin{gather}\label{boundary-simple-case}
\displaystyle \left.( \varphi _{t}, \varphi _{tt},\Phi _{t},\Phi _{tt})\right.\vert _{x=0}=0,\ \ \left.\left( \left( p'(\bar{\rho})\varphi _{x} \right)_{x} +\mathcal{F}\right) \right\vert _{x=0}=0,\ \  \left.\left( \left( p'(\bar{\rho})\varphi _{x} \right)_{x} +\mathcal{F}\right)_{t} \right\vert _{x=0}=0.
\end{gather}
Moreover, the following Hardy inequality plays a key role in deriving the \emph{a priori} estimates.
\begin{lemma}
Let $k>0 $ be a constant, it holds that
\begin{gather}\label{Hardy}
\displaystyle \int _{\mathbb{R}_{+}}		{\mathop{\mathrm{e}}} ^{- k x}f ^{2}\mathrm{d}x \leq C _{k} \int _{\mathbb{R}_{+}}f _{x}^{2} \mathrm{d}x
\end{gather}
for any $ f \in H _{0}^{1}(\mathbb{R}_{+}) $, where $ C _{k}>0 $ is a constant depending on $ k $ but independent of $ f $.
\end{lemma}
\begin{proof}
 From Lemma 3.4 in \cite{Li-Wang-spike-layer}, we get
 \begin{gather}
\displaystyle \int _{\mathbb{R}_{+}} (1+ x)^{-2}f ^{2}\mathrm{d}x \leq 4\int _{\mathbb{R}_{+}}f _{x}^{2}\mathrm{d}x\nonumber
\end{gather}
for any $ f \in H _{0}^{1}(\mathbb{R}_{+})  $. This along with the basic fact $ 		{\mathop{\mathrm{e}}}^{-kx}(1+x) ^{2} \leq C _{k} $ with some positive constant $ C _{k}>0 $ for any $ x \in \mathbb{R}_{+} $, implies that
\begin{gather}
\displaystyle  \int _{\mathbb{R}_{+}}	{\mathop{\mathrm{e}}} ^{- k x}f ^{2}\mathrm{d}x \leq C _{k}\int _{\mathbb{R}_{+}}		 (1+ x)^{-2}f ^{2}\mathrm{d}x \leq C _{k}\int _{\mathbb{R}_{+}}f _{x}^{2}\mathrm{d}x.\nonumber
\end{gather}
We thus get \eqref{Hardy}.
\end{proof}
\vspace*{3mm}

\subsection{Energy estimates} 
\label{sub:the_em}
In this section, we will derive the some estimates for the solution $(\varphi,\Phi)$ of \eqref{problem-linearized-simple-a}--\eqref{problem-linearized-simple-e} under the \emph{a priori} assumption \eqref{aproiri-assumpt} by the method of energy estimates. The estimates for $ \psi $ follows from the fact $ \psi=-\varphi _{t} $.

We begin with the lower-order estimates.
\begin{lemma}\label{lem-energy-lever}
Let the assumptions in Proposition \ref{prop-key} hold. If $ \varepsilon $ and $ \delta := \left\vert \phi _{-}- \phi _{+}\right\vert $ are sufficiently small, then the solution $ (\varphi,\Phi) $ of \eqref{problem-linearized-simple-a}--\eqref{problem-linearized-simple-e} satisfies
 \begin{align}\label{con-lem-basic}
    &\displaystyle \|(\varphi,\Phi)\|_{1}^{2}+\|\varphi _{t}\|^{2}
     +\int _{0}^{t}\|(\varphi _{x},\varphi _{\tau}, \Phi, \Phi _{x}, \Phi _{\tau})\|^{2}\mathrm{d}\tau\leq C (\|(\varphi _{0},\Phi _{0})\|_{1}^{2}+\|\psi _{0}\|^{2})
    \end{align}
    for any $ t \in (0,T) $, where $ C>0 $ is a constant independent of $ T $.
\end{lemma}
 \begin{proof}
Multiplying \eqref{problem-linearized-simple-a} by $ \varphi $ and integrating the resulting equation over $ \mathbb{R}_{+} $, we get
  \begin{align}\label{first-di-ine}
  \displaystyle  & \frac{\mathrm{d}}{\mathrm{d}t} \int _{\mathbb{R}_{+}} \left( \varphi \varphi_{t}+\frac{\alpha}{2} \varphi^{2} \right)  \mathrm{d} x+\int _{\mathbb{R}_{+}} p^{\prime}(\bar{\rho}) \varphi_{x}^{2} \mathrm{d} x
   \nonumber \\
     &\displaystyle ~=  \|\varphi_{t}\|^{2}  +  \int _{\mathbb{R}_{+}} \mathcal{F} _{1}\varphi \mathrm{d} x+  \int _{\mathbb{R}_{+}} \mathcal{F}_{2} \varphi \mathrm{d} x+\int _{\mathbb{R}_{+}}\mathcal{H} \varphi \mathrm{d}x.
          \end{align}
   By the Taylor expansion, we get
  \begin{gather*}
  \displaystyle  p(\varphi _{x}+\bar{\rho})-p(\bar{\rho})-p'(\bar{\rho})\varphi _{x} = p ^{\prime \prime}(\bar{\rho}+\vartheta _{1}\varphi _{x})\varphi _{x}^{2}
  \end{gather*}
  for some $ \vartheta _{1} \in (0,1) $. Then it follows from \eqref{p-condition}, \eqref{F-1-2}, \eqref{samlness-sumup} and \eqref{bar-rho-bounds} that
  \begin{gather}\label{fir-f1}
  \displaystyle \int _{\mathbb{R}_{+}} \mathcal{F} _{1}\varphi \mathrm{d} x=- \int _{\mathbb{R}_{+}}\left[ p(\varphi _{x}+\bar{\rho})-p(\bar{\rho})-p'(\bar{\rho})\varphi _{x} \right]\varphi _{x}  \mathrm{d}x \leq C \varepsilon \|\varphi _{x}\|^{2},
  \end{gather}
  provided $ \varepsilon $ and $ \delta $ are suitably small. Integrating by parts, we have
  \begin{align}\label{fir-f2-hardy-0}
  \displaystyle   \int _{\mathbb{R}_{+}} \mathcal{F}_{2} \varphi \mathrm{d} x&= -\mu\int _{\mathbb{R}_{+}} \varphi _{x}\Phi _{x}\varphi\mathrm{d}x-\mu\int _{\mathbb{R}_{+}}\varphi _{x}\bar{\phi}_{x}\varphi \mathrm{d}x-\mu \int _{\mathbb{R}_{+}}\bar{\rho}\Phi _{x} \varphi \mathrm{d}x
   \nonumber \\
   & \displaystyle =-\mu\int _{\mathbb{R}_{+}} \varphi _{x}\Phi _{x}\varphi\mathrm{d}x-\mu\int _{\mathbb{R}_{+}}\varphi _{x}\bar{\phi}_{x}\varphi \mathrm{d}x+ \mu \int _{\mathbb{R}_{+}}\bar{\rho}_{x}\Phi \varphi \mathrm{d}x+\mu \int _{\mathbb{R}_{+}}\bar{\rho}\Phi \varphi _{x}\mathrm{d}x.
  \end{align}
  In view of \eqref{samlness-sumup} and the Cauchy-Schwarz inequality, we deduce
  \begin{align}\label{fir-f2-hardy-1}
  \displaystyle  -\mu\int _{\mathbb{R}_{+}} \varphi _{x}\Phi _{x}\varphi\mathrm{d}x \leq C \|\varphi\|_{L ^{\infty}}\|\varphi _{x}\| \|\Phi _{x}\| \leq C \varepsilon \|(\varphi _{x},\Phi _{x})\|^{2}.
  \end{align}
  By the fact $ \left\vert (\bar{\rho}_{x},\bar{\phi}_{x})\right\vert \leq C\delta 		{\mathop{\mathrm{e}}}^{-\lambda x} $ from \eqref{con-stat-decay} and the Hardy inequality \eqref{Hardy}, it holds that
  \begin{align}\label{fir-f2-hardy-2}
  &\displaystyle  -\mu\int _{\mathbb{R}_{+}}\varphi _{x}\bar{\phi}_{x}\varphi \mathrm{d}x+ \mu \int _{\mathbb{R}_{+}}\bar{\rho}_{x}\Phi \varphi \mathrm{d}x
   \nonumber \\
   &~\displaystyle \leq C \delta \|\varphi _{x}\|\|    {\mathop{\mathrm{e}}}^{-\lambda x}\varphi\|+C \delta \|    {\mathop{\mathrm{e}}}^{- \frac{\lambda}{2}x}\varphi\|\|    {\mathop{\mathrm{e}}}^{- \frac{\lambda}{2}x}\Phi\|
    \nonumber \\
    &~\displaystyle \leq C \delta\|(\varphi _{x}, \Phi _{x})\|^{2},
  \end{align}
  where the Cauchy-Schwarz inequality has been used. Inserting \eqref{fir-f2-hardy-1} and \eqref{fir-f2-hardy-2} into \eqref{fir-f2-hardy-0} leads to
\begin{align}\label{fir-f2-hardy}
  \displaystyle  \int _{\mathbb{R}_{+}} \mathcal{F}_{2} \varphi \mathrm{d} x
    & \leq C \left( \delta+ \varepsilon \right)\|(\varphi _{x}, \Phi _{x})\|^{2}+\mu \int _{\mathbb{R}_{+}}\bar{\rho}\Phi \varphi _{x} \mathrm{d}x.
  \end{align}
    For the last term on the right-hand side of \eqref{first-di-ine}, from \eqref{F-defi-simple}, \eqref{samlness-sumup}, \eqref{rho-lower-bd}, integration by parts and  Cauchy-Schwarz inequality, we have
  \begin{align}\label{fir-f3}
  \displaystyle \int _{\mathbb{R}_{+}}\mathcal{H}\varphi \mathrm{d}x &=- \int _{\mathbb{R}_{+}}\frac{\varphi _{t} ^{2}}{\varphi _{x}+\bar{\rho}}  \varphi _{x}  \mathrm{d}x \leq C\int _{\mathbb{R}_{+}}\varphi _{t} ^{2}\left\vert \varphi _{x}\right\vert  \mathrm{d}x \leq C \|\varphi _{t}\|_{L ^{\infty}}\|\varphi _{t}\|\|\varphi _{x}\|\leq C \varepsilon \|( \varphi _{t}, \varphi _{x}) \|^{2}.
  \end{align}
  Substituting \eqref{fir-f1}, \eqref{fir-f2-hardy} and \eqref{fir-f3} into \eqref{first-di-ine}, we get
  \begin{align}\label{esti-vif}
  &\displaystyle  \displaystyle  \frac{\mathrm{d}}{\mathrm{d}t} \int _{\mathbb{R}_{+}}\left( \varphi \varphi_{t}+\frac{\alpha}{2} \varphi^{2} \right)  \mathrm{d} x+\int _{\mathbb{R}_{+}} p^{\prime}(\bar{\rho}) \varphi_{x}^{2} \mathrm{d} x
   \nonumber \\
   &~\displaystyle\leq C \left( \delta+ \varepsilon \right)\|(\varphi _{x}, \Phi _{x})\|^{2}+C\|\varphi _{t}\|^{2}+\mu \int _{\mathbb{R}_{+}}\bar{\rho}\Phi \varphi _{x} \mathrm{d}x.
  \end{align}
  Multiplying \eqref{problem-linearized-simple-c} by $ \frac{\mu}{a}\bar{\rho}\Phi $ and integrating the resulting equation over $ \mathbb{R}_{+} $, one has
      \begin{align*}
    \displaystyle & \frac{1}{2}\frac{\mathrm{d}}{\mathrm{d}t}\int _{\mathbb{R}_{+}} \frac{\mu}{a}\bar{\rho}\Phi ^{2} \mathrm{d}x+\frac{\mu}{a}\int _{\mathbb{R}_{+}}\left(  b \bar{\rho} \Phi ^{2}+\bar{\rho}\Phi _{x}^{2} \right)   \mathrm{d}x
     \nonumber \\
     & \displaystyle ~ =-\int _{\mathbb{R}_{+}}\frac{\mu }{a}\bar{\rho} _{x}\Phi \Phi _{x} \mathrm{d}x+\mu\int _{\mathbb{R}_{+}}\bar{\rho}\Phi  \varphi_{x} \mathrm{d}x,
    \end{align*}
    where, due to the fact $ \left\vert \bar{\rho}_{x}\right\vert\leq C\delta     {\mathop{\mathrm{e}}}^{- \lambda x} $ from \eqref{con-stat-decay} and the Hardy inequality \eqref{Hardy}, the following inequality holds
\begin{align}
\displaystyle  -\int _{\mathbb{R}_{+}}\frac{\mu }{a}\bar{\rho} _{x}\Phi \Phi _{x} \mathrm{d}x \leq C \delta \|\Phi _{x}\|\|    {\mathop{\mathrm{e}}}^{-\lambda x}\Phi\| \leq C \delta \|\Phi _{x}\|^{2}.\nonumber
\end{align}
Therefore,
\begin{align}\label{esti-fip}
     \frac{1}{2}\frac{\mathrm{d}}{\mathrm{d}t}\int _{\mathbb{R}_{+}} \frac{\mu}{a}\bar{\rho}\Phi ^{2} \mathrm{d}x+\frac{\mu}{a}\int _{\mathbb{R}_{+}}\left( b \bar{\rho} \Phi ^{2}+\bar{\rho}\Phi _{x}^{2} \right)  \mathrm{d}x
     \leq \mu\int _{\mathbb{R}_{+}}\bar{\rho}\Phi  \varphi _{x} \mathrm{d}x+C\delta\|\Phi _{x}\|^{2}.
    \end{align}
    Combining \eqref{esti-fip} with \eqref{esti-vif}, we obtain
    \begin{align}\label{fip-esti-di-eq}
    &\displaystyle
\displaystyle  \frac{\mathrm{d}}{\mathrm{d}t} \int _{\mathbb{R}_{+}}\left(\varphi \varphi_{t}+\frac{\alpha}{2} \varphi^{2}+\frac{\mu}{2a}\bar{\rho}\Phi ^{2} \right)\mathrm{d} x+\int _{\mathbb{R}_{+}} \left[ \frac{\mu}{2a}\bar{\rho}\Phi _{x}^{2} +\left( p^{\prime}(\bar{\rho}) \varphi_{x}^{2}-2 \mu \bar{\rho}\Phi \varphi _{x} +\frac{\mu b}{a}\bar{\rho}\Phi ^{2}\right)  \right]  \mathrm{d} x
 \nonumber \\
 &~\displaystyle \leq  C \left( \delta+ \varepsilon \right)\|\varphi _{x}\|^{2} +C\|\varphi _{t}\|^{2}
    \end{align}
    for suitably small $ \varepsilon $ and $ \delta $. By \eqref{p-condition} and \eqref{bar-rho-bounds}, we have
    \begin{gather*}
    \displaystyle  p^{\prime}(\bar{\rho}) \varphi_{x}^{2}-2 \mu \bar{\rho}\Phi \varphi _{x} +\frac{\mu b}{a}\bar{\rho}\Phi ^{2} \geq C \left( \varphi _{x}^{2}+\Phi ^{2} \right)
    \end{gather*}
    for some constant $ C>0 $ independent of $ t $. Then, for sufficiently small $ \varepsilon $ and $ \delta $, we have from \eqref{fip-esti-di-eq} that
    \begin{gather}\label{fip-fida}
    \displaystyle  \frac{\mathrm{d}}{\mathrm{d}t} \int _{\mathbb{R}_{+}}\left( \varphi \varphi_{t}+\frac{\alpha}{2} \varphi^{2}+\frac{\mu}{a}\bar{\rho}\Phi ^{2}  \right)  \mathrm{d} x+C\|(\varphi _{x}, \Phi, \Phi _{x})\|^{2} \leq C _{1} \|\varphi _{t}\|^{2}
    \end{gather}
    for some constant $ C _{1}>0 $, where we have used \eqref{bar-rho-bounds}. Now let us turn to the estimate for $ \varphi _{t} $. Multiplying \eqref{problem-linearized-simple-a} by $ \varphi _{t} $ and integrating the resulting equation over $ \mathbb{R}_{+} $, we get
    \begin{align}\label{vfi-t-ine-fir}
    &\displaystyle  \frac{1}{2} \frac{\mathrm{d}}{\mathrm{d}t}\int _{\mathbb{R}_{+}}\left[ \varphi_{t}^{2}+p^{\prime}(\bar{\rho}) \varphi_{x}^{2}\right] \mathrm{d} x+\int _{\mathbb{R}_{+}} \alpha \varphi_{t}^{2} \mathrm{d} x =\int _{\mathbb{R}_{+}}\mathcal{F}_{1} \varphi_{t} \mathrm{d}x +\int _{\mathbb{R}_{+}}\mathcal{F}_{2} \varphi_{t} \mathrm{d}x +\int _{\mathbb{R}_{+}}\mathcal{H} \varphi_{t} \mathrm{d}x.
         \end{align}
         Next, we estimate the terms on the right-hand side of \eqref{vfi-t-ine-fir}. First, it follows from a direct computation that
         \begin{align}
         \displaystyle \int _{\mathbb{R}_{+}}\mathcal{F}_{1}\varphi _{t} \mathrm{d}x&=- \int _{\mathbb{R}_{+}}\left[ p(\varphi _{x}+\bar{\rho})-p(\bar{\rho})-p'(\bar{\rho})\varphi _{x}  \right]  \varphi _{xt} \mathrm{d}x
          \nonumber \\
          &\displaystyle =- \frac{\mathrm{d}}{\mathrm{d}t} \int _{\mathbb{R}_{+}}\left( \int _{\bar{\rho}}^{\bar{\rho}+\varphi _{x}}p(s) \mathrm{d}y- p(\bar{\rho})\varphi _{x}- \frac{1}{2}p'(\bar{\rho})\varphi _{x}^{2}\right)  \mathrm{d}x.\nonumber
         \end{align}
Second, due to \eqref{samlness-sumup}, Cauchy-Schwarz inequality and the fact $ \left\vert (\bar{\rho}_{x},\bar{\phi}_{x})\right\vert \leq \delta 		{\mathop{\mathrm{e}}}^{-\lambda x} $ from \eqref{con-stat-decay}, we have
\begin{align}\label{zero-f-2-vfi-t}
\displaystyle \int _{\mathbb{R}_{+}}\mathcal{F}_{2}\varphi _{t} \mathrm{d}x& =- \mu\int _{\mathbb{R}_{+}} \varphi _{x}\Phi _{x}\varphi _{t}\mathrm{d}x- \mu\int _{\mathbb{R}_{+}}\varphi _{x}\bar{\phi}_{x}\varphi _{t} \mathrm{d}x+ \mu \int _{\mathbb{R}_{+}}\bar{\rho} _{x}\Phi\varphi _{t} \mathrm{d}x+\mu \int _{\mathbb{R}_{+}}\bar{\rho}\Phi \varphi _{xt} \mathrm{d}x,
 \nonumber \\
 & \displaystyle \leq C \|\varphi _{x}\|_{L ^{\infty}}\|\Phi _{x}\|\|\varphi _{t}\|+\|\bar{\phi}_{x}\|_{L ^{\infty}} \|\varphi _{x}\|\|\varphi _{t}\| + C \delta \|		{\mathop{\mathrm{e}}}^{-\lambda x}\Phi\|\|\varphi _{t}\|
  \nonumber \\
  & \displaystyle \quad +\frac{\mathrm{d}}{\mathrm{d}t}\int _{\mathbb{R}_{+}}\mu \bar{\rho}\Phi \varphi _{x} \mathrm{d}x- \mu \int _{\mathbb{R}_{+}}\bar{\rho}\Phi _{t}\varphi _{x} \mathrm{d}x
  \nonumber \\
   & \displaystyle \leq C(\varepsilon+ \delta)\|(\varphi _{x},\varphi _{t},\Phi _{x})\|^{2}+
   \frac{\mathrm{d}}{\mathrm{d}t}\int _{\mathbb{R}_{+}}\mu \bar{\rho}\Phi \varphi _{x} \mathrm{d}x
    - \mu \int _{\mathbb{R}_{+}}\bar{\rho}\Phi _{t}\varphi _{x} \mathrm{d}x.
\end{align}
where we have used the fact $  \|		{\mathop{\mathrm{e}}}^{-\lambda x}\Phi\|^{2} \leq C \|\Phi _{x}\| ^{2} $ by the Hardy inequality \eqref{Hardy}. Finally, using \eqref{samlness-sumup}, \eqref{rho-lower-bd} and integration by parts, one has
\begin{align}\label{zero-h-f-t}
\displaystyle \int _{\mathbb{R}_{+}}\mathcal{H}\varphi _{t} \mathrm{d}x &=- \int _{\mathbb{R}_{+}} \frac{\varphi _{t} ^{2}}{\varphi _{x}+\bar{\rho}}  \varphi _{xt} \mathrm{d}x \leq C \|\varphi _{xt}\|_{L  ^{\infty}}\|\varphi _{t}\|^{2} \leq C \varepsilon \|\varphi _{t} \|^{2}.
\end{align}
Hence, for suitably small $ \delta $ and $ \varepsilon $, we find from \eqref{vfi-t-ine-fir} that
    \begin{align}\label{var-t}
    &\displaystyle  \frac{1}{2} \frac{\mathrm{d}}{\mathrm{d}t}\int _{\mathbb{R}_{+}}\left[ \varphi_{t}^{2}+ p^{\prime}(\bar{\rho}) \varphi_{x}^{2}\right] \mathrm{d} x+\frac{\alpha}{2}\int _{\mathbb{R}_{+}}  \varphi_{t}^{2} \mathrm{d} x
     \nonumber \\
      &~ \leq  \frac{\mathrm{d}}{\mathrm{d}t}\int _{\mathbb{R}_{+}}\mu \bar{\rho}\Phi \varphi _{x} \mathrm{d}x
      - \frac{\mathrm{d}}{\mathrm{d}t} \int _{\mathbb{R}_{+}}\left( \int _{\bar{\rho}}^{\bar{\rho}+\varphi _{x}}p(s) \mathrm{d}y- p(\bar{\rho})\varphi _{x}- \frac{1}{2}p'(\bar{\rho})\varphi _{x}^{2}\right) \mathrm{d}x
  \nonumber \\
   &~ \quad \displaystyle - \mu \int _{\mathbb{R}_{+}}\bar{\rho}\Phi _{t}\varphi _{x} \mathrm{d}x +
C(\delta+\varepsilon) \|(\varphi _{x},\Phi _{x}
)\|^{2},
    \end{align}
    where the terms $ \|\varphi _{t}\|^{2} $ on the right-hand side of \eqref{zero-f-2-vfi-t} and \eqref{zero-h-f-t} have been absorbed. To proceed, we multiply \eqref{problem-linearized-simple-c} by $ \frac{\mu}{a}\bar{\rho}\Phi _{t} $ and integrate the resulting equation over $ \mathbb{R}_{+} $ to get
        \begin{align}\label{zero-fida-dif-ine}
    \displaystyle & \frac{\mu}{2a}\frac{\mathrm{d}}{\mathrm{d}t}\int _{\mathbb{R}_{+}} \left( b \bar{\rho} \Phi ^{2}+\bar{\rho}\Phi _{x}^{2}\right) \mathrm{d}x+\frac{\mu}{a}\int _{\mathbb{R}_{+}} \bar{\rho} \Phi _{t} ^{2}  \mathrm{d}x
     \nonumber \\
     & \displaystyle ~ =\int _{\mathbb{R}_{+}}\frac{\mu }{a}\bar{\rho} _{x}\Phi _{t} \Phi _{x} \mathrm{d}x+\mu\int _{\mathbb{R}_{+}}\bar{\rho}\Phi _{t} \varphi _{x} \mathrm{d}x
      \nonumber \\
      & \displaystyle~\leq \mu\int _{\mathbb{R}_{+}}\bar{\rho}\Phi _{t} \varphi _{x} \mathrm{d}x+\delta\|(\Phi _{t},\Phi _{x})\|^{2},
    \end{align}
    where \eqref{con-stat-decay} and the Cauchy-Schwarz inequality have been used. Combining \eqref{zero-fida-dif-ine} with \eqref{var-t} gives
    \begin{align}\label{vfi-t-fida-t}
    &\displaystyle  \frac{1}{2}\frac{\mathrm{d}}{\mathrm{d}t} \int _{\mathbb{R}_{+}}\left( \frac{b\mu}{a} \bar{\rho} \Phi ^{2}- 2 \mu \bar{\rho}\Phi \varphi _{x}+ p^{\prime}(\bar{\rho}) \varphi_{x}^{2}+\frac{\mu}{a}\bar{\rho}\Phi _{x}^{2}+ \varphi_{t}^{2}\right) \mathrm{d}x+  \frac{\alpha}{2}\|\varphi _{t}\|^{2}+\frac{\mu}{a}\int _{\mathbb{R}_{+}} \bar{\rho} \Phi _{t} ^{2}  \mathrm{d}x
     \nonumber \\
     &~ \displaystyle \leq  - \frac{\mathrm{d}}{\mathrm{d}t} \int _{\mathbb{R}_{+}}\left( \int _{\bar{\rho}}^{\bar{\rho}+\varphi _{x}}p(s) \mathrm{d}y- p(\bar{\rho})\varphi _{x}- \frac{1}{2}p'(\bar{\rho})\varphi _{x}^{2}\right) \mathrm{d}x  +
C(\delta+\varepsilon) \|(\varphi _{x}, \Phi _{x},\Phi _{t})\|^{2},
    \end{align}
 where
\begin{gather*}
\displaystyle   \frac{b\mu}{a} \bar{\rho} \Phi ^{2}- 2 \mu \bar{\rho}\Phi \varphi _{x}+ p^{\prime}(\bar{\rho}) \varphi_{x}^{2} \sim \Phi ^{2}+\varphi _{x}^{2},
\end{gather*}
due to \eqref{p-condition} and \eqref{bar-rho-bounds}. Given any constant $ K _{0}>0 $, adding \eqref{fip-fida} with \eqref{vfi-t-fida-t} multiplied by $ K _{0} $ leads to
 \begin{align}\label{zero-sumup}
    &\displaystyle  \frac{1}{2}\frac{\mathrm{d}}{\mathrm{d}t} \int _{\mathbb{R}_{+}}\left[\alpha\varphi ^{2}+ 2\varphi \varphi_{t}+ K _{0}\varphi_{t}^{2}+ K _{0}\Big( \frac{b\mu}{a} \bar{\rho} \Phi ^{2}- 2 \mu \bar{\rho}\Phi \varphi _{x}+ p^{\prime}(\bar{\rho}) \varphi_{x}^{2}+\frac{\mu}{a}\bar{\rho}\Phi _{x}^{2} \Big) \right] \mathrm{d}x
     \nonumber \\
     & ~\quad\displaystyle +C\|(\varphi _{x}, \Phi, \Phi _{x})\|^{2}+\left( \frac{ \alpha K _{0}}{2} -C _{1}\right)\|\varphi _{t}\|^{2}+ \frac{\mu K _{0}}{a}\int _{\mathbb{R}_{+}} \bar{\rho} \Phi _{t} ^{2}  \mathrm{d}x
     \nonumber \\
     &~ \displaystyle \leq - \frac{\mathrm{d}}{\mathrm{d}t} \int _{\mathbb{R}_{+}}K _{0}\left( \int _{\bar{\rho}}^{\bar{\rho}+\varphi _{x}}p(s) \mathrm{d}y- p(\bar{\rho})\varphi _{x}- \frac{1}{2}p'(\bar{\rho})\varphi _{x}^{2}\right) \mathrm{d}x+C K _{0}(\delta+\varepsilon) \|(\varphi _{x}, \Phi _{x},\Phi _{t})\|^{2},
    \end{align}
    where $ C _{1} $ is as in \eqref{fip-fida}. From \eqref{bar-rho-bounds}, it holds that
    \begin{gather*}
    \displaystyle  \frac{\mu K _{0}}{a}\int _{\mathbb{R}_{+}} \bar{\rho} \Phi _{t} ^{2}  \mathrm{d}x \leq C K _{0}\|\Phi _{t}\|^{2}
    \end{gather*}
    for some constant $ C>0 $ which depends on $ \rho _{+} $, $ \mu $ and $ a $. Taking $ K _{0} $ large enough such that  $ \frac{\alpha K _{0}}{2}>C _{1}$ and
\begin{gather*}
\displaystyle \frac{\alpha}{2}\varphi ^{2}+ \varphi \varphi_{t}+\frac{K _{0}}{2} \varphi _{t}^{2} \geq C \left( \varphi ^{2}+ \varphi _{t}^{2} \right)
\end{gather*}
for some constant $ C>0 $, then for suitably small $ \delta $ and $ \varepsilon $, we have from \eqref{zero-sumup} that
 \begin{align}\label{lem-suma}
    &\displaystyle  \frac{1}{2}\frac{\mathrm{d}}{\mathrm{d}t} \int _{\mathbb{R}_{+}}\left[\alpha\varphi ^{2}+ 2\varphi \varphi_{t}+ K _{0}\varphi_{t}^{2}+ K _{0}\Big( \frac{b\mu}{a} \bar{\rho} \Phi ^{2}- 2 \mu \bar{\rho}\Phi \varphi _{x}+ p^{\prime}(\bar{\rho}) \varphi_{x}^{2}+\frac{\mu}{a}\bar{\rho}\Phi _{x}^{2} \Big) \right] \mathrm{d}x
          \nonumber \\
     &~\displaystyle \quad+\frac{\mathrm{d}}{\mathrm{d}t} \int _{\mathbb{R}_{+}}\left( \int _{\bar{\rho}}^{\bar{\rho}+\varphi _{x}}p(s) \mathrm{d}y- p(\bar{\rho})\varphi _{x}- \frac{1}{2}p'(\bar{\rho})\varphi _{x}^{2}\right) \mathrm{d}x+C \left( \|\Phi\|_{1}^{2}+\|(\varphi _{x},\varphi _{t}, \Phi _{t})\|^{2}  \right)
      \leq 0 ,
    \end{align}
 where
    \begin{align*}
    \displaystyle \alpha\varphi ^{2}+ 2\varphi \varphi_{t}+ K _{0}\left(\varphi_{t}^{2}+ \frac{b\mu}{a} \bar{\rho} \Phi ^{2}- 2 \mu \bar{\rho}\Phi \varphi _{x}+ p^{\prime}(\bar{\rho}) \varphi_{x}^{2}+\frac{\mu}{a}\bar{\rho}\Phi _{x}^{2} \right)~\sim \varphi ^{2}+\varphi _{x} ^{2}+ \varphi _{t}^{2}+ \Phi ^{2}+ \Phi _{x}^{2}.
    \end{align*}
Applying the Taylor expansion to the function $h(s):= \int _{\bar{\rho}(x)}^{s}p(s)\mathrm{d}s $ along with \eqref{samlness-sumup} leads to
\begin{align}\label{Taylor-Exp}
  \displaystyle  \left\vert \int _{\bar{\rho}}^{\bar{\rho}+\varphi _{x}}p(s) \mathrm{d}y- p(\bar{\rho})\varphi _{x}- \frac{1}{2}p'(\bar{\rho})\varphi _{x}^{2}\right\vert=\frac{1}{6}\left\vert p''(\bar{\rho}+\vartheta _{2}\varphi _{x})\varphi _{x}^{3}\right\vert \leq C \varepsilon \varphi _{x}^{2}
  \end{align}
  for some constant $ \vartheta _{2} \in (0,1)  $. With \eqref{Taylor-Exp}, integrating \eqref{lem-suma} with respect to $ t $, by  taking $ \delta $ and $ \varepsilon $ suitably small, we get \eqref{con-lem-basic} and hence complete the proof.
 \end{proof}

 \begin{lemma}\label{lem-fip-xx}
Let the assumptions in Proposition \ref{prop-key} hold. If $ \varepsilon $ and $ \delta $ are sufficiently small, then for any $ t \in(0,T) $, the solution $ (\varphi,\Phi) $ of \eqref{problem-linearized-simple-a}--\eqref{problem-linearized-simple-e} satisfies
\begin{align}\label{con-fip-xx}
& \displaystyle \|(\varphi _{x},\Phi _{x})\|_{1}^{2}+\|\varphi _{xt}\|^{2}
 + \int _{0}^{t}\|(\varphi _{xx},\Phi _{x},\Phi _{xx},\varphi_{x \tau}, \Phi _{x \tau})\|^{2} \mathrm{d}\tau
  \nonumber \\
 & \displaystyle ~  \leq C (\|(\varphi _{0x},\Phi _{0x})\|_{1}^{2}+\|\psi _{0x}\|^{2})+  C\int _{0}^{t}\left( \|\varphi _{x}\| ^{2}+\|\varphi _{x}\|\|\Phi _{xx \tau} \| \right) \mathrm{d}\tau.
 \end{align}
 \end{lemma}
 \begin{proof}
    Differentiating \eqref{problem-linearized-simple-a}--\eqref{problem-linearized-simple-c} with respect to $ x $, we get
 \begin{subnumcases}
 \displaystyle \displaystyle \varphi _{xtt}+\alpha \varphi _{xt}-(p ^{\prime}(\bar{\rho})\varphi _{x})_{xx} =\mathcal{F} _{x}+ \mathcal{H}_{x},\label{Eq-x-simple-a}\\[1mm]
  \displaystyle   \displaystyle
  \Phi _{xt}=\Phi _{xxx}+a \varphi _{xx}-b \Phi _{x}.\label{Eq-x-simple-b}
 \end{subnumcases}
  Multiplying \eqref{Eq-x-simple-a} by $ \varphi _{x} $, and integrating it over $ \mathbb{R}_{+} $, we get, thanks to \eqref{boundary-simple-case}, that
 \begin{align}\label{vfi-x-fir-simple}
 &\displaystyle  \frac{\mathrm{d}}{\mathrm{d}t} \int _{\mathbb{R}_{+}}\left( \frac{\alpha}{2} \varphi _{x}^{2}+ \varphi _{x} \varphi _{xt} \right)  \mathrm{d}x +\int _{\mathbb{R}_{+}}p'(\bar{\rho}) \varphi _{xx}^{2} \mathrm{d}x - \int _{\mathbb{R}_{+}}\varphi _{xt}^{2} \mathrm{d}x
  \nonumber \\
  & ~\displaystyle =- \int _{\mathbb{R}_{+}}p''(\bar{\rho})\bar{\rho}_{x} \varphi _{x} \varphi _{xx} \mathrm{d}x -\int _{\mathbb{R}_{+}}\mathcal{F} _{1} \varphi _{xx} \mathrm{d}x-\int _{\mathbb{R}_{+}}\mathcal{F}_{2} \varphi _{xx} \mathrm{d}x
  +\int _{\mathbb{R}_{+}} \mathcal{H} _{x} \varphi _{x} \mathrm{d}x.
 \end{align}
Recalling the definitions of $ \mathcal{F}_{1} $ and $ \mathcal{F}_{2} $ in \eqref{F-1-2}, using \eqref{p-condition}, \eqref{bar-rho-bounds}, \eqref{samlness-sumup}, the fact $ \left\vert (\bar{\rho}_{x},\bar{\phi}_{x})\right\vert \leq C\delta 		{\mathop{\mathrm{e}}}^{-\lambda x} $ from \eqref{con-stat-decay} and Cauchy-Schwarz inequality, we have
  \begin{align}
 &\displaystyle - \int _{\mathbb{R}_{+}}p''(\bar{\rho})\bar{\rho}_{x} \varphi _{x} \varphi _{xx} \mathrm{d}x-\int _{\mathbb{R}_{+}}\mathcal{F} _{2} \varphi _{xx} \mathrm{d}x
  \nonumber \\
  & ~\displaystyle=  - \int _{\mathbb{R}_{+}}\left( p''(\bar{\rho})\bar{\rho}_{x}-\mu \bar{\phi}_{x} \right) \varphi _{x}\varphi  _{xx} \mathrm{d}x + \mu\int _{\mathbb{R}_{+}}\varphi _{x}\Phi _{x}\varphi _{xx} \mathrm{d}x+\mu\int _{\mathbb{R}_{+}}\bar{\rho}\Phi _{x}\varphi _{xx} \mathrm{d}x
   \nonumber \\
   &~\displaystyle \leq C \delta \|\varphi _{x}\|\|\varphi _{xx}\|+C \|\varphi _{x}\|\|\Phi _{x}\|_{L ^{\infty}}\|\varphi _{xx}\|+\mu \int _{\mathbb{R}_{+}}\bar{\rho}\Phi _{x}\varphi _{xx} \mathrm{d}x
    \nonumber \\
       &~ \displaystyle \leq C (\delta+\varepsilon) \|(\varphi _{x},\varphi _{xx})\|^{2}+\mu\int _{\mathbb{R}_{+}}\bar{\rho}\Phi _{x}\varphi _{xx} \mathrm{d}x
 \end{align}
 and
  \begin{align*}
 \displaystyle \int _{\mathbb{R}_{+}}\mathcal{F}_{1}\varphi _{xx} \mathrm{d}x
  &\displaystyle = -\int _{\mathbb{R}_{+}}
\int _{0}^{1}\int _{0}^{s}p'''(\bar{\rho}+\tau \varphi _{x})\mathrm{d}\tau \mathrm{d}s \varphi _{x}^{2}\bar{\rho}_{x}\varphi _{xx} \mathrm{d}x
 \nonumber \\
 & \displaystyle \quad- \int _{\mathbb{R}_{+}}\int _{0}^{1}p''(\bar{\rho}+s \varphi _{x}) \mathrm{d}s \varphi _{x}\varphi _{xx}^{2}  \mathrm{d}x
    \nonumber \\
    &\displaystyle \leq C \|\varphi _{x}\|_{L ^{\infty}}\|\varphi _{x}\| \|\varphi _{xx}\|+C \|\varphi _{x}\|_{L ^{\infty}}\|\varphi _{xx}\|^{2}
     \nonumber \\
         & \displaystyle \leq C \varepsilon \|(\varphi _{x},\varphi _{xx})\|^{2},
 \end{align*}
 where we have used the following identity
\begin{align}\label{cal-f-1-compu}
 \displaystyle \mathcal{F}_{1} &= [p'(\varphi _{x}+\bar{\rho})-p'(\bar{\rho})-p''(\bar{\rho})\varphi _{x}]\bar{\rho}_{x} +[p'(\varphi _{x}+\bar{\rho})-p'(\bar{\rho})]\varphi _{xx}
  \nonumber \\
  &\displaystyle = \int _{0}^{1}\int _{0}^{s}p'''(\bar{\rho}+\tau \varphi _{x})\mathrm{d}\tau \mathrm{d}s \varphi _{x}^{2}\bar{\rho}_{x}+\int _{0}^{1}p''(\bar{\rho}+s \varphi _{x}) \mathrm{d}s \varphi _{x}\varphi _{xx}
 \end{align}
 due to the Taylor expansion. For the last term on the right hand of \eqref{vfi-x-fir-simple}, thanks to \eqref{con-stat-decay}, \eqref{F-defi-simple}, \eqref{samlness-sumup}, \eqref{rho-lower-bd}, Cauchy-Schwarz inequality and the Hardy inequality \eqref{Hardy}, it holds for suitably small $ \varepsilon $ and $ \delta $ that
 \begin{align}\label{hx-esti}
 \displaystyle  \int _{\mathbb{R}_{+}}\mathcal{H}_{x}\varphi _{x} \mathrm{d}x&=-\int _{\mathbb{R}_{+}}\left( \frac{\varphi _{t} ^{2}}{\varphi _{x}+\bar{\rho}} \right)_{x}\varphi _{xx}  \mathrm{d}x
  \nonumber \\
  & \displaystyle =\int _{\mathbb{R}_{+}}\left( \frac{2 \varphi _{t}\varphi _{tx}}{\varphi _{x}+\bar{\rho}}-\frac{\varphi _{t}^{2}(\varphi _{xx}+\bar{\rho}_{x})}{(\varphi _{x}+\bar{\rho})^{2}} \right) \varphi _{xx} \mathrm{d}x
   \nonumber \\
   & \displaystyle \leq C \int _{\mathbb{R}_{+}}\left\vert \varphi _{t}\right\vert \left\vert \varphi _{tx}\right\vert \left\vert \varphi _{xx}\right\vert \mathrm{d}x+C \int _{\mathbb{R}_{+}}\varphi _{t}^{2}(\left\vert \varphi _{xx}\right\vert+\left\vert \bar{\rho}_{x}\right\vert)\left\vert \varphi _{xx}\right\vert \mathrm{d}x
    \nonumber \\
    &\displaystyle \leq C\|\varphi _t\|_{L ^{\infty}}\|\varphi _{tx}\| \|\varphi _{xx}\|+C\|\varphi _{t}\|_{L ^{\infty}}^{2}\|\varphi _{xx}\|^{2}+C \delta \|\varphi _{t}\|_{L ^{\infty
    }}\|    {\mathop{\mathrm{e}}}^{-\lambda x}\varphi _{t}\| \|\varphi _{xx}\|
     \nonumber \\
     & \leq\displaystyle C \varepsilon \|\varphi _{tx}\| \|\varphi _{xx}\|+C \varepsilon ^{2}\|\varphi _{xx}\|^{2}+C \varepsilon \delta \|\varphi _{tx}\|\|\varphi _{xx}\|
      \nonumber \\
           & \displaystyle \leq C(\varepsilon+\delta)\|(\varphi _{tx},\varphi _{xx})\|^{2}.
 \end{align}
 We thus conclude from \eqref{vfi-x-fir-simple}--\eqref{hx-esti} that
 \begin{align}\label{fi-x-di-eq-simple}
 &\displaystyle  \frac{\mathrm{d}}{\mathrm{d}t} \int _{\mathbb{R}_{+}}\left( \frac{\alpha}{2} \varphi _{x}^{2}+ \varphi _{x} \varphi _{xt} \right)  \mathrm{d}x +\int _{\mathbb{R}_{+}}p'(\bar{\rho}) \varphi _{xx}^{2} \mathrm{d}x
  \nonumber \\
  & ~ \displaystyle \leq C (\delta+\varepsilon) \|(\varphi _{x},\varphi _{xt},\varphi _{xx})\|^{2}+\mu\int _{\mathbb{R}_{+}}\bar{\rho}\Phi _{x}\varphi _{xx} \mathrm{d}x+ \left\|\varphi _{xt}\right\|^{2}.
 \end{align}
Multiplying \eqref{Eq-x-simple-b} by $ \frac{\mu}{a}\bar{\rho}\Phi _{x} $, and integrating it to get
 \begin{align}\label{fip-x}
 &\displaystyle \frac{\mu}{2a}\frac{\mathrm{d}}{\mathrm{d}t}\int _{\mathbb{R}_{+}}\bar{\rho}\Phi _{x}^{2} \mathrm{d}x +\frac{ \mu}{a}\int _{\mathbb{R}_{+}}\bar{\rho}\Phi _{xx}^{2} \mathrm{d}x+\frac{b \mu}{a}\int _{\mathbb{R}_{+}}\bar{\rho}\Phi _{x}^{2} \mathrm{d}x
  \nonumber \\
  & ~\displaystyle =  - \left.\frac{ \mu}{a}\bar{\rho}(0)\Phi _{xx}\Phi _{x}\right \vert _{x=0}-\int _{\mathbb{R}_{+}}\bar{\rho}_{x}\Phi _{xx}\Phi _{x} \mathrm{d}x +\mu\int _{\mathbb{R}_{+}}\bar{\rho}\varphi _{xx}\Phi _{x}  \mathrm{d}x
   \nonumber \\
   &~ \displaystyle \leq C \delta \|(\Phi _{x},\Phi _{xx})\|^{2}- \left.\frac{ \mu}{a}\rho _{-}\Phi _{xx}\Phi _{x}\right \vert _{x=0}+\mu\int _{\mathbb{R}_{+}}\bar{\rho}\Phi _{x} \varphi _{xx} \mathrm{d}x,
 \end{align}
 where \eqref{con-stat-decay}, \eqref{bar-rho-bounds} and Cauchy-Schwarz inequality have been used. Since $ \Phi =0 $ at $ x=0 $ and hence $ \Phi _{t}=0 $ at $ x=0 $, recalling \eqref{problem-linearized-simple-c}, we have
\begin{align}\label{fida-bdcon-xx}
\displaystyle  \Phi _{xx}=-a \varphi _{x}\  \mbox{at}\ x=0.
\end{align}
This along with the Sobolev inequality $ \|f\|_{L ^{\infty}}\leq C \|f\|^{\frac{1}{2}}\|f _{x}\| ^{\frac{1}{2}}$ and Young's inequality implies that
  \begin{align}\label{bd-fi-x-simp}
 \displaystyle - \left.\frac{ \mu}{a}\bar{\rho}(0)\Phi _{xx}\Phi _{x}\right \vert _{x=0} &= \left.\mu \bar{\rho}(0)\varphi _{x} \Phi_{x}\right \vert _{x=0} \leq C \|\varphi _{x}\|_{L ^{\infty}} \|\Phi _{x}\|_{L ^{\infty}}
   \nonumber \\
  &\leq C\|\varphi _{x}\|^{\frac{1}{2}}\|\varphi _{xx}\| ^{\frac{1}{2}}\|\Phi _{x}\|^{\frac{1}{2}}\|\Phi _{xx}\| ^{\frac{1}{2}}
   \nonumber \\
   & \leq C  \|\varphi _{x}\|^{\frac{1}{2}}\|\varphi _{xx}\| ^{\frac{1}{2}}(\|\Phi _{x}\|+\|\Phi _{xx}\|)
    \nonumber \\
         &\leq  \eta \|(\Phi _{x},\Phi _{xx},\varphi _{xx})\|^{2}+C _{\eta} \|\varphi _{x}\|^{2}
   \end{align}
   for any $ \eta>0 $. Substituting \eqref{bd-fi-x-simp} into \eqref{fip-x}, we get
\begin{align}\label{FIDAX-X-INEQ}
 &\displaystyle \frac{\mu}{2a}\frac{\mathrm{d}}{\mathrm{d}t}\int _{\mathbb{R}_{+}}\bar{\rho}\Phi _{x}^{2} \mathrm{d}x +\frac{ \mu}{a}\int _{\mathbb{R}_{+}}\bar{\rho}\Phi _{xx}^{2} \mathrm{d}x+\frac{b \mu}{a}\int _{\mathbb{R}_{+}}\bar{\rho}\Phi _{x}^{2} \mathrm{d}x
  \nonumber \\
  & ~ \displaystyle \leq C (\eta+ \delta) \|(\Phi _{x},\Phi _{xx})\|^{2}+ \eta \|\varphi _{xx}\|^{2}+C _{ \eta}\|\varphi _{x}\|^{2}+\mu\int _{\mathbb{R}_{+}}\bar{\rho}\varphi _{xx}\Phi _{x}  \mathrm{d}x.
\end{align}
Combining \eqref{FIDAX-X-INEQ} with \eqref{fi-x-di-eq-simple}, we get after taking $ \delta $, $ \varepsilon $ and $ \eta $ suitably small that
\begin{align}\label{vfi-x-fin-ineq}
\displaystyle  &\displaystyle  \frac{\mathrm{d}}{\mathrm{d}t} \int _{\mathbb{R}_{+}}\left( \frac{\alpha}{2} \varphi _{x}^{2}+ \varphi _{x} \varphi _{xt}+\frac{\mu}{2a}\bar{\rho}\Phi _{x}^{2} \right)  \mathrm{d}x +C\|(\varphi _{xx},\Phi _{x},\Phi _{xx})\|^{2}
 \nonumber \\
 & \displaystyle
\leq \left[ 1+C (\delta+\varepsilon) \right]  \|\varphi _{xt}\|^{2}+C \|\varphi _{x}\|^{2},
\end{align}
where we have used \eqref{bar-rho-bounds} and the following inequality
\begin{gather}\label{vfi-xx-equiv}
\displaystyle   p'(\bar{\rho}) \varphi _{xx}^{2}-2 \mu \bar{\rho}\varphi _{xx}\Phi _{x} +\frac{b \mu}{a}\bar{\rho}\Phi _{x}^{2} \geq C \left( \varphi _{xx}^{2}+ \Phi _{x}^{2}\right)
\end{gather}
due to \eqref{p-condition}. Next, we integrate \eqref{Eq-x-simple-a} multiplied by $ \varphi _{xt} $ over $ \mathbb{R}_{+} $ to get
 \begin{align}\label{vfi-xx-simple}
 &\displaystyle  \frac{1}{2}\frac{\mathrm{d}}{\mathrm{d}t}\int _{\mathbb{R}_{+}}(\varphi _{xt}^{2} +p'(\bar{\rho})\varphi _{xx}^{2})\mathrm{d}x+\alpha \|\varphi_{xt}\|^{2}
  \nonumber \\
  &~ \displaystyle=- \int _{\mathbb{R}_{+}}p''(\bar{\rho})\bar{\rho}_{x} \varphi _{x} \varphi _{xxt} \mathrm{d}x -\int _{\mathbb{R}_{+}}\mathcal{F} _{1}\varphi _{xxt} \mathrm{d}x-\int _{\mathbb{R}_{+}}\mathcal{F}_{2}\varphi _{xxt} \mathrm{d}x+\int _{\mathbb{R}_{+}}\mathcal{H} _{x} \varphi _{xt} \mathrm{d}x.
 \end{align}
A direct computation along with \eqref{con-stat-decay} and Cauchy-Schwarz inequality gives
\begin{align}\label{vfi-xt-firs}
\displaystyle  - \int _{\mathbb{R}_{+}}p''(\bar{\rho})\bar{\rho}_{x} \varphi _{x} \varphi _{xxt} \mathrm{d}x&=- \frac{\mathrm{d}}{\mathrm{d}t}\int _{\mathbb{R}_{+}}p''(\bar{\rho})\bar{\rho}_{x} \varphi _{x} \varphi _{xx} \mathrm{d}x+\int _{\mathbb{R}_{+}}p''(\bar{\rho})\bar{\rho}_{x} \varphi _{xt} \varphi _{xx} \mathrm{d}x
 \nonumber \\
 & \displaystyle \leq - \frac{\mathrm{d}}{\mathrm{d}t}\int _{\mathbb{R}_{+}}p''(\bar{\rho})\bar{\rho}_{x} \varphi _{x} \varphi _{xx} \mathrm{d}x+ C \delta \|(\varphi _{xt}, \varphi _{xx})\|^{2}.
\end{align}
Recalling \eqref{cal-f-1-compu}, we arrive at
 \begin{align}\label{f1-vxxt-comp}
 \displaystyle -\int _{\mathbb{R}_{+}}\mathcal{F} _{1}\varphi _{xxt} \mathrm{d}x
   &\displaystyle =-\frac{1}{2}\frac{\mathrm{d}}{\mathrm{d}t}\int _{\mathbb{R}_{+}}[p'(\varphi _{x}+\bar{\rho})-p'(\bar{\rho})]\varphi _{xx}^{2} \mathrm{d}x +\frac{1}{2}\int _{\mathbb{R}_{+}}p^{\prime \prime}(\varphi _{x}+\bar{\rho})\varphi _{xt} \varphi_{x x}^{2} \mathrm{d}x
    \nonumber \\
    &~\displaystyle \quad-  \frac{\mathrm{d}}{\mathrm{d}t}\int _{\mathbb{R}_{+}}[p'(\varphi _{x}+\bar{\rho})-p'(\bar{\rho})-p''(\bar{\rho})\varphi _{x}]\bar{\rho}_{x} \varphi _{xx}\mathrm{d}x
     \nonumber \\
      &~\displaystyle \quad + \int _{\mathbb{R}_{+}} \left[ p''(\varphi _{x}+\bar{\rho})-p''(\bar{\rho})  \right] \varphi _{xt}\varphi _{xx}\bar{\rho}_{x}  \mathrm{d}x.
         \end{align}
      From \eqref{p-condition}, \eqref{con-stat-decay}, \eqref{samlness-sumup} and \eqref{rho-lower-bd}, it holds that
\begin{align*}
&\displaystyle  \frac{1}{2}\int _{\mathbb{R}_{+}}p^{\prime \prime}(\varphi _{x}+\bar{\rho})\varphi _{xt} \varphi_{x x}^{2} \mathrm{d}x + \int _{\mathbb{R}_{+}} \left[  p''(\varphi _{x}+\bar{\rho})-p''(\bar{\rho}) \right] \varphi _{xt}\varphi _{xx}\bar{\rho}_{x}  \mathrm{d}x
 \nonumber \\
 &~\displaystyle \leq C \|\varphi _{xx}\|_{L ^{\infty}}\|\varphi _{tx}\|\|\varphi _{xx}\|+C \|\bar{\rho}_{x}\|_{L ^{\infty}}\left\|\varphi _{xt}\right\|\|\varphi _{xx}\|
  \nonumber \\
  &~\displaystyle \leq C(\delta+\varepsilon)\|(\varphi _{xx},\varphi _{xt})\|^{2}.
\end{align*}
Therefore, we have
\begin{align}\label{vfi-xt-f1}
 \displaystyle -\int _{\mathbb{R}_{+}}\mathcal{F} _{1}\varphi _{xxt} \mathrm{d}x
   &\displaystyle \leq C(\delta+\varepsilon)\|(\varphi _{xx},\varphi _{xt})\|^{2}  -\frac{1}{2}\frac{\mathrm{d}}{\mathrm{d}t}\int _{\mathbb{R}_{+}}[p'(\varphi _{x}+\bar{\rho})-xp'(\bar{\rho})]\varphi _{xx}^{2} \mathrm{d}x
        \nonumber \\
        &~ \quad - \frac{\mathrm{d}}{\mathrm{d}t}\int _{\mathbb{R}_{+}}[p'(\varphi _{x}+\bar{\rho})-p'(\bar{\rho})-p''(\bar{\rho})\varphi _{x}]\bar{\rho}_{x} \varphi _{xx}\mathrm{d}x.
         \end{align}
Similar to \eqref{f1-vxxt-comp}--\eqref{vfi-xt-f1}, the third term on the right-hand side of \eqref{vfi-xx-simple} can be estimated as follows:
              \begin{align}\label{vfi-xt-f2}
 \displaystyle  -\int _{\mathbb{R}_{+}}\mathcal{F}_{2}\varphi _{xxt} \mathrm{d}x&= \frac{\mathrm{d}}{\mathrm{d}t} \int _{\mathbb{R}_{+}}\mu (\varphi _{x}\Phi _{x}+\varphi _{x}\bar{\phi}_{x}+\bar{\rho}\Phi _{x})\varphi _{xx} \mathrm{d}x-\mu \int _{\mathbb{R}_{+}}\bar{\rho}\Phi _{xt} \varphi _{xx} \mathrm{d}x
      \nonumber \\
  & \displaystyle \quad - \mu\int _{\mathbb{R}_{+}} \left( \varphi _{xt}\Phi _{x}+\varphi _{x}\Phi _{xt}+\varphi _{xt}\bar{\phi}_{x} \right)\varphi _{xx} \mathrm{d}x
   \nonumber \\
   & \displaystyle \leq \frac{\mathrm{d}}{\mathrm{d}t} \int _{\mathbb{R}_{+}}\mu (\varphi _{x}\Phi _{x}+\varphi _{x}\bar{\phi}_{x}+\bar{\rho}\Phi _{x})\varphi _{xx} \mathrm{d}x -\mu \int _{\mathbb{R}_{+}}\bar{\rho}\Phi _{xt} \varphi _{xx} \mathrm{d}x
    \nonumber \\
    & \displaystyle \quad +C \|\Phi _{x}\|_{L ^{\infty}}\|\varphi _{xt}\| \|\varphi _{xx}\|+C \|\varphi _{x}\|_{L ^{\infty}}\|\Phi _{xt}\|\|\varphi _{xx}\|+C \|\bar{\phi}_{x}\|_{L ^{\infty}}\|\varphi _{xt}\| \|\varphi _{xx}\|
     \nonumber \\
        & \displaystyle \leq \frac{\mathrm{d}}{\mathrm{d}t} \int _{\mathbb{R}_{+}}\mu (\varphi _{x}\Phi _{x}+\varphi _{x}\bar{\phi}_{x}+\bar{\rho}\Phi _{x})\varphi _{xx} \mathrm{d}x -\mu \int _{\mathbb{R}_{+}}\bar{\rho}\Phi _{xt} \varphi _{xx} \mathrm{d}x
    \nonumber \\
    & \displaystyle \quad +C(\varepsilon+\delta)\|(\varphi _{xx},\varphi _{xt},\Phi _{xt})\|^{2}.
 \end{align}
 Noticing
 \begin{align}
 \displaystyle \mathcal{H}=\left( \frac{\varphi _{t}^{2}}{\varphi _{x}+\bar{\rho}} \right)_{x}= \frac{2 \varphi _{t} \varphi _{tx}}{\varphi _{x}+\bar{\rho}}- \frac{\varphi _{t} ^{2}(\varphi _{xx}+\bar{\rho}_{x})}{(\varphi _{x}+\bar{\rho})^2},\nonumber
 \end{align}
 we get, thanks to integration by parts and the boundary condition $ \varphi _{t}=0 $ at $ x=0 $, that
 \begin{align}\label{hx-vfi-xt}
 \displaystyle  \int _{\mathbb{R}_{+}}\mathcal{H}_{x}\varphi _{xt} \mathrm{d}x&=- \int _{\mathbb{R}_{+}}\varphi _{xt}\left(  \frac{\varphi _{t} ^{2}(\varphi _{xx}+\bar{\rho}_{x})}{(\varphi _{x}+\bar{\rho})^2}\right) _{x} \mathrm{d}x +2\int _{\mathbb{R}_{+}}\varphi _{xt}\left( \frac{\varphi _{t} \varphi _{tx}}{\varphi _{x}+\bar{\rho}} \right)_{x}  \mathrm{d}x
  \nonumber \\
  & \displaystyle = \int _{\mathbb{R}_{+}}\left[ \left( \frac{\varphi _{xx}^{2}}{2} \right)_{t}+ \bar{\rho}_{x}\varphi _{xxt} \right] \frac{\varphi _{t}^{2}}{(\varphi _{x}+\bar{\rho})^{2}}  \mathrm{d}x +\int _{\mathbb{R}_{+}}\varphi _{xt}^{2}\left( \frac{\varphi _{t}}{\varphi _{x}+\bar{\rho}} \right)_{x}  \mathrm{d}x
  \nonumber \\
  &\displaystyle =\frac{\mathrm{d}}{\mathrm{d}t}\int _{\mathbb{R}_{+}}\frac{\varphi _{t}^{2}}{(\varphi _{x}+\bar{\rho})^{2}} \left( \frac{1}{2} \varphi _{xx}^{2}+\bar{\rho}_{x}\varphi _{xx}\right) \mathrm{d}x+\int _{\mathbb{R}_{+}}\varphi _{xt}^{2}\left( \frac{\varphi _{t}}{\varphi _{x}+\bar{\rho}} \right)_{x}  \mathrm{d}x
   \nonumber \\
   & \displaystyle \quad -\int _{\mathbb{R}_{+}}\left( \frac{1}{2} \varphi _{xx}^{2}+\bar{\rho}_{x}\varphi _{xx}\right)\left( \frac{\varphi _{t}^{2}}{(\varphi _{x}+\bar{\rho})^{2}} \right) _{t} \mathrm{d}x
    \nonumber \\
       &=\frac{\mathrm{d}}{\mathrm{d}t}\int _{\mathbb{R}_{+}}\frac{\varphi _{t}^{2}}{(\varphi _{x}+\bar{\rho})^{2}} \left( \frac{1}{2} \varphi _{xx}^{2}+\bar{\rho}_{x}\varphi _{xx}\right) \mathrm{d}x+\mathcal{D}  _{1}+ \mathcal{D} _{2}.
 \end{align}
 Next, we estimate $ \mathcal{D} _{1} $ and $ \mathcal{D} _{2} $. First, we utilize \eqref{con-stat-decay}, \eqref{samlness-sumup} and \eqref{rho-lower-bd} to get
 \begin{align}\label{FIR-d-1}
 \displaystyle  \mathcal{D} _{1}&= \int _{\mathbb{R}_{+}}\varphi _{xt}^{2}\left( \frac{\varphi _{xt}}{\varphi _{x}+\bar{\rho}} - \frac{\varphi _{t}(\varphi _{xx}+\bar{\rho}_{x})}{\left( \varphi _{x}+\bar{\rho} \right)^{2} }\right)  \mathrm{d}x
  \nonumber \\
  & \displaystyle \leq C \|\varphi _{xt}\|_{L ^{\infty}}\|\varphi _{xt}\|^{2}+C \|\varphi _{t}\|_{L ^{\infty}}( \|\varphi _{xx}\|_{L ^{\infty}}+\|\bar{\rho}_{x}\|_{L ^{\infty}})\|\varphi _{xt}\|^{2}
   \nonumber \\
   &\displaystyle \leq C \left( \varepsilon +\delta \right) \|\varphi _{xt}\|^{2},
 \end{align}
 provided $ \varepsilon $ and $ \delta $ are suitably small. For $ \mathcal{D}_{2} $, by using  \eqref{con-stat-decay}, \eqref{samlness-sumup}, \eqref{rho-lower-bd}, Cauchy-Schwarz inequality and the Hardy inequality \eqref{Hardy}, we obtain
 \begin{align}\label{FIR-d-2}
 \displaystyle \mathcal{D}_{2}& =-\int _{\mathbb{R}_{+}}\left( \frac{1}{2} \varphi _{xx}^{2}+\bar{\rho}_{x}\varphi _{xx}\right)\left( \frac{2\varphi _{t}\varphi _{tt}}{(\varphi _{x}+\bar{\rho})^{2}}-2  \frac{\varphi _{t}^{2}\varphi _{xt}}{(\varphi _{x}+\bar{\rho})^{3}} \right)\mathrm{d}x
  \nonumber \\
  & \displaystyle \leq  C \|\varphi _{xx}\|^{2}\left( \|\varphi _{t}\|_{L ^{\infty}}\|\varphi _{tt}\|_{L ^{\infty}}+\|\varphi _{t}\|_{L ^{\infty}}^{2}\|\varphi _{xt}\|_{L ^{\infty}} \right)
   \nonumber \\
   &\displaystyle \quad+C \delta\|\varphi _{xx}\|\|    {\mathop{\mathrm{e}}}^{-\lambda x}\varphi _{t}\|_{L ^{2}}\|\varphi _{tt}\|_{L ^{\infty}}+C \|\bar{\rho}\|_{L ^{\infty}}\|\varphi _{t}\|_{L ^{\infty}}^{2}\|\varphi _{xt}\|
   \nonumber \\
   & \displaystyle \leq C \left( \delta+ \varepsilon \right)\|(\varphi _{xt},\varphi _{xx})\|^{2},
 \end{align}
where we have used the following inequality
     \begin{gather}\label{varphi-tt}
     \displaystyle \|\varphi _{tt}\| _{L ^{\infty}}\leq C \|\left( p'(\bar{\rho})\varphi _{x} \right)_{x}\|_{L ^{\infty}}+ C \|\varphi _{t}\|_{L ^{\infty}}+\|\mathcal{F}_{1}\|_{L ^{\infty}}+ \|\mathcal{F}_{2}\|_{L ^{\infty}}+C \|\mathcal{H}\|_{L ^{\infty}} \leq C \varepsilon,
     \end{gather}
     due to \eqref{con-stat-decay}, \eqref{problem-linearized-simple-a}, \eqref{F-defi-simple}, \eqref{F-1-2}, \eqref{samlness-sumup} and \eqref{rho-lower-bd}. With \eqref{FIR-d-1} and \eqref{FIR-d-2}, we update \eqref{hx-vfi-xt} as
 \begin{align}\label{vfi-xt-h}
 \displaystyle \int _{\mathbb{R}_{+}}\mathcal{H}_{x}\varphi _{xt} \mathrm{d}x
      & \displaystyle \leq \frac{\mathrm{d}}{\mathrm{d}t}\int _{\mathbb{R}_{+}}\frac{\varphi _{t}^{2}}{(\varphi _{x}+\bar{\rho})^{2}} \left( \frac{1}{2} \varphi _{xx}^{2}+\bar{\rho}_{x}\varphi _{xx}\right) \mathrm{d}x+C \left( \delta+ \varepsilon \right)\|(\varphi _{xt},\varphi _{xx})\|^{2}.
     \end{align}
     Substituting \eqref{vfi-xt-firs}, \eqref{vfi-xt-f1}, \eqref{vfi-xt-f2} and \eqref{vfi-xt-h} into \eqref{vfi-xx-simple}, we get
     \begin{align}\label{vfi-xx-final}
      &\displaystyle  \frac{1}{2}\frac{\mathrm{d}}{\mathrm{d}t}\int _{\mathbb{R}_{+}}(\varphi _{xt}^{2} +p'(\bar{\rho})\varphi _{xx}^{2}-2 \mu\bar{\rho}\Phi _{x} \varphi _{xx})\mathrm{d}x+\alpha \|\varphi_{xt}\|^{2}
       \nonumber \\
       & ~\displaystyle \leq -\frac{1}{2}\frac{\mathrm{d}}{\mathrm{d}t}\int _{\mathbb{R}_{+}}[p'(\varphi _{x}+\bar{\rho})-p'(\bar{\rho})]\varphi _{xx}^{2} \mathrm{d}x
       - \frac{\mathrm{d}}{\mathrm{d}t}\int _{\mathbb{R}_{+}}[p'(\varphi _{x}+\bar{\rho})-p'(\bar{\rho})-p''(\bar{\rho})\varphi _{x}]\bar{\rho}_{x} \varphi _{xx}\mathrm{d}x
        \nonumber \\
        &~\displaystyle \quad +\frac{\mathrm{d}}{\mathrm{d}t} \int _{\mathbb{R}_{+}}\mu (\varphi _{x}\Phi _{x}+\varphi _{x}\bar{\phi}_{x})\varphi _{xx} \mathrm{d}x +
        \frac{\mathrm{d}}{\mathrm{d}t}\int _{\mathbb{R}_{+}}\frac{\varphi _{t}^{2}}{(\varphi _{x}+\bar{\rho})^{2}} \left( \frac{1}{2} \varphi _{xx}^{2}+\bar{\rho}_{x}\varphi _{xx}\right) \mathrm{d}x
         \nonumber \\
         &~ \displaystyle \quad+ C(\delta+ \varepsilon) \|( \varphi _{xt},\varphi _{xx},\Phi _{xt})\|^{2}-\mu \int _{\mathbb{R}_{+}}\bar{\rho}\Phi _{xt} \varphi _{xx} \mathrm{d}x .
      \end{align}
      Multiplying \eqref{Eq-x-simple-b} by $ \frac{\mu}{a}\bar{\rho}\Phi _{xt} $, and integrating it to get
   \begin{align}\label{fida-xx}
  & \displaystyle  \frac{\mu b}{2a}\frac{\mathrm{d}}{\mathrm{d}t}\int _{\mathbb{R}_{+}}\bar{\rho}\Phi  _{x}^{2} \mathrm{d}x+ \frac{  \mu}{2a}\frac{\mathrm{d}}{\mathrm{d}t}\int _{\mathbb{R}_{+}}\bar{\rho}\Phi _{xx}^{2} \mathrm{d}x+ \frac{\mu}{a}\int _{\mathbb{R}_{+}}\bar{\rho}\Phi _{xt}^{2} \mathrm{d}x
    \nonumber \\
    &~\displaystyle =- \left.\frac{ \mu}{a} \bar{\rho}\Phi _{xx}\Phi _{xt}\right \vert _{x=0}+\mu\int _{\mathbb{R}_{+}}\bar{\rho}\varphi _{xx}\Phi _{xt} \mathrm{d}x - \int _{\mathbb{R}_{+}}\Phi _{xx}\Phi _{xt}\bar{\rho}_{x} \mathrm{d}x,
   \end{align}
 where, in view of \eqref{con-stat-decay}, \eqref{fida-bdcon-xx}, the Sobolev inequality $ \|f\|_{L ^{\infty}}\leq C \|f\|^{\frac{1}{2}}\|f _{x}\| ^{\frac{1}{2}}$ and Cauchy-Schwarz inequality, the following inequalities hold:
 \begin{align*}
 \displaystyle  - \int _{\mathbb{R}_{+}}\Phi _{xx}\Phi _{xt}\bar{\rho}_{x} \mathrm{d}x &\leq C \|\bar{\rho}_{x}\|_{L ^{\infty}}\|\Phi _{xx}\|\|\Phi _{xt}\| \leq C \delta \|(\Phi _{xt},\Phi _{xx})\|^{2},
   \nonumber \\[1mm]
   \displaystyle - \left.\frac{ \mu}{a}\bar{\rho}\Phi _{xx}\Phi _{xt}\right \vert _{x=0} &\leq C \|\varphi _{x}\|_{L ^{\infty}} \|\Phi _{xt}\|_{L ^{\infty}}
   \nonumber \\
  &\leq C\|\varphi _{x}\|^{\frac{1}{2}}\|\varphi _{xx}\| ^{\frac{1}{2}}\|\Phi _{xt}\|^{\frac{1}{2}}\|\Phi _{xxt}\| ^{\frac{1}{2}}
   \nonumber \\
   & \leq C  \|\varphi _{x}\|^{\frac{1}{2}}\|\Phi _{xxt}\| ^{\frac{1}{2}}(\|\Phi _{xt}\|+\|\varphi _{xx}\|)
    \nonumber \\
    & \leq \eta \|(\Phi _{xt},\varphi _{xx})\|^{2}+C _{\eta} \|\varphi _{x}\|\|\Phi _{xxt}\|
 \end{align*}
 for any $ \eta>0 $. Then, combining \eqref{fida-xx} with \eqref{vfi-xx-final} gives
 \begin{align}\label{vfi-xt-final}
      &\displaystyle  \frac{1}{2}\frac{\mathrm{d}}{\mathrm{d}t}\int _{\mathbb{R}_{+}}\left( \varphi _{xt}^{2} +p'(\bar{\rho})\varphi _{xx}^{2}-2 \mu\bar{\rho}\Phi _{x} \varphi _{xx}+\frac{\mu b}{a}\bar{\rho}\Phi  _{x}^{2}+\frac{\mu}{a}\bar{\rho}\Phi _{xx}^{2} \right) \mathrm{d}x+ C \|\Phi _{xt}\|^{2}+ \alpha\|\varphi_{xt}\|^{2}
       \nonumber \\
            & ~\displaystyle \leq -\frac{1}{2}\frac{\mathrm{d}}{\mathrm{d}t}\int _{\mathbb{R}_{+}}[p'(\varphi _{x}+\bar{\rho})-p'(\bar{\rho})]\varphi _{xx}^{2} \mathrm{d}x
       +  \frac{\mathrm{d}}{\mathrm{d}t}\int _{\mathbb{R}_{+}}[p'(\varphi _{x}+\bar{\rho})-p'(\bar{\rho})-p''(\bar{\rho})\varphi _{x}]\bar{\rho}_{x} \varphi _{xx}\mathrm{d}x
        \nonumber \\
        &~\displaystyle \quad +\frac{\mathrm{d}}{\mathrm{d}t} \int _{\mathbb{R}_{+}}\mu (\varphi _{x}\Phi _{x}+\varphi _{x}\bar{\phi}_{x})\varphi _{xx} \mathrm{d}x +
        \frac{\mathrm{d}}{\mathrm{d}t}\int _{\mathbb{R}_{+}}\frac{\varphi _{t}^{2}}{(\varphi _{x}+\bar{\rho})^{2}} \left( \frac{1}{2} \varphi _{xx}^{2}+\bar{\rho}_{x}\varphi _{xx}\right) \mathrm{d}x
         \nonumber \\
         &~ \displaystyle \quad+  C(\delta+ \varepsilon) \|(\varphi _{xt},\varphi _{xx},\Phi _{xt})\|^{2}+ \eta \|(\Phi _{xt},\varphi _{xx})\|^{2}+C _{ \eta} \|\varphi _{x}\|\|\Phi _{xxt}\|
 \end{align}
 for any $ \eta>0 $, provided $ \varepsilon $ and $ \delta $ are suitably small, where we have used \eqref{bar-rho-bounds}.

Finally, similar to the proof of Lemma \ref{lem-energy-lever}, adding \eqref{vfi-x-fin-ineq} with \eqref{vfi-xt-final} multiplied by a constant $ K _{1}>\frac{2}{\alpha} $, it follows that
 \begin{align}\label{con-vfi-xx-ineq-diff}
      &\displaystyle  \frac{\mathrm{d}}{\mathrm{d}t}\mathcal{G}_{1}(t)
      +C\|(\varphi _{xx},\Phi _{x},\Phi _{xx}, \Phi _{xt})\|^{2} +\left( K _{1}\alpha-1 \right)\|\varphi _{xt}\| ^{2}
       \nonumber \\
       &~\displaystyle
       \leq C(\delta+ \varepsilon) \|(\varphi _{xt},\varphi _{xx},\Phi _{xt})\|^{2}+ \eta \|(\Phi _{xt},\varphi _{xx})\|^{2}+C _{\eta} \|\varphi _{x}\|\|\Phi _{xxt}\|+C _{\eta}\|\varphi _{x}\|^{2}
      ,
 \end{align}
 where $ K _{1}\alpha-1>1 $, $  \mathcal{G}_{1}(t) $ is given by
 \begin{align}\label{g-1}
 \displaystyle \mathcal{G}_{1}(t):&=\int _{\mathbb{R}_{+}}\left\{  \frac{K _{1}}{2}\varphi _{xt}^{2}+ \frac{\alpha}{2} \varphi _{x}^{2}+ \varphi _{x} \varphi _{xt}+\frac{\mu\bar{\rho}}{2a}\Phi _{x}^{2} +\frac{K _{1}}{2}\Big[p'(\bar{\rho})\varphi _{xx}^{2}-2 \mu\bar{\rho}\Phi _{x} \varphi _{xx}+\frac{\mu\bar{\rho}}{a}( b\Phi  _{x}^{2}+\Phi _{xx}^{2})\Big] \right\}\mathrm{d}x
  \nonumber \\
              & ~\displaystyle \quad +\frac{K _{1}}{2}\int _{\mathbb{R}_{+}}[p'(\varphi _{x}+\bar{\rho})-p'(\bar{\rho})]\varphi _{xx}^{2} \mathrm{d}x
       -  \int _{\mathbb{R}_{+}}K _{1}[p'(\varphi _{x}+\bar{\rho})-p'(\bar{\rho})-p''(\bar{\rho})\varphi _{x}]\bar{\rho}_{x} \varphi _{xx}\mathrm{d}x
        \nonumber \\
        &~\displaystyle \quad - \int _{\mathbb{R}_{+}}\mu K _{1}(\varphi _{x}\Phi _{x}+\varphi _{x}\bar{\phi}_{x})\varphi _{xx} \mathrm{d}x +K _{1}
        \int _{\mathbb{R}_{+}}\frac{\varphi _{t}^{2}}{(\varphi _{x}+\bar{\rho})^{2}} \left( \frac{1}{2} \varphi _{xx}^{2}+\bar{\rho}_{x}\varphi _{xx}\right) \mathrm{d}x
         \nonumber \\
         &\displaystyle =\mathcal{G}_{1,0}+\mathcal{G}_{1,1}+\mathcal{G}_{1,2}+ \mathcal{G} _{1,3}+\mathcal{G} _{1,4}.
 \end{align}
 Taking $ K _{1} $ large enough such that
 \begin{gather*}
 \displaystyle  \frac{K _{1}}{2}\varphi _{xt}^{2}+ \frac{\alpha}{2} \varphi _{x}^{2}+ \varphi _{x} \varphi _{xt} \geq C \left( \varphi _{xt}^{2}+\varphi _{x}^{2} \right)
 \end{gather*}
for some constant $ C>0 $ independent of $ t $, recalling \eqref{vfi-xx-equiv}, we have
\begin{align*}
\displaystyle  \mathcal{G}_{1,0} \sim \|(\varphi _{x},\varphi _{xx},\varphi _{xt},\Phi _{x},\Phi _{xx})\|^{2}.
\end{align*}
By \eqref{p-condition}, \eqref{con-stat-decay}, \eqref{samlness-sumup}, \eqref{rho-lower-bd}, Cauchy-Schwarz inequality and the Taylor expansion, we have
\begin{align}
\displaystyle \mathcal{G}_{1,1} &\leq  C \|\bar{\rho}_{x}\|_{L ^{\infty}}\|\varphi _{x}\|\|\varphi _{xx}\| \leq  C \delta \|(\varphi _{x},\varphi _{xx})\|^{2},
 \nonumber \\[4pt]
 \displaystyle  \mathcal{G}_{1,2}& \leq C \|\varphi _{xx}\|\|\bar{\rho}_{x}\|_{L ^{\infty}} \|\varphi _{x}\|^{2} \leq  C \delta \|(\varphi _{x},\varphi _{xx})\|^{2},
 \nonumber \\[4pt]
 \displaystyle \mathcal{G}_{1,3}& \leq C \|\varphi _{xx}\|\left( \|\varphi _{x}\|\|\Phi _{x}\|_{L ^{\infty}}+ \|\bar{\phi}_{x}\|_{L ^{\infty}}\|\varphi _{x}\| \right) \leq C  \left( \varepsilon+\delta \right)   \|(\varphi _{x},\varphi _{xx})\|^{2},
  \nonumber \\[4pt]
  \displaystyle \mathcal{G}_{1,4}& \leq  C \|\varphi _{xx}\|\left( \|\varphi _{x}\|\|\Phi _{x}\|_{L ^{\infty}}+ \|\bar{\phi}_{x}\|_{L ^{\infty}}\|\varphi _{x}\| \right) \leq C  \left( \varepsilon+\delta \right)   \|(\varphi _{x},\varphi _{xx})\|^{2}.
   \nonumber
     \end{align}
Therefore, for sufficiently small $ \varepsilon $ and $ \delta $, we have from \eqref{g-1} that
\begin{align}\label{g-1-equiv}
\displaystyle \mathcal{G}_{1}
   & \sim \|(\varphi _{x},\varphi _{xx},\varphi _{xt},\Phi _{x},\Phi _{xx})\|^{2}.
\end{align}
With \eqref{g-1-equiv}, after integrating \eqref{con-vfi-xx-ineq-diff} over $ (0,t) $ and taking $ \varepsilon $, $ \delta $, $ \eta $ sufficiently small, we get \eqref{con-fip-xx} and thus finish the proof of Lemma \ref{lem-fip-xx}.
  \end{proof}
  To close the \emph{a priori} assumption \eqref{aproiri-assumpt}, some higher-order estimates of solutions are needed. Let us begin with the estimates on $ (\varphi _{xxx},\varphi _{xxt})$.
\begin{lemma}\label{lem-vfi-xxx}
Let the assumptions in Proposition \ref{prop-key} hold. If $\varepsilon $ and $ \delta$ are sufficiently small, then the solution $ (\varphi,\Phi) $ of \eqref{problem-linearized-simple-a}--\eqref{problem-linearized-simple-e} satisfies
\begin{align}\label{con-vfi-xxx}
&\displaystyle \|(\varphi _{xxx},\varphi _{xxt})\|^{2}+ \int _{0}^{t}\|\varphi _{xx \tau}\|^{2}\mathrm{d}\tau
 \nonumber \\
 &~ \displaystyle \leq C (\delta+ \varepsilon)\int _{0}^{t}\|(\Phi _{xx \tau},\varphi _{xxx})\|^{2}\mathrm{d}\tau+
  C \int _{0}^{t}\left(  \|(\varphi _{x},\varphi _{\tau})\|_{1}^{2}+\|( \Phi _{x \tau},\Phi _{xx})\|^{2} +\|\Phi _{x \tau \tau}\|^{2}\right) \mathrm{d}\tau
  \nonumber \\
  &~ \quad +
   C \left( \|(\varphi _{0x},\psi _{0})\|_{2}^{2}+\|\Phi _{0xx}\|^{2} \right) +C \left( \|(\varphi _{x},\varphi _{\tau})\|_{1}^{2}+\|(\Phi _{xx},\Phi _{x \tau})\|^{2} \right)
\end{align}
for any $ t \in (0,T) $, where the constant $C>0  $ is independent of $ T $.
\end{lemma}
 \begin{proof}
 Multiplying \eqref{Eq-x-simple-a} by $ -(( p'(\bar{\rho})\varphi _{x})_{xx}+ \mathcal{F}_{x}) _{t}$ followed by an integration over $ \mathbb{R}_{+} $, we obtain
 \begin{align}\label{var-xxx-single}
 & \displaystyle \frac{1}{2}\frac{\mathrm{d}}{\mathrm{d}t}\int _{\mathbb{R}_{+}}p'(\bar{\rho})\varphi _{xxt}^{2} \mathrm{d}x+ \alpha \int _{\mathbb{R}_{+}}p'(\bar{\rho})\varphi _{xxt}^{2} \mathrm{d}x
  \nonumber \\
  & ~\displaystyle =- \frac{\mathrm{d}}{\mathrm{d}t}\int _{\mathbb{R}_{+}}(\varphi _{xxt}+\alpha \varphi _{xx})\mathcal{F}_{t} \mathrm{d}x -\frac{1}{2}\frac{\mathrm{d}}{\mathrm{d}t}\int _{\mathbb{R}_{+}} [( p'(\bar{\rho})\varphi _{x})_{xx}+\mathcal{F}_{x}] ^{2} \mathrm{d}x  \nonumber \\
   & ~\displaystyle \quad  - \frac{\mathrm{d}}{\mathrm{d}t} \int _{\mathbb{R}_{+}}\varphi _{xxt}p'(\bar{\rho})\bar{\rho}_{x}\varphi _{xt} \mathrm{d}x+
 \underbrace{ \int _{\mathbb{R}_{+}} \varphi _{xxt} p'(\bar{\rho})\bar{\rho}_{x}\left(\varphi _{xtt}-\varphi _{xt} \right)   \mathrm{d}x }_{\mathcal{I}_{1}}
    \nonumber \\
  &~ \displaystyle \quad\underbrace{+ \int _{\mathbb{R}_{+}}(\varphi _{xxt}+\alpha \varphi _{xx})\mathcal{F}_{tt} \mathrm{d}x}_{\mathcal{I}_{2}} \underbrace{- \int _{\mathbb{R}_{+}}\mathcal{H} _{x}(( p'(\bar{\rho})\varphi _{x})_{xx}+ \mathcal{F}_{x}) _{t}  \mathrm{d}x}_{\mathcal{I}_{3}}.
 \end{align}
 Denote
 \begin{align}
 \displaystyle \mathcal{G}_{2}(t):&= \frac{1}{2}\int _{\mathbb{R}_{+}}p'(\bar{\rho})\varphi _{xxt}^{2} \mathrm{d}x+\int _{\mathbb{R}_{+}}(\varphi _{xxt}+\alpha \varphi _{xx})\mathcal{F}_{t} \mathrm{d}x
  \nonumber \\
  &\displaystyle \quad+\frac{1}{2}\int _{\mathbb{R}_{+}} [( p'(\bar{\rho})\varphi _{x})_{xx}+\mathcal{F}_{x}] ^{2} \mathrm{d}x  + \int _{\mathbb{R}_{+}}\varphi _{xxt}p'(\bar{\rho})\bar{\rho}_{x}\varphi _{xt} \mathrm{d}x,\nonumber
 \end{align}
 then \eqref{var-xxx-single} can be rewritten as
 \begin{align}\label{g-2-eq-refor}
 \displaystyle \frac{\mathrm{d}}{\mathrm{d}t }\mathcal{G}_{2}(t)+ \alpha \int _{\mathbb{R}_{+}}p'(\bar{\rho})\varphi _{xxt}^{2} \mathrm{d}x=\mathcal{I}_{1}+\mathcal{I}_{2}+\mathcal{I}_{3}.
   \end{align}
Noting that
\begin{align}
\displaystyle   \mathcal{F}_{t}&= [p'(\varphi _{x}+\bar{\rho})-p'(\bar{\rho})]\varphi _{xxt}+[p''(\varphi _{x}+\bar{\rho})-p''(\bar{\rho})]\varphi _{xt}\bar{\rho}_{x}+p''(\varphi _{x}+\bar{\rho})\varphi _{xt}\varphi _{xx}
 \nonumber \\
 & \quad \displaystyle - \mu \left(  \varphi _{xt}\Phi _{x}+\varphi _{x}\Phi _{xt}+\varphi _{xt}\bar{\phi}_{x}+\bar{\rho}\Phi _{xt} \right), \nonumber
\end{align}
we utilize \eqref{p-condition}, \eqref{con-stat-decay}, \eqref{samlness-sumup}, \eqref{bar-rho-bounds} and the mean value theorem to get
\begin{align*}
\displaystyle \left\vert \mathcal{F}_{t}\right\vert \leq C \varepsilon \left\vert \varphi _{xxt}\right\vert+C(\varepsilon+\delta) \left( \left\vert \varphi _{xt}\right\vert+\left\vert \Phi _{xt}\right\vert\right) +C \left\vert \Phi _{xt}\right\vert \leq C \varepsilon \left\vert \varphi _{xxt}\right\vert+C\left( \left\vert \varphi _{xt}\right\vert+\left\vert \Phi _{xt}\right\vert \right).
\end{align*}
Then by \eqref{con-stat-decay} and Cauchy-Schwarz inequality, we have
\begin{align}\label{g-2-1}
&\displaystyle \left\vert  \int _{\mathbb{R}_{+}}\varphi _{xxt}p'(\bar{\rho})\bar{\rho}_{x}\varphi _{xt} \mathrm{d}x\right\vert+\left\vert \int _{\mathbb{R}_{+}}(\varphi _{xxt}+\alpha \varphi _{xx})\mathcal{F}_{t} \mathrm{d}x \right\vert
 \nonumber \\
 &~\displaystyle\leq C( \varepsilon + \delta)\left\|\varphi _{xxt}\right\|^{2}+C\|(\varphi _{xx},\varphi _{xt},\Phi _{xt})\|^{2}.
\end{align}
A direct computation leads to
 \begin{align}
 \displaystyle  \mathcal{F}_{x}&= [p'(\varphi _{x}+\bar{\rho})-p'(\bar{\rho})-p''(\bar{\rho})\varphi _{x}]\bar{\rho}_{xx} +[p'(\varphi _{x}+\bar{\rho})-p'(\bar{\rho})]\varphi _{xxx}
  \nonumber \\
  & \displaystyle \quad+[p''(\varphi _{x}+\bar{\rho})-p''(\bar{\rho})-p'''(\bar{\rho})\varphi _{x}]\bar{\rho}_{x}^{2}+2[p''(\varphi _{x}+\bar{\rho})-p''(\bar{\rho})]\varphi _{xx}\bar{\rho}_{x}
   \nonumber \\
   & \displaystyle \quad +p''(\varphi _{x}+\bar{\rho})\varphi _{xx}^{2}+\varphi _{xx}\Phi _{x}+\varphi _{x}\Phi _{xx}+\varphi _{xx}\bar{\phi}_{x}+\bar{\rho}\Phi _{xx},\label{f-1-x-com}
 \\
       \displaystyle (p'(\bar{\rho})\varphi _{x})_{xx}&=p'(\bar{\rho})\varphi _{xxx}+2p''(\bar{\rho})\bar{\rho}_{x}\varphi _{xx}+p''(\bar{\rho})\bar{\rho}_{xx}\varphi _{x}+p'''(\bar{\rho})\bar{\rho}_{x}^{2}\varphi _{x}.\label{p-xxx-compu}
 \end{align}
Combining the above identities with \eqref{p-condition}, \eqref{con-stat-decay}, \eqref{samlness-sumup}--\eqref{rho-lower-bd} and the Taylor expansion yields that
 \begin{gather}
 \displaystyle  \left\vert \mathcal{F}_{x}\right\vert \leq C \varepsilon \left\vert \varphi _{xxx}\right\vert+C (\varepsilon+\delta)\left( \left\vert \varphi _{xx}\right\vert+\left\vert \varphi _{x}\right\vert \right) +C \left\vert \Phi _{xx}\right\vert, \label{F-x-bdd}\\
 \displaystyle \left\vert (p'(\bar{\rho})\varphi _{x})_{xx}- p'(\bar{\rho})\varphi _{xxx}\right\vert \leq C \delta(\left\vert \varphi _{x}\right\vert+\left\vert \varphi _{xx}\right\vert). \label{p-xxx-bd}
 \end{gather}
 We thus deduce that
\begin{align}\label{vfi-xxx-1}
&\displaystyle  \int _{\mathbb{R}_{+}} [( p'(\bar{\rho})\varphi _{x})_{xx}+\mathcal{F}_{x}] ^{2} \mathrm{d}x=\int _{\mathbb{R}_{+}} \left\vert p'(\bar{\rho})\varphi _{xxx}+( p'(\bar{\rho})\varphi _{x})_{xx}-p'(\bar{\rho})\varphi _{xxx}+\mathcal{F}_{x}\right\vert ^{2}\mathrm{d}x
 \nonumber \\
  &~\displaystyle \geq \int _{\mathbb{R}_{+}}\left\vert p'(\bar{\rho})\varphi _{xxx}\right\vert ^{2}\mathrm{d}x -2 \int _{\mathbb{R}_{+}}\left\vert p'(\bar{\rho})\varphi _{xxx}\right\vert  \left(  \left\vert \mathcal{F}_{x}\right\vert+ \left\vert (p'(\bar{\rho})\varphi _{x})_{x}- p'(\bar{\rho})\varphi _{xxx}\right\vert\right) \mathrm{d}x
 \nonumber \\
&~\displaystyle \geq \frac{1}{2}\int _{\mathbb{R}_{+}}\left\vert p'(\bar{\rho})\varphi _{xxx}\right\vert ^{2} \mathrm{d}x-C \|(\varphi _{x},\varphi _{xx},\Phi _{xx})\|^{2}
\end{align}
for suitably small $ \varepsilon $ and $ \delta $, and that
\begin{gather}\label{g-2-3}
\displaystyle  \int _{\mathbb{R}_{+}} [( p'(\bar{\rho})\varphi _{x})_{xx}+\mathcal{F}_{x}] ^{2} \mathrm{d}x \leq C \left( \|\varphi _{x}\|_{2}^{2}+\|\Phi _{xx}\|^{2} \right).
\end{gather}
Here \eqref{p-condition}, \eqref{bar-rho-bounds} and Cauchy-Schwarz inequality have been used. Thanks to \eqref{g-2-1}, \eqref{vfi-xxx-1} and \eqref{g-2-3}, it follows that
\begin{align}\label{G-2bd}
\begin{cases}
  \mathcal{G}_{2}(t) \leq C \left( \|\varphi _{xt}\|_{1}^{2}+\|\varphi _{x}\|_{2}^{2}+\|(\Phi _{xx},\Phi _{xt})\|^{2} \right),
  \\
 \displaystyle  \mathcal{G}_{2}(t) \geq C\|(\varphi _{xxt},\varphi _{xxx})\|^{2}-C \left( \|\varphi _{x}\|_{1}^{2}+\|(\Phi _{xx},\Phi _{xt},\varphi _{xt})\|^{2} \right).
\end{cases}
 \end{align}
where we have used \eqref{p-condition} and the bounds of $ \bar{\rho} $. Now let us turn to the estimates of $ \mathcal{I}_{i} $. From \eqref{problem-linearized-simple-a}, we get
\begin{gather}\label{vfi-xtt-iden}
\displaystyle \varphi _{xtt}= \left( p'(\bar{\rho})\varphi _{x} \right)_{xx}-\alpha \varphi _{xt}+\mathcal{F}_{x}+ \mathcal{H}_{x}.
\end{gather}
A direct computation leads to
\begin{align}\label{h-x}
 \displaystyle \mathcal{H}_{x}&= \frac{2 \varphi _{tx}^{2}+ 2 \varphi _{t} \varphi _{txx}}{\varphi _{x}+\bar{\rho}}- \frac{4\varphi _{t} \varphi _{tx}(\varphi _{xx}+\bar{\rho}_{x})}{(\varphi _{x}+\bar{\rho})^2} - \frac{\varphi _{t}^{2}(\varphi _{xxx}+\bar{\rho}_{xx})}{(\varphi _{x}+\bar{\rho})^2}+ \frac{2\varphi _{t}^{2}(\varphi _{xx}+\bar{\rho}_{x})^{2}}{(\varphi _{x}+\bar{\rho})^3}.
 \end{align}
 This along with \eqref{p-condition}, \eqref{con-stat-decay}, \eqref{samlness-sumup} and \eqref{rho-lower-bd} implies that
\begin{gather}\label{cal-h-x-bd}
\displaystyle \left\vert \mathcal{H}_{x}\right\vert \leq C \varepsilon(\left\vert \varphi _{xxx}\right\vert+\left\vert \varphi _{xxt}\right\vert)+ C \left( \left\vert \varphi _{xt}\right\vert+\left\vert \varphi _{t}\right\vert \right)
\end{gather}
for sufficiently small $ \varepsilon $ and $ \delta $. Therefore, it holds that
\begin{align}\label{vfi-xtt-l2}
\displaystyle  \| \varphi _{xtt}\| \leq C \| \left( p'(\bar{\rho})\varphi _{x} \right)_{xx}\|+\|\alpha \varphi _{xt}\|+\|\mathcal{F}_{x}\|+ \|\mathcal{H}_{x}\|\leq C \left( \|(\varphi _{x},\varphi _{t})\|_{2}+\|\Phi _{xx}\| \right),
\end{align}
where we have used \eqref{p-condition}, \eqref{con-stat-decay}, \eqref{bar-rho-bounds}, \eqref{F-x-bdd} and \eqref{cal-h-x-bd}. Resorting to \eqref{con-stat-decay}, \eqref{vfi-xtt-l2} and Cauchy-Schwarz inequality, we get
\begin{align}
\displaystyle  \mathcal{I} _{1}\leq C \delta \|\varphi _{xxt}\|\|(\varphi _{xtt},\varphi _{xt})\| \leq C \delta \|(\varphi _{xxx},\varphi _{xxt})\|^{2}+C \left( \|(\varphi _{x},\varphi _{t})\|_{1}^{2}+\left\|\Phi _{xx}\right\|^{2} \right) .\nonumber
\end{align}
For $ \mathcal{I}_{2} $, a direct computation gives
\begin{align}
 \displaystyle\mathcal{F}_{tt}&=  [p'(\varphi _{x}+\bar{\rho})-p'(\bar{\rho})]\varphi _{xxtt}+2p''(\varphi _{x}+\bar{\rho})\varphi _{xt}\varphi _{xxt} +[p''(\varphi _{x}+\bar{\rho})-p''(\bar{\rho})]\varphi _{xtt}\bar{\rho}_{x}
 \nonumber \\
 & \displaystyle \quad +p'''(\varphi _{x}+\bar{\rho})\varphi _{xt}^{2}\bar{\rho}_{x}+p'''(\varphi _{x}+\bar{\rho})\varphi _{xt}^{2}\varphi _{xx}+p''(\varphi _{x}+ \bar{\rho})\varphi _{xtt}\varphi _{xx}
  \nonumber \\
  & \displaystyle \quad  - \mu \left(  \varphi _{xtt}\Phi _{x}+2\varphi _{xt}\Phi _{xt}+\varphi _{x}\Phi _{xtt}+\varphi _{xtt}\bar{\phi}_{x}+\bar{\rho}\Phi _{xtt} \right)
   \nonumber \\
   &\displaystyle =[p'(\varphi _{x}+\bar{\rho})-p'(\bar{\rho})]\varphi _{xxtt}+\mathcal{J},
    \nonumber
\end{align}
where, due to \eqref{p-condition}, \eqref{con-stat-decay} and \eqref{samlness-sumup}, $ \mathcal{J} $ can be estimated as follows
\begin{align}
 \displaystyle \left\vert \mathcal{J}\right\vert \leq C(\varepsilon+\delta)\left( \left\vert \varphi _{xtt}\right\vert+\left\vert \varphi _{xt}\right\vert +\left\vert \Phi _{xt}\right\vert\right)  +C \varepsilon \left\vert \varphi _{xxt}\right\vert+C \left\vert \Phi _{xtt}\right\vert.
  \nonumber
 \end{align}
Then, owing to \eqref{p-condition}, \eqref{samlness-sumup}--\eqref{rho-lower-bd}, \eqref{vfi-xtt-l2}, Cauchy-Schwarz inequality and the mean value theorem, we have
\begin{align}
\displaystyle \mathcal{I}_{2}&=\int _{\mathbb{R}_{+}}(\varphi _{xxt}+\alpha \varphi _{xx}) \left( [p'(\varphi _{x}+\bar{\rho})-p'(\bar{\rho})]\varphi _{xxtt}+\mathcal{J} \right)  \mathrm{d}x
 \nonumber \\
 &\displaystyle= \frac{\mathrm{d}}{\mathrm{d}t}\int _{\mathbb{R}_{+}}[p'(\varphi _{x}+\bar{\rho})-p'(\bar{\rho})]\Big( \frac{1}{2}\varphi _{xxt}^{2}+\varphi _{xx}\varphi _{xxt}\Big)  \mathrm{d}x -\alpha \int _{\mathbb{R}_{+}}[p'(\varphi _{x}+\bar{\rho})-p'(\bar{\rho})]\varphi _{xxt}^{2} \mathrm{d}x
 \nonumber \\
 &\displaystyle \quad-\int _{\mathbb{R}_{+}}p''(\varphi _{x}+\bar{\rho})\varphi _{xt} \Big( \frac{1}{2}\varphi _{xxt}^{2}+\varphi _{xx}\varphi _{xxt} \Big) \mathrm{d}x +\int _{\mathbb{R}_{+}}\left( \varphi _{xxt}+ \alpha \varphi _{xx} \right)\mathcal{J}  \mathrm{d}x
  \nonumber \\
  &\displaystyle  \leq\frac{\mathrm{d}}{\mathrm{d}t}\int _{\mathbb{R}_{+}}[p'(\varphi _{x}+\bar{\rho})-p'(\bar{\rho})]\Big( \frac{1}{2}\varphi _{xxt}^{2}+\varphi _{xx}\varphi _{xxt}\Big)  \mathrm{d}x +C \|\varphi _{x}\|_{L ^{\infty}}\|\varphi _{xxt}\|^{2}
   \nonumber \\
   &\displaystyle \quad+ C \left\|\varphi _{xt}\right\|_{L ^{\infty}} \left( \left\|\varphi _{xxt}\right\|^{2}+\left\|\varphi _{xx}\right\| \|\varphi _{xxt}\| \right)+C \|(\varphi _{xxt},\varphi _{xx})\|\|\mathcal{J}\|
    \nonumber \\
    &\displaystyle  \leq\frac{\mathrm{d}}{\mathrm{d}t}\int _{\mathbb{R}_{+}}[p'(\varphi _{x}+\bar{\rho})-p'(\bar{\rho})]\Big( \frac{1}{2}\varphi _{xxt}^{2}+\varphi _{xx}\varphi _{xxt}\Big)  \mathrm{d}x
    +C(\varepsilon+\delta)\|(\varphi _{xxt},\varphi _{xxx})\|^{2}
   \nonumber \\
   & \displaystyle \quad+ C \left( \|(\varphi _{x},\varphi _{t})\|_{1} ^{2}+\|(\Phi _{xt},\Phi _{xx})\| ^{2} \right) + \eta  \|\varphi _{xxt}\|^{2}+C _{\eta }\|\Phi _{xtt}\|^{2}
    \nonumber
\end{align}
for any $ \eta >0 $. To deal with $ \mathcal{I}_{3} $, we rearrange $ \mathcal{H}_{x} $ in \eqref{h-x} as follows
 \begin{align}\label{h-x-rearrang}
 \displaystyle \mathcal{ H}_{x}&
=\frac{ 2 \varphi _{t} \varphi _{txx}}{\varphi _{x}+\bar{\rho}}-\frac{\varphi _{t}^{2}(\varphi _{xxx}+\bar{\rho}_{xx})}{(\varphi _{x}+\bar{\rho})^{2}}+\left[ \frac{2 \varphi _{tx}^{2}}{\varphi _{x}+\bar{\rho}}- \frac{4\varphi _{t} \varphi _{tx}(\varphi _{xx}+\bar{\rho}_{x})}{(\varphi _{x}+\bar{\rho})^2} + \frac{2\varphi _{t}^{2}(\varphi _{xx}+\bar{\rho}_{x})^{2}}{(\varphi _{x}+\bar{\rho})^3} \right]
 \nonumber \\
 \displaystyle& =\frac{ 2 \varphi _{t} \varphi _{txx}}{\varphi _{x}+\bar{\rho}}-\frac{\varphi _{t}^{2}(\varphi _{xxx}+\bar{\rho}_{xx})}{(\varphi _{x}+\bar{\rho})^{2}}+\mathcal{L}_{1},
 \end{align}
 with
 \begin{align}\label{pa-x-L1}
 \displaystyle \left\vert \partial _{x} \mathcal{L}_{1}\right\vert\leq C (\varepsilon+\delta)\left\vert \varphi _{xxt}\right\vert+C \varepsilon \left\vert \varphi _{xxx}\right\vert+ C(\varepsilon+\delta)\left( \left\vert \varphi _{t}\right\vert +\left\vert \varphi _{xt}\right\vert \right)
 \end{align}
 for suitably small $ \varepsilon $ and $ \delta $, due to \eqref{con-stat-decay}, \eqref{samlness-sumup} and \eqref{rho-lower-bd}. Based on \eqref{h-x-rearrang}, we split $ \mathcal{I}_{3} $ into three parts:
 \begin{align*}
  \displaystyle \mathcal{I}_{3}&=\int _{\mathbb{R}_{+}}\frac{ 2 \varphi _{t} \varphi _{txx}}{\varphi _{x}+\bar{\rho}}(( p'(\bar{\rho})\varphi _{x})_{xx}+ \mathcal{F}_{x}) _{t}   \mathrm{d}x -\int _{\mathbb{R}_{+}}\frac{\varphi _{t}^{2}(\varphi _{xxx}+\bar{\rho}_{xx})}{(\varphi _{x}+\bar{\rho})^{2}}(( p'(\bar{\rho})\varphi _{x})_{xx}+ \mathcal{F}_{x}) _{t}  \mathrm{d}x
   \nonumber \\
   & \displaystyle \quad +\int _{\mathbb{R}_{+}}(( p'(\bar{\rho})\varphi _{x})_{xx}+ \mathcal{F}_{x}) _{t}\mathcal{L}_{1} \mathrm{d}x=\mathcal{I}_{3,1}+\mathcal{I}_{3,2}+\mathcal{I}_{3,3}.
  \end{align*}
  Next, we estimate $ \mathcal{I}_{3,i} $. Recalling \eqref{f-1-x-com} and \eqref{p-xxx-compu}, we have
  \begin{align}
  \displaystyle (( p'(\bar{\rho})\varphi _{x})_{xx}+ \mathcal{F}_{x}) _{t}& =p'(\varphi _{x}+\bar{\rho})\varphi _{xxxt}+\left[ (( p'(\bar{\rho})\varphi _{x})_{xx}+ \mathcal{F}_{x}) _{t} -p'(\varphi _{x}+\bar{\rho})\varphi _{xxxt}\right]
   \nonumber \\
   &\displaystyle=  p'(\varphi _{x}+\bar{\rho})\varphi _{xxxt}+\mathcal{L}_{2},
    \nonumber
  \end{align}
  where, in view of \eqref{p-condition}, \eqref{con-stat-decay},  \eqref{samlness-sumup} and \eqref{bar-rho-bounds}, the following inequality holds
  \begin{align*}
   \displaystyle \left\vert \mathcal{L}_{2}\right\vert \leq C(\varepsilon+\delta)(\left\vert \varphi _{xxt}\right\vert +\left\vert \varphi _{xxx}\right\vert)+C \left\vert \Phi _{xxt}\right\vert+C (\left\vert \varphi _{xt}\right\vert+\left\vert \Phi _{xt}\right\vert).
   \end{align*}
   We thus have
   \begin{align}
   \displaystyle \mathcal{I}_{3,1} &=\int _{\mathbb{R}_{+}}\frac{ 2 \varphi _{t} \varphi _{txx}}{\varphi _{x}+\bar{\rho}}\left(  p'(\varphi _{x}+\bar{\rho})\varphi _{xxxt}+\mathcal{L}_{2} \right)    \mathrm{d}x
    \nonumber \\
    & \displaystyle \leq- \int _{\mathbb{R}_{+}}\varphi _{xxt}^{2}\left( \frac{ p'(\varphi _{x}+\bar{\rho})\varphi _{t}}{\varphi _{x}+\bar{\rho}} \right)_{x}  \mathrm{d}x+C \|\varphi _{t}\|_{L ^{\infty}}\|\varphi _{xxt}\|\|\mathcal{L}_{2}\|
     \nonumber \\
     &\displaystyle  \leq \int _{\mathbb{R}_{+}}\varphi _{xxt}^{2}\left( \left\vert \varphi _{xt}\right\vert +\left\vert \varphi _{t}(\varphi _{xx}+\bar{\rho}_{x})\right\vert\right)  \mathrm{d}x +C \varepsilon \|\varphi _{xxt}\|\|\mathcal{L}_{2}\|
      \nonumber \\
           &\displaystyle \leq C (\delta+ \varepsilon)\|(\varphi _{xxt},\varphi _{xxx},\varphi _{tx})\|^{2}+C \varepsilon \|\Phi _{xt}\|_{1}^{2} \nonumber
   \end{align}
   and
   \begin{align}
   \displaystyle  \mathcal{I}_{3,2}&=-\int _{\mathbb{R}_{+}}\frac{\varphi _{t}^{2}(\varphi _{xxx}+\bar{\rho}_{xx})}{(\varphi _{x}+\bar{\rho})^{2}}\left(  p'(\varphi _{x}+\bar{\rho})\varphi _{xxxt}+\mathcal{L}_{2}  \right)  \mathrm{d}x
    \nonumber \\
    & \displaystyle = -\frac{1}{2}\frac{\mathrm{d}}{\mathrm{d}t}\int _{\mathbb{R}_{+}}\frac{p'(\varphi _{x}+\bar{\rho})\varphi _{t}^{2}(\varphi _{xxx}+\bar{\rho}_{xx})^{2}}{(\varphi _{x}+\bar{\rho})^{2}} \mathrm{d}x
    + C (\varepsilon+\delta) \left( \|\varphi _{xxx}\|+\|\varphi _{t}\| \right) \|\mathcal{L}_{2}\|
     \nonumber \\
     &\displaystyle \quad+\frac{1}{2}\int _{\mathbb{R}_{+}}(\varphi _{xxx}+\bar{\rho}_{xx})^{2}\left( \frac{p''(\varphi _{x}+\bar{\rho})\varphi _{xt}\varphi _{t}^{2}}{(\varphi _{x}+\bar{\rho})^{2}}+ \frac{2p'(\varphi _{x}+\bar{\rho})\varphi _{t}\varphi _{tt}}{(\varphi _{x}+\bar{\rho})^{2}}-2  \frac{p'(\varphi _{x}+\bar{\rho})\varphi _{xt}\varphi _{t}^{2}}{(\varphi _{x}+\bar{\rho})^{3}} \right) \mathrm{d}x
     \nonumber \\
     & \displaystyle =-\frac{1}{2}\frac{\mathrm{d}}{\mathrm{d}t}\int _{\mathbb{R}_{+}}\frac{p'(\varphi _{x}+\bar{\rho})\varphi _{t}^{2}(\varphi _{xxx}+\bar{\rho}_{xx})^{2}}{(\varphi _{x}+\bar{\rho})^{2}} \mathrm{d}x+C (\varepsilon+\delta) \left( \|\varphi _{t}\| _{1}^{2}+\|\Phi _{xt}\|_{1}^{2}+\|\varphi _{xxx}\|^{2} \right), \nonumber
   \end{align}
   where we have used \eqref{p-condition}, \eqref{con-stat-decay}, \eqref{samlness-sumup}, \eqref{rho-lower-bd}, \eqref{varphi-tt}, Cauchy-Schwarz inequality and the integration by parts. For $ \mathcal{I}_{3,3} $, the integration by parts leads to
   \begin{align}
   \displaystyle \mathcal{I}_{3,3}=-\int _{\mathbb{R}_{+}}(( p'(\bar{\rho})\varphi _{x})_{x}+ \mathcal{F}) _{t}\partial _{x}\mathcal{L}_{1} \mathrm{d}x,
    \nonumber
   \end{align}
   which combined with \eqref{F-defi-simple}, \eqref{F-1-2}, \eqref{samlness-sumup}, \eqref{pa-x-L1} and Cauchy-Schwarz inequality implies that
   \begin{align}
   \displaystyle  \mathcal{I}_{3,3}& \leq C \int _{\mathbb{R}_{+}}\left( \left\vert \varphi _{xt}\right\vert +\left\vert \Phi _{xt}\right\vert\right)\Big [ C (\varepsilon+\delta)(\left\vert \varphi _{xxt}\right\vert+\left\vert \varphi _{t}\right\vert +\left\vert \varphi _{xt}\right\vert )+C \varepsilon \left\vert \varphi _{xxx}\right\vert\Big]    \mathrm{d}x
    \nonumber \\
    &\leq \displaystyle C (\varepsilon+\delta)\|(\varphi _{xxx},\varphi _{xxt})\| ^{2}+C \left( \|\varphi _{t}\|_{1}^{2}+\|\Phi _{xt}\|^{2} \right).  \nonumber
   \end{align}
   Hence, we have
   \begin{align}
   \displaystyle \mathcal{I}_{3}&\leq  -\frac{1}{2}\frac{\mathrm{d}}{\mathrm{d}t}\int _{\mathbb{R}_{+}}\frac{p'(\varphi _{x}+\bar{\rho})\varphi _{t}^{2}(\varphi _{xxx}+\bar{\rho}_{xx})^{2}}{(\varphi _{x}+\bar{\rho})^{2}} \mathrm{d}x
    \nonumber \\
    &\displaystyle \quad+C (\varepsilon+\delta)\|(\varphi _{xxx},\varphi _{xxt},\Phi _{xxt})\| ^{2}+C \left( \|\varphi _{t}\|_{1}^{2}+\|\Phi _{xt}\|^{2} \right). \nonumber
   \end{align}
   Plugging $ \mathcal{I}_{i}~(i=1,2,3) $ into \eqref{g-2-eq-refor}, we now reach
\begin{align}\label{vfixxx-sumup}
 & \displaystyle \displaystyle \frac{\mathrm{d}}{\mathrm{d}t }\mathcal{G}_{2}(t)+ \alpha \int _{\mathbb{R}_{+}}p'(\bar{\rho})\varphi _{xxt}^{2} \mathrm{d}x+\frac{1}{2}\frac{\mathrm{d}}{\mathrm{d}t}\int _{\mathbb{R}_{+}}\frac{p'(\varphi _{x}+\bar{\rho})\varphi _{t}^{2}(\varphi _{xxx}+\bar{\rho}_{xx})^{2}}{(\varphi _{x}+\bar{\rho})^{2}} \mathrm{d}x
  \nonumber \\
  & \displaystyle =\frac{\mathrm{d}}{\mathrm{d}t}\int _{\mathbb{R}_{+}}[p'(\varphi _{x}+\bar{\rho})-p'(\bar{\rho})]\Big( \frac{1}{2}\varphi _{xxt}^{2}+\varphi _{xx}\varphi _{xxt}\Big)  \mathrm{d}x +C (\delta+ \varepsilon)\|(\varphi _{xxx},\varphi _{xxt},\Phi _{xxt})\|^{2}
      \nonumber \\
  & \displaystyle \quad+  \eta \|\varphi _{xxt}\|^{2}+ C _{ \eta} \|\Phi _{xtt}\|^{2}+ C \|(\varphi _{x},\varphi _{t})\|_{1}^{2}+C \|(\Phi _{xx},\Phi _{xt})\|^{2}
   \end{align}
   for any $  \eta>0 $. Consequently, thanks to \eqref{p-condition}, \eqref{samlness-sumup}--\eqref{rho-lower-bd}, \eqref{G-2bd} and the mean value theorem, we obtain \eqref{con-vfi-xxx} after integrating \eqref{vfixxx-sumup} over $ (0,t) $ and taking $ \delta $, $ \varepsilon$ and $ \eta $ small enough. The proof of Lemma \ref{lem-vfi-xxx} is complete.
\end{proof}
In the next lemma, we shall estimate the higher-order terms on the right-hand side of \eqref{con-vfi-xxx}.
 \begin{lemma}\label{lem-h-multi}
Let the assumptions in Proposition \ref{prop-key} hold. If $\varepsilon $ and $ \delta$ are sufficiently small, then the solution $ (\varphi,\Phi) $ of \eqref{problem-linearized-simple-a}--\eqref{problem-linearized-simple-e} satisfies
\begin{align}\label{con-h-multi-inte}
 & \displaystyle  \|(\Phi _{xt},\Phi _{tt},\Phi _{xxx})\|^{2}+\int _{0}^{t}\|\Phi _{xx \tau}\|^{2}\mathrm{d}\tau
   + \int _{0}^{t}\|(\Phi _{\tau \tau},\Phi _{x \tau}, \Phi _{x \tau \tau},\Phi _{xx \tau})\|^{2}\mathrm{d}\tau
    \nonumber \\
    &~\displaystyle\leq C(\|\Phi _{0}\|_{4}^{2}+\|\varphi _{0x}\|_{2}^{2}+\|\psi _{0x}\|^{2})+C\|(\varphi _{xx},\Phi _{x})\|^{2}+
    C  \int _{0}^{t}(\|(\varphi _{\tau},\varphi _{x})\|_{1}^{2}+ \|\Phi _{xx}\|^{2})\mathrm{d}\tau
   \nonumber \\
   & \displaystyle \quad+ C\int _{0}^{t}\left(  \varepsilon\|(\varphi _{xx \tau},\varphi _{xxx})\|^{2}+ \|\varphi _{xx \tau}\|\|\Phi  _{x \tau}\|\right)  \mathrm{d}\tau,
  \end{align}
    and that
 \begin{align}\label{con-fida-xtt}
 \displaystyle  \int _{0}^{t}\int _{\mathbb{R}_{+}}\varphi _{xxx}^{2} \mathrm{d}x \mathrm{d}\tau
   &\displaystyle \leq C \|(\varphi _{x},\varphi _{xt},\Phi _{x})\|_{1}^{2}+ C \left( \|\varphi _{0x}\|_{2}^{2}+\|\psi _{0x}\|^{2} \right)
 \nonumber \\
 & \displaystyle \quad +
 C(\delta+\varepsilon) \int _{0}^{t}\|\Phi _{xx \tau}\|^{2}\mathrm{d}\tau
  +C \int _{0}^{t}\left( \|\varphi _{\tau}\|_{2}^{2}+ \|(\Phi _{x \tau},\Phi _{xx})\|^{2}+\left\|\varphi _{x}\right\|_{1}^{2} \right)\mathrm{d}\tau.
 \end{align}

  \end{lemma}

\begin{proof}We divide the proof into two steps.

\vspace*{4mm}
\noindent{\bf Step 1:} \emph{Estimates on $ (\Phi _{xt},\Phi _{xxt},\Phi _{xtt}) $}. Differentiating \eqref{Eq-x-simple-b} with respect to $ t $, we have
 \begin{gather}\label{eq-fida-xt}
  \displaystyle \Phi _{xtt}=\Phi _{xxxt}+ a \varphi _{xxt}-b \Phi _{xt}.
  \end{gather}
  Multiplying \eqref{eq-fida-xt} by $ \Phi _{xt} $ and integrating the resulting equation over $ \mathbb{R}_{+} $, we get by the H\"older's inequality that
  \begin{align}\label{fida-xt-sigl-esti}
  &\displaystyle \frac{1}{2}\frac{\mathrm{d}}{\mathrm{d}t}\int _{\mathbb{R}_{+}}\Phi _{xt}^{2} \mathrm{d}x +b \int _{\mathbb{R}_{+}}\Phi _{xt}^{2} \mathrm{d}x+ \int _{\mathbb{R}_{+}}\Phi _{xxt}^{2} \mathrm{d}x
   \nonumber \\
   \displaystyle&~ =- \left. \Phi _{xxt}\Phi _{xt}\right \vert _{x=0}+ a\int _{\mathbb{R}_{+}}\varphi _{xxt}\Phi _{xt} \mathrm{d}x
   \nonumber \\
   \displaystyle&~\displaystyle \leq - \left. \Phi _{xxt}\Phi _{xt}\right \vert _{x=0}+C \|\Phi _{xt}\|\|\varphi _{xxt}\|.
  \end{align}
  In view of \eqref{fida-bdcon-xx}, for smooth solution $ (\varphi,\Phi) $, we have $ \Phi _{xxt}=a \varphi _{xt} $ at $ x=0 $. Then it holds that
  \begin{align}\label{fida-xxt-bdesti}
  \displaystyle  - \left. \Phi _{xxt}\Phi _{xt}\right \vert _{x=0}& \leq \|\varphi _{xt}\|_{L ^{\infty}}\|\Phi _{xt}\|_{L ^{\infty}}
   \nonumber \\
   & \displaystyle \leq C \|\varphi _{xt}\|^{\frac{1}{2}}\|\varphi _{xxt}\|^{\frac{1}{2}}\|\Phi _{xt}\|^{\frac{1}{2}}\|\Phi _{xxt}\|^{\frac{1}{2}}
    \nonumber \\
    & \displaystyle \leq C \|\varphi _{xxt}\|^{\frac{1}{2}}\|\Phi _{xt}\|^{\frac{1}{2}}\left( \|\varphi _{xt}\|+\left\|\Phi _{xxt}\right\| \right)
     \nonumber \\
     &\displaystyle \leq \eta  \|\Phi _{xxt}\|^{2}+C _{\eta } \|\Phi _{xt}\|\|\varphi _{xxt}\|  + C \|\varphi _{xt}\|^{2}
  \end{align}
  for any $ \eta >0 $, where the Sobolev inequality $ \|f\|_{L ^{\infty}}\leq C \|f\|^{\frac{1}{2}}\|f _{x}\|^{\frac{1}{2}} $ and Cauchy-Schwarz inequality have been used. Plugging \eqref{fida-xxt-bdesti} into \eqref{fida-xt-sigl-esti}, we get after taking $ \eta $ suitably small that
  \begin{align}\label{fida-xt-single}
&\displaystyle
\|\Phi _{xt}\|^{2}+\int _{0}^{t}\|(\Phi _{xx \tau}, \Phi _{x \tau})\|^{2}\mathrm{d}\tau
 \nonumber \\
  &~\displaystyle\leq C \int _{0}^{t}\left( \|\varphi _{xx \tau}\|\|\Phi  _{x \tau}\|+\|\varphi_{x \tau}\|^{2}\right) \mathrm{d}\tau+C(\|\Phi _{0}\|_{3}^{2}+\|\varphi _{0xx}\|^{2}),
\end{align}
where we have used \eqref{fida-t-initial}. By \eqref{Eq-x-simple-b} and \eqref{fida-xt-single}, we have
  \begin{align}\label{fida-xxx}
  \displaystyle  \|\Phi _{xxx}\|^{2} &\leq  C \|\Phi _{xt}\|^{2}+ C\|(\varphi _{xx},\Phi _{x})\|^{2}
   \nonumber \\
   & \displaystyle \leq  C \int _{0}^{t}\left( \|\varphi _{xx \tau}\|\|\Phi  _{x \tau}\|+\|\varphi_{x \tau}\|^{2}\right) \mathrm{d}\tau+C\|(\varphi _{xx},\Phi _{x})\|^{2} +C(\|\Phi _{0}\|_{3}^{2}+\|\varphi _{0xx}\|^{2}).
  \end{align}
  Differentiating in \eqref{problem-linearized-simple-c} with respect to $ t $ twice gives
  \begin{gather}\label{EQ-fida-tt}
   \displaystyle \Phi _{ttt}=\Phi _{xxtt}+a \varphi _{xtt}-b \Phi _{tt}.
   \end{gather}
   Multiplying \eqref{EQ-fida-tt} by $ \Phi _{tt} $, we have
\begin{align}\label{esti-fidp-tt}
\displaystyle \frac{1}{2}\frac{\mathrm{d}}{\mathrm{d}t}\int _{\mathbb{R}_{+}}\Phi _{tt}^{2} \mathrm{d}x + b\int _{\mathbb{R}_{+}}\Phi _{tt}^{2} \mathrm{d}x+ \int _{\mathbb{R}_{+}}\Phi _{xtt}^{2} \mathrm{d}x =  a \int _{\mathbb{R}_{+}}\varphi _{xtt}\Phi _{tt} \mathrm{d}x.
  \end{align}
Thanks to \eqref{p-condition}, \eqref{con-stat-decay}, \eqref{bar-rho-bounds},  \eqref{F-x-bdd}, \eqref{p-xxx-bd}, \eqref{vfi-xtt-iden}, \eqref{cal-h-x-bd}, Cauchy-Schwarz inequality and integration by parts, one has
  \begin{align}\label{vfi-tt-error}
  &\displaystyle   a \int _{\mathbb{R}_{+}}\varphi _{xtt}\Phi _{tt} \mathrm{d}x
   \nonumber \\
    &~\displaystyle = a \int _{\mathbb{R}_{+}}p'(\bar{\rho})\varphi _{xxx}\Phi _{tt} \mathrm{d}x+a\int _{\mathbb{R}_{+}}\left[ \left( \left( p'(\bar{\rho})\varphi _{x} \right)_{xx}-p'(\bar{\rho})\varphi _{xxx}\right) -\varphi _{xt} +\mathcal{F}_{x}+\mathcal{H}_{x}\right] \Phi _{tt}\mathrm{d}x
   \nonumber \\
   &\displaystyle~\leq -a \int _{\mathbb{R}_{+}}p'(\bar{\rho})\varphi _{xx}\Phi _{xtt} \mathrm{d}x-a \int _{\mathbb{R}_{+}}p''(\bar{\rho})\bar{\rho}_{x}\varphi _{xx}\Phi _{tt} \mathrm{d}x+ C\|\varphi _{xt}\|\|\Phi _{tt}\|
    \nonumber \\
    &~\displaystyle \quad+C\|(\mathcal{F}_{x},\mathcal{H}_{x})\|\|\Phi _{tt}\|+C \|\left( p'(\bar{\rho})\varphi _{x} \right)_{xx}-p'(\bar{\rho})\varphi _{xxx}\|\|\Phi _{tt}\|
   \nonumber \\
      &\displaystyle ~\leq C (\eta +\varepsilon+\delta) \|\Phi _{tt}\|^{2}+\eta  \|\Phi _{xtt}\|^{2}+C \varepsilon\|(\varphi _{xxt},\varphi _{xxx})\|^{2}+
    C _{\eta } (\|(\varphi _{t},\varphi _{x})\|_{1}^{2}+ \|\Phi _{xx}\|^{2})
  \end{align}
  for any $ \eta>0 $. Substituting \eqref{vfi-tt-error} into \eqref{esti-fidp-tt}, we get after taking $ \eta $, $ \varepsilon $, $ \delta $ small enough that
  \begin{align}\label{fida-xtt-multi}
  &\displaystyle  \|\Phi _{tt}\|^{2}+ \int _{0}^{t}\|(\Phi _{\tau \tau}, \Phi _{x \tau \tau})\|^{2}\mathrm{d}\tau
   \nonumber \\
   &~\displaystyle\leq C(\|\Phi _{0}\|_{4}^{2}+\|\varphi _{0x}\|_{2}^{2}+\|\psi _{0x}\|^{2})+C \varepsilon \int _{0}^{t}\|(\varphi _{xx \tau},\varphi _{xxx})\|^{2} \mathrm{d}\tau
   \nonumber \\
   & ~\displaystyle \quad+
    C  \int _{0}^{t}(\|(\varphi _{\tau},\varphi _{x})\|_{1}^{2}+ \|\Phi _{xx}\|^{2})\mathrm{d}\tau,
  \end{align}
  where we have used \eqref{fida-tt-intial}. Combining \eqref{fida-xtt-multi} with \eqref{fida-xt-single} and \eqref{fida-xxx}, then we get \eqref{con-h-multi-inte}.

  \vspace*{4mm}
\noindent  {\bf Step 2:}\emph{ Estimates on $ \varphi _{xxx} $}. Multiplying \eqref{Eq-x-simple-a} by $ -( \varphi _{xxx}- \frac{\mu \rho _{-}}{p'(\rho _{-})} \Phi _{xx} )$ with $ \rho _{-} $ as in \eqref{rho-fu}, and integrating the resulting equation over $ \mathbb{R}_{+}\times(0,t) $ for any $ t \in (0,T) $, we get
  \begin{align}\label{var-xx-simple}
  &\displaystyle  \int _{0}^{t}\int _{\mathbb{R}_{+}}p'(\bar{\rho})\varphi _{xxx}^{2} \mathrm{d}x \mathrm{d}\tau
   \nonumber \\
   &~\displaystyle  =- \int _{0}^{t}\int _{\mathbb{R}_{+}}\left(  (p ^{\prime}(\bar{\rho})\varphi _{x})_{xx} -p'(\bar{\rho})\varphi _{xxx}\right)\varphi _{xxx}  \mathrm{d}x \mathrm{d}\tau+\int _{0}^{t}\int _{\mathbb{R}_{+}} (p ^{\prime}(\bar{\rho})\varphi _{x})_{xx} \frac{\mu \rho _{-}}{p'(\rho _{-})} \Phi _{xx} \mathrm{d}x \mathrm{d}\tau
    \nonumber \\
     &\displaystyle \quad -\int _{\mathbb{R}_{+}}(\varphi _{xt}+\varphi _{x}) \left( \varphi _{xxx}- \frac{\mu \rho _{-} \Phi _{xx} }{p'(\rho _{-})}\right)(x,0) \mathrm{d}x+\int _{\mathbb{R}_{+}}(\varphi _{xt}+\varphi _{x}) \left( \varphi _{xxx}- \frac{\mu \rho _{-} \Phi _{xx} }{p'(\rho _{-})}\right) (x,t)\mathrm{d}x
      \nonumber \\
      &~\displaystyle \quad
       -  \int _{0}^{t}\int _{\mathbb{R}_{+}} (\varphi _{xt}+\alpha\varphi _{x})\left( \varphi _{xxx}- \frac{\mu \rho _{-}\Phi _{xx}}{p'(\rho _{-})}  \right)_{\tau}  \mathrm{d}x \mathrm{d}\tau - \int _{0}^{t}\int _{\mathbb{R}_{+}}(\mathcal{H}_{x}+\mathcal{F}_{x})\left(  \varphi _{xxx}- \frac{\mu \rho _{-}\Phi _{xx}}{p'(\rho _{-})}  \right)  \mathrm{d}x \mathrm{d}\tau
        \nonumber \\
        &\displaystyle =\sum _{i=1}^{6}\mathcal{P}_{i}.
  \end{align}
  Thanks to \eqref{p-condition} and \eqref{bar-rho-bounds}, one can find a constant $ C>0 $ such that
  \begin{align}
  \displaystyle C   \int _{0}^{t}\int _{\mathbb{R}_{+}}\varphi _{xxx}^{2} \mathrm{d}x \mathrm{d}\tau \leq  \int _{0}^{t}\int _{\mathbb{R}_{+}}p'(\bar{\rho})\varphi _{xxx}^{2} \mathrm{d}x \mathrm{d}\tau .\nonumber
  \end{align}
  Now let us estimates $ \mathcal{P}_{i}~(1 \leq i \leq 6) $ term by term. Recalling \eqref{p-xxx-bd}, it holds that
  \begin{align}
  \displaystyle \mathcal{P}_{1}\leq C \delta \int _{0}^{t}\int _{\mathbb{R}_{+}}(\left\vert \varphi _{x}\right\vert+\left\vert \varphi _{xx}\right\vert) \left\vert \varphi _{xxx}\right\vert\mathrm{d}x \mathrm{d} \tau\leq C \delta \int _{0}^{t}\|\varphi _{x}\|_{2}^{2}\mathrm{d}\tau,\nonumber
  \end{align}
  and that
  \begin{align}
  \displaystyle  \mathcal{P}_{2}&= \int _{0}^{t}\int _{\mathbb{R}_{+}}\left[  (p ^{\prime}(\bar{\rho})\varphi _{x})_{xx} -p'(\bar{\rho})\varphi _{xxx} +p'(\bar{\rho})\varphi _{xxx}\right] \frac{\mu \rho _{-}}{p'(\rho _{-})} \Phi _{xx} \mathrm{d}x \mathrm{d}\tau
   \nonumber \\
   & \leq \int _{0}^{t}\int _{\mathbb{R}_{+}} \left( \delta(\left\vert \varphi _{x}\right\vert+\left\vert \varphi _{xx}\right\vert) +\left\vert \varphi _{xxx}\right\vert \right)\left\vert \Phi _{xx}\right\vert \mathrm{d}x \mathrm{d}\tau
    \nonumber \\
       & \displaystyle \leq  \int _{0}^{t} \|\varphi _{xxx}\|\|\Phi _{xx}\|  \mathrm{d}\tau +C \delta \int _{0}^{t}\|(\varphi _{x},\varphi _{xx}, \Phi _{xx})\|^{2}  \mathrm{d}\tau
        \nonumber \\
        &\displaystyle \leq \eta   \int _{0}^{t}\|\varphi _{xxx}\|^{2}\mathrm{d}\tau+C _{ \eta } \int _{0}^{t}\|\Phi _{xx}\|^{2}\mathrm{d}\tau+C \delta  \int _{0}^{t}\|(\varphi _{x},\varphi _{xx})\|^{2}\mathrm{d}\tau\nonumber
  \end{align}
  for any $  \eta >0 $, where the Cauchy-Schwarz has been used. Again, by Cauchy-Schwarz inequality, it holds that
  \begin{align}
  \displaystyle \mathcal{P}_{3}& \leq C \int _{\mathbb{R}_{+}}\left(\left\vert \psi _{0x}\right\vert+ \left\vert \varphi _{0x}\right\vert \right) \left(\left\vert \varphi _{0xxx}\right\vert+\left\vert \Phi _{0xx}\right \vert\right)\mathrm{d}x
   \nonumber \\
   &\displaystyle
  \leq C \left( \|\varphi _{0x}\|_{2}^{2}+\|(\psi _{0x},\Phi _{0xx})\|^{2} \right).\nonumber
  \end{align}
 From the boundary conditions in \eqref{boundary-simple-case}, we get
 \begin{align*}
 \displaystyle \left. \left( \varphi _{xx}-\frac{\mu \rho _{-}}{p'(\rho _{-})}\Phi _{x}  \right) \right \vert _{x=0}&= \left\{ \mu \bar{\rho}\left( \frac{1}{p'(\varphi _{x}+\bar{\rho})}- \frac{1}{p'(\bar{\rho})}+\frac{\varphi _{x}}{p'(\varphi _{x}+\bar{\rho})} \right)\Phi _{x}\right.
  \nonumber \\
   & \quad \left.\left. -\frac{1}{p'(\varphi _{x}+\bar{\rho})}\left( p'(\varphi _{x}+\bar{\rho})-p'(\bar{\rho}) \right)\bar{\rho}_{x}+ \frac{\varphi _{x}\bar{\phi}_{x}}{p'(\varphi _{x}+\bar{\rho})}  \right\}\right \vert _{x=0},
    \nonumber \\
       \displaystyle \left. \left( \varphi _{xx}-\frac{\mu \rho _{-}}{p'(\rho _{-})}\Phi _{x}  \right)_{t} \right \vert _{x=0}&= \left\{ \mu \bar{\rho}\left( \frac{1}{p'(\varphi _{x}+\bar{\rho})}- \frac{1}{p'(\bar{\rho})}+\frac{\varphi _{x}}{p'(\varphi _{x}+\bar{\rho})} \right)\Phi _{x}\right.
  \nonumber \\
   & \quad \left.\left. -\frac{1}{p'(\varphi _{x}+\bar{\rho})}\left( p'(\varphi _{x}+\bar{\rho})-p'(\bar{\rho}) \right)\bar{\rho}_{x}+ \frac{\varphi _{x}\bar{\phi}_{x}}{p'(\varphi _{x}+\bar{\rho})}  \right\}_{t}\right \vert _{x=0}.
     \end{align*}
  Thanks to \eqref{p-condition}, \eqref{con-stat-decay} and \eqref{samlness-sumup}--\eqref{rho-lower-bd}, we further have that
 \begin{align}\label{hig-bd-one}
 \begin{cases}
 	\displaystyle  \left\vert \left. \left( \varphi _{xx}-\frac{\mu \rho _{-}}{p'(\rho _{-})}\Phi _{x}  \right) \right \vert _{x=0}\right\vert &\leq C ( \varepsilon+ \delta )\left( \|\Phi _{x}\|_{L ^{\infty}}+\|\varphi _{x}\|_{L ^{\infty}} \right),
  \\[10pt]
   \displaystyle  \left\vert\left. \left( \varphi _{xx}-\frac{\mu \rho _{-}}{p'(\rho _{-})}\Phi _{x}  \right)_{t} \right \vert _{x=0}\right\vert &\leq C ( \varepsilon+ \delta )\left( \|\Phi _{xt}\|_{L ^{\infty}}+\|\varphi _{xt}\|_{L ^{\infty}} \right).
 \end{cases}
 \end{align}
By \eqref{hig-bd-one}, Cauchy-Schwarz inequality, the Sobolev inequality $ \left\|f\right\|_{L ^{\infty}} \leq C\|f\|_{1 } $ and the integration by parts, we have
 \begin{align}
 \mathcal{P}_{4}&=\displaystyle \int _{\mathbb{R}_{+}}(\varphi _{xt}+\varphi _{x})( \varphi _{xxx}- \frac{\mu \rho _{-}}{p'(\rho _{-})} \Phi _{xx} ) \mathrm{d}x
  \nonumber \\
   &\displaystyle =-\int _{\mathbb{R}_{+}}(\varphi _{xxt}+\varphi _{xx})( \varphi _{xx}- \frac{\mu \rho _{-}}{p'(\rho _{-})} \Phi _{x} ) \mathrm{d}x -\left. (\varphi _{xt}+\varphi _{x})\left( \varphi _{xx}- \frac{\mu \rho _{-}}{p'(\rho _{-})} \Phi _{x} \right)\right \vert _{x=0}
    \nonumber \\
    &\displaystyle \leq C \|\varphi _{xxt}\|\|(\varphi _{xx},\Phi _{x})\|+C \|(\varphi _{xx},\Phi _{x})\|^{2}+C(\varepsilon+\delta)\|\varphi _{xt}\|_{L ^{\infty}}\left(  \|\Phi _{x}\|_{L ^{\infty}}+\|\varphi _{x}\|_{L ^{\infty}}\right)
     \nonumber \\
     &\displaystyle \quad+C(\varepsilon+\delta)\|\varphi _{x}\|_{L ^{\infty}}\left(  \|\Phi _{x}\|_{L ^{\infty}}+\|\varphi _{x}\|_{L ^{\infty}}\right)
      \nonumber \\
      &\displaystyle \leq C \|(\varphi _{x},\varphi _{xt},\Phi _{x})\|_{1}^{2}\nonumber
 \end{align}
 and
\begin{align}
 \displaystyle \mathcal{P}_{5} &= \int _{0}^{t}\int _{\mathbb{R}_{+}} (\varphi _{xx \tau}+\alpha\varphi _{xx})\left( \varphi _{xx}- \frac{\mu \rho _{-}\Phi _{x}}{p'(\rho _{-})}  \right)_{\tau}  \mathrm{d}x \mathrm{d}\tau
  -\int _{0}^{t}\left. (\varphi _{x \tau}+\varphi _{x})\left( \varphi _{xx}- \frac{\mu \rho _{-}\Phi _{x}}{p'(\rho _{-})}  \right)_{\tau}\right \vert _{x=0} \mathrm{d}\tau
   \nonumber \\
   &\displaystyle \leq C \int _{0}^{t} \|(\varphi _{xx \tau},\varphi _{xx},\Phi _{x \tau})\|^{2}\mathrm{d}\tau +C(\varepsilon+\delta)\int _{0}^{t}(\|\varphi _{x \tau}\|_{L ^{\infty}}+\|\varphi _{x}\|_{L ^{\infty}})(\|\Phi _{x \tau}\|_{L ^{\infty}}+\|\varphi _{x \tau}\|_{L ^{\infty}})\mathrm{d}\tau
    \nonumber \\
       &\displaystyle \leq C(\varepsilon+\delta)\int _{0}^{t} \|\Phi _{xx \tau}\|^{2}\mathrm{d}\tau+C \int _{0}^{t}\left( \|(\varphi _{x}, \varphi _{x \tau})\|_{1}^{2}+\|\Phi _{x \tau}\|^{2} \right)  \mathrm{d}\tau\nonumber.
 \end{align}
For $ \mathcal{P}_{6} $, we utilize \eqref{F-x-bdd}, \eqref{cal-h-x-bd} and Cauchy-Schwarz inequality to get
\begin{align}
\displaystyle  \mathcal{P}_{6}
& \leq C \int _{0}^{t}\|(\mathcal{H}_{x},\mathcal{F}_{x})\|\|(\varphi _{xxx},\Phi _{xx})\|\mathrm{d}\tau
 \nonumber \\
 &\displaystyle\leq C(\varepsilon+\delta)\int _{0}^{t}\|(\varphi _{xxx},\varphi _{xx \tau})\|^{2} \mathrm{d}\tau+C\int _{0}^{t}\left( \|(\varphi _{\tau},\varphi _{x})\|_{1}^{2}+\|\Phi _{xx}\|^{2} \right)\mathrm{d}\tau.
  \nonumber
\end{align}
 Inserting the estimates for $ \mathcal{P} _{i}~(i=1,\cdots,6)$ into \eqref{var-xx-simple}, we get
 \begin{align}
 \displaystyle  \int _{0}^{t}\int _{\mathbb{R}_{+}}\varphi _{xxx}^{2} \mathrm{d}x \mathrm{d}\tau
   &\displaystyle \leq C \|(\varphi _{x},\varphi _{xt},\Phi _{x})\|_{1}^{2}+ C \left( \|\varphi _{0x}\|_{2}^{2}+\|\psi _{0x}\|^{2} \right)
 \nonumber \\
 & \displaystyle \quad +
 C(\delta+\varepsilon) \int _{0}^{t}\|\Phi _{xx \tau}\|^{2}\mathrm{d}\tau
  +C \int _{0}^{t}\left( \|\varphi _{\tau}\|_{2}^{2}+ \|(\Phi _{x \tau},\Phi _{xx})\|^{2}+\|\varphi _{x}\|_{1}^{2} \right)\mathrm{d}\tau,
   \nonumber
 \end{align}
 provided $ \varepsilon $ and $ \delta $ are suitably small. The proof of Lemma \ref{lem-h-multi} is complete.
 \end{proof}

 From Lemmas \ref{lem-vfi-xxx} and \ref{lem-h-multi}, we have the following higher-order estimates.
\begin{lemma}\label{lem-con-vfi-xxx}
Under the conditions of Proposition \ref{prop-key}, for any $ t \in (0,T) $, we have
\begin{align}\label{con-xxx-esti}
&\displaystyle  \|(\varphi _{xxx},\varphi _{xxt},\Phi _{xt},\Phi _{tt},\Phi _{xxx})\|^{2}+\int _{0}^{t}\left( \|(\varphi _{xx \tau},\varphi _{xxx})\|^{2}+\|(\varphi _{x \tau},\Phi _{\tau \tau})\|_{1}^{2} \right) \mathrm{d}\tau
 \nonumber \\
 &~\displaystyle \leq  C \left(\|\Phi _{0}\|_{4}^{2}+\|(\varphi _{0x},\psi _{0})\|_{2}^{2} \right) +C \|(\varphi _{x},\varphi _{t},\Phi _{x})\|_{1}^{2}+C  \int _{0}^{t}\left( \|(\varphi _{\tau},\varphi _{x})\|_{1}^{2}+ \|(\Phi _{x \tau},\Phi _{xx})\|^{2} \right) \mathrm{d}\tau,
\end{align}
provided $ \varepsilon $ and $ \delta $ are sufficiently small.
\end{lemma}
\begin{proof}
First, to control the terms related to $ (\Phi _{xt},\Phi _{xxt}, \Phi _{xtt}) $ on the right-hand side of \eqref{con-vfi-xxx}, we add \eqref{con-vfi-xxx} with \eqref{con-h-multi-inte} multiplied by a large positive constant to get
\begin{align}\label{vfi-xxx-cancel1}
&\displaystyle  \|(\varphi _{xxx},\varphi _{xxt},\Phi _{xt},\Phi _{tt},\Phi _{xxx})\|^{2}+\int _{0}^{t}\|(\Phi _{x \tau},\Phi _{\tau}, \Phi _{x \tau \tau},\Phi _{xx \tau},\varphi _{xx \tau})\|^{2}\mathrm{d}\tau
 \nonumber \\
 &~\displaystyle \leq  C(\|\Phi _{0}\|_{4}^{2}+\|(\varphi _{0x},\psi _{0})\|_{2}^{2})+C \left( \|(\varphi _{x},\varphi _{t})\|_{1}^{2}+\|\Phi _{xx}\|^{2} \right)
  \nonumber \\
  &~\displaystyle \quad+C (\delta+ \varepsilon)\int _{0}^{t}\|\varphi _{xxx}\|^{2}\mathrm{d}\tau+C  \int _{0}^{t}(\|(\varphi _{\tau},\varphi _{x})\|_{1}^{2}+ \|(\Phi _{x \tau},\Phi _{xx})\|^{2})\mathrm{d}\tau,
\end{align}
provided $ \varepsilon $ and $ \delta $ are suitably small, where the Cauchy-Schwarz inequality has been used. Combining \eqref{vfi-xxx-cancel1} with \eqref{con-fida-xtt}, for  sufficiently small $ \varepsilon $ and $ \delta $, we have
\begin{align*}
&\displaystyle  \|(\varphi _{xxx},\varphi _{xxt},\Phi _{xt},\Phi _{tt},\Phi _{xxx})\|^{2}+\int _{0}^{t}\|(\Phi _{x \tau},\Phi _{\tau \tau}, \Phi _{x \tau \tau},\Phi _{xx \tau},\varphi _{xx \tau},\varphi _{xxx})\|^{2}\mathrm{d}\tau
 \nonumber \\
 &~\displaystyle \leq  C(\|\Phi _{0}\|_{4}^{2}+\|(\varphi _{0x},\psi _{0})\|_{2}^{2})+C \|(\varphi _{x},\varphi _{t},\Phi _{x})\|_{1}^{2}+C  \int _{0}^{t}(\|(\varphi _{\tau},\varphi _{x})\|_{1}^{2}+ \|(\Phi _{x \tau},\Phi _{xx})\|^{2})\mathrm{d}\tau.
\end{align*}
This gives rise to \eqref{con-xxx-esti} and thus finishes the proof of Lemma \ref{lem-con-vfi-xxx}.
\end{proof}
\vspace*{3mm}

\subsection{Proof of Proposition \ref{prop-key}} 
\label{sub:proof_of_proposition_ref}
By the local existence result in Proposition \ref{prop-local} and the standard extension criterion, it suffices to show the estimates \eqref{regularity-prop} and \eqref{prop-esti} to prove Proposition \ref{prop-key}. We first close the \emph{a priori} assumption \eqref{aproiri-assumpt}. To this end, we add \eqref{con-xxx-esti} to \eqref{con-fip-xx} multiplied by a suitable positive constant and get
\begin{align}
& \displaystyle \|(\varphi _{x},\Phi _{x})\|_{2}^{2}+\|\varphi _{xt}\|_{1}^{2}+\|(\Phi _{xt},\Phi _{tt})\|^{2}
 + \int _{0}^{t}\|(\varphi _{xx},\varphi _{x \tau},\Phi _{x},\Phi _{x \tau},\Phi _{\tau \tau})\|_{1}^{2}\mathrm{d}\tau
  \nonumber \\
 &~ \displaystyle \leq C \left( \|\Phi _{0}\|_{4}^{2}+\|(\varphi _{0x},\psi _{0})\|_{2}^{2} \right)
 +C  \|\varphi _{t}\|^{2}
 +
 C  \int _{0}^{t}\|(\varphi _{x},\varphi _{\tau})\|^{2} \mathrm{d}\tau+ C\int _{0}^{t}\|\varphi _{x}\|\|\Phi _{xx \tau} \|  \mathrm{d}\tau,\nonumber
\end{align}
where the smallness of $ \varepsilon $ and $ \delta $ has been used.
This along with the Cauchy-Schwarz inequality gives
\begin{align}\label{varifie-step-2}
& \displaystyle \|(\varphi _{x},\Phi _{x})\|_{2}^{2}+\|\varphi _{xt}\|_{1}^{2}+\|(\Phi _{xt},\Phi _{tt})\|^{2}
 + \int _{0}^{t}\|(\varphi _{xx},\varphi _{x \tau},\Phi _{x},\Phi _{x \tau},\Phi _{\tau \tau})\|_{1}^{2}\mathrm{d}\tau
  \nonumber \\
 &~ \displaystyle \leq C \left( \|\Phi _{0}\|_{4}^{2}+\|(\varphi _{0x},\psi _{0})\|_{2}^{2} \right)
 +C  \|\varphi _{t}\|^{2}
 +
 C  \int _{0}^{t}\|(\varphi _{x},\varphi _{\tau})\|^{2} \mathrm{d}\tau.
\end{align}
Furthermore, multiplying \eqref{con-lem-basic} by a suitable positive constant, and adding the resulting inequality to \eqref{varifie-step-2}, we obtain
\begin{align}\label{final-verifi}
&\displaystyle \|(\varphi,\Phi)\|_{3}^{2}+\|\varphi _{t}\|_{2}^{2}+\|(\Phi _{xt},\Phi _{tt})\|^{2}+
\int _{0}^{t} \Big(\|\varphi _{x}\|_{2}^{2}+\|(\varphi _{\tau},\Phi _{\tau},\Phi)\|_{2}^{2}+\|\Phi _{\tau \tau}\|_{1}^{2} \Big)\mathrm{d}\tau
 \nonumber \\
 &~\leq
C \left( \|\varphi _{0}\|_{3}^{2}+\|\psi _{0}\|_{2}^{2}+\|\Phi _{0}\|_{4}^{2} \right).
\end{align}
Then by setting
\begin{gather*}
\displaystyle \varepsilon ^{2}= 2  C \left( \|\varphi _{0}\|_{3}^{2}+\|\psi _{0}\|_{2}^{2}+\|\Phi _{0}\|_{4}^{2} \right)
\end{gather*}
and taking $ \|\varphi _{0}\|_{3}+\|\psi _{0}\|_{2}+\|\Phi _{0}\|_{4} $ suitably small, we have
\begin{gather*}
\displaystyle \sup _{0 \leq t<T}\left\{\left\|(\varphi,\Phi)(\cdot, t)\right\|_{3}^{2}+\left\| \psi(\cdot, t)\right\|_{2}^{2}\right\} <\varepsilon ^{2}
\end{gather*}
which hence closes the \emph{a priori} assumption \eqref{aproiri-assumpt}.
To complete the proof of \eqref{regularity-prop}-\eqref{prop-esti}, now it remains to show the following
\begin{gather*}
\displaystyle \|\Phi _{xxxx}\|^{2}+\int _{0}^{t}\left( \|\Phi _{xxx}\|^{2}+\|(\varphi _{\tau \tau},\varphi _{\tau \tau x})\|^{2} \right)  \mathrm{d} \tau \leq C\left( \|\varphi _{0}\|_{3}^{2}+\|\psi _{0}\|_{2}^{2}+\|\Phi _{0}\|_{4}^{2} \right).
 \end{gather*}
Collecting \eqref{Eq-x-simple-b}, \eqref{varphi-tt}, \eqref{vfi-xtt-l2} and \eqref{final-verifi}, one immediately has
\begin{gather*}
\displaystyle \int _{0}^{t}\left( \|\Phi _{xxx}\|^{2}+\|(\varphi _{\tau \tau},\varphi _{\tau \tau x})\|^{2} \right)  \mathrm{d} \tau \leq C\left( \|\varphi _{0}\|_{3}^{2}+\|\psi _{0}\|_{2}^{2}+\|\Phi _{0}\|_{4}^{2} \right).
\end{gather*}
To derive the estimate for $ \Phi _{xxxx} $, we first deduce from \eqref{problem-linearized-simple-c} and \eqref{final-verifi} that
\begin{gather}\label{fida-t-single}
 \displaystyle \|\Phi _{t}\| ^{2} \leq C \|(\Phi _{xx},\varphi _{x},\Phi )\|^{2}\leq C\left( \|\varphi _{0}\|_{3}^{2}+\|\psi _{0}\|_{2}^{2}+\|\Phi _{0}\|_{4}^{2} \right).
 \end{gather}
Next, differentiating \eqref{problem-linearized-simple-c} with respect to $ t $ leads to
\begin{gather*}
\displaystyle \Phi _{tt}=\Phi _{xxt}+a \varphi _{xt}+b \Phi _{t}.
\end{gather*}
This along with \eqref{fida-t-single} and \eqref{final-verifi} yields that
\begin{gather*}
\displaystyle  \|\Phi _{xxt}\|^{2} \leq C\|(\Phi _{tt},\Phi _{t},\varphi _{xt})\|^{2} \leq C\left( \|\varphi _{0}\|_{3}^{2}+\|\psi _{0}\|_{2}^{2}+\|\Phi _{0}\|_{4}^{2} \right).
\end{gather*}
Finally, differentiating \eqref{Eq-x-simple-b} with respect to $ x $, we have
\begin{gather*}
\displaystyle \Phi _{xxxx}=\Phi _{xxt}-a \varphi _{xxx}+b \Phi _{xx},
\end{gather*}
and thus
\begin{align*}
\displaystyle \|\Phi _{xxxx}\|^{2}\leq C \|(\Phi _{xx},\Phi _{xxt},\varphi _{xxx})\|^{2}\leq  C\left( \|\varphi _{0}\|_{3}^{2}+\|\psi _{0}\|_{2}^{2}+\|\Phi _{0}\|_{4}^{2} \right).
\end{align*}
The proof of Proposition \ref{prop-key} is complete.\hfill $ \square $

\vspace*{3mm}
 \subsection{Proof of Theorem \ref{thm-stability}} 
 \label{sub:proof_of_theorem_ref}
 In view of Proposition \ref{prop-key}, the  problem \eqref{ori-eq-a}--\eqref{ori-eq-c}, \eqref{initial-ori}--\eqref{far-field} admits a unique classical solution $ (\rho,m,\phi) $ in $ \mathbb{R}_{+}\times (0, \infty) $. Moreover, thanks to \eqref{pertu-variable}, \eqref{regularity-prop} and \eqref{prop-esti},  it holds that
 \begin{align}
 &\displaystyle   \|(\rho- \bar{\rho},m)\|_{2}^{2}+\|\phi- \bar{\phi}\|_{4}^{2}
  \leq  C\left( \|\varphi _{0}\|_{3}^{2}+\|\psi _{0}\|_{2}^{2}+\|\Phi _{0}\|_{4}^{2} \right),
   \nonumber
 \end{align}
 and that
 \begin{gather}\label{esti-for-orig-solu}
 \displaystyle  \int _{0}^{t}\left( \|(\rho- \bar{\rho},m)\|_{2}^{2}+\|\phi- \bar{\phi}\|_{3}^{2} +\|(\rho _{\tau}, m _{\tau}, \phi _{\tau})\|_{1}^{2} \right)\mathrm{d}\tau\leq  C\left( \|\varphi _{0}\|_{3}^{2}+\|\psi _{0}\|_{2}^{2}+\|\Phi _{0}\|_{4}^{2} \right)
 \end{gather}
 for any $ t >0 $. In the following, we shall prove the large time behavior of $ (\rho,m,\phi) $ as in \eqref{large-time-thm}. For this, recalling the Sobolev inequality $ \|f\|_{L ^{\infty}} \leq C\|f\|^{\frac{1}{2}}\|f _{x}\|^{\frac{1}{2}} $, it suffices to show that
 \begin{gather}\label{L-2-decay-final}
 \displaystyle \lim _{t \rightarrow \infty} \|( \rho -\bar{\rho},m ,\phi- \bar{\phi}) (\cdot,t)\|^{2}\rightarrow 0.
 \end{gather}
 In fact, with the help of \eqref{esti-for-orig-solu} and Cauchy-Schwarz inequality, we get
 \begin{align}\label{soln-xt}
 &\displaystyle  \int _{0}^{+\infty}\left\vert\frac{\mathrm{d}}{\mathrm{d}t}\|( \rho -\bar{\rho},m ,\phi- \bar{\phi} )(\cdot,t) \|^{2} \right\vert \mathrm{d}t
  \nonumber \\
  &~\displaystyle\leq C
 \int _{0}^{+\infty}\left( \|( \rho -\bar{\rho},m ,\phi- \bar{\phi} ) \|^{2}+\| (\rho _{t},m _{t},\phi_{t} ) \|^{2} \right) \mathrm{d}t \leq \infty.
 \end{align}
The estimate \eqref{esti-for-orig-solu} in combination with \eqref{soln-xt} gives \eqref{L-2-decay-final}. Then \eqref{large-time-thm} is proved and we complete the proof of Theorem \ref{thm-stability}.\hfill $ \square $

 \section*{Acknowledgement} 
G. Hong is partially supported from the CAS AMSS-POLYU Joint Laboratory of Applied Mathematics postdoctoral fellowship scheme.  H.Y. Peng was supported from the National Natural Science Foundation of China (No. 11901115), Natural Science Foundation of Guangdong Province (No. 2019A1515010706). Z.-A. Wang was supported in part by the Hong Kong RGC GRF grant No. PolyU 153055/18P (P0005472) and an internal grant No. ZZKN from HKPU (P0031013). C.J. Zhu was supported by the National Natural Science Foundation of China (No. 11771150, 11831003 and 11926346) and Guangdong Basic and Applied Basic Research Foundation (No. 2020B1515310015).




\bibliographystyle{mysiam}
\bibliography{hwhp}

\end{document}